\documentclass[10pt]{amsart}
\usepackage{tgpagella}
\usepackage{ifpdf}

\usepackage[english]{babel}
\usepackage{inputenc}
\usepackage{mathabx}
\usepackage{amsmath,amsthm,amsfonts,amssymb,amscd}
\usepackage{palatino}
\usepackage{mathpazo}
\usepackage{amssymb,amsfonts}
\usepackage[all,arc]{xy}
\usepackage{enumerate}
\usepackage{mathrsfs}
\usepackage{tgtermes}
\usepackage{marginnote}
\usepackage{mathtools}
\usepackage{amsthm,thmtools,xcolor}

\theoremstyle{plain}
\newtheorem{thm}{Theorem}[section]
\newtheorem{cor}[thm]{Corollary}
\newtheorem{prop}[thm]{Proposition}
\newtheorem{lem}[thm]{Lemma}

\newtheorem{thmx}{Theorem}
 
\theoremstyle{definition}

\theoremstyle{remark}
\newtheorem{rem}[thm]{\textit{Remark}}

\usepackage{color}

\usepackage[colorlinks = true,
linkcolor = blue,
urlcolor  = blue,
citecolor = blue,
anchorcolor = blue]{hyperref}
\usepackage{hyperref}
\makeatletter
\let\c@equation\c@thm
\makeatother
\numberwithin{equation}{section}
\DeclareMathOperator{\ex}{e}
\DeclareMathOperator{\dist}{\mathrm{dist}}
\DeclareMathOperator{\sgn}{\mathrm{sgn}}
\DeclareMathOperator{\supp}{supp}
\newcommand\underrel[3][]{\mathrel{\mathop{#3}\limits_{%
			\ifx c#1\relax\mathclap{#2}\else#2\fi}}}
\begin{document}
\title[Dispersion Generalized Benjamin-Ono Equation]{On the propagation of regularity for solutions of   the  Dispersion Generalized Benjamin-Ono Equation}

\author{Argenis. J. Mendez}

\address{Instituto Nacional  de  Matematica Pura e  Aplicada, Rio de Janeiro, RJ, Brasil}

\email{amendez@impa.br}

\thanks{This work was partially supported by  CNPq, Brazil.}



\date{December, 2018.}
\dedicatory{To my parents}
\commby{Argenis Mendez}
\begin{abstract}
   In this paper we study  some  properties of propagation of regularity  of solutions of the dispersive generalized Benjamin-Ono (BO) equation. This model defines a   family of dispersive equations,  that can be seen as  a dispersive interpolation between Benjamin-Ono equation and Korteweg-de Vries (KdV) equation. 

  Recently, it has been shown that  solutions of the  KdV  equation and Benjamin-Ono equation, satisfy the following property:  if the   initial data  has    some  prescribed regularity on the right hand side of the  real line, then this  regularity is propagated with infinite speed by the flow solution.
    
 In this case the nonlocal term present in the dispersive generalized Benjamin-Ono equation  is more challenging that the one  in BO equation. To deal with this a new approach is needed. The new ingredient is to combine  commutator expansions into the weighted energy estimate. This allow us to obtain the property of propagation  and explicitly the smoothing effect.
\end{abstract}
\maketitle

\section{Introduction}
The aim of this work is to study some  special regularity properties of
solutions to the initial value problem (IVP) associated to   the \emph{dispersive generalized Benjamin-Ono equation}
\begin{equation}\label{eq7}
\left\{
\begin{array}{ll}
\partial_{t}u-D_{x}^{\alpha+1}\partial_{x}u+u\partial_{x}u=0, & x,t\in\mathbb{R},\,0<\alpha<1, \\
u(x,0)=u_{0}(x),&  \\
\end{array} 
\right.
\end{equation}
where $D_{x}^{s},$ denotes the homogeneous derivative of order $s\in \mathbb{R},$
\begin{equation*}
D_{x}^{s}=(-\partial_{x}^{2})^{s/2}\quad\mbox{thus}\quad  D_{x}^{s}f=c_{s}\left(|\xi|^{s}\widehat{f}(\xi)\right)^{\widecheck{}},\quad 
\end{equation*}
which in its polar form is  decomposed as $ D_{x}^{s}=(\mathcal{H}\partial_{x})^{s},$  where $\mathcal{H}$ denotes the \emph{Hilbert transform},
\begin{equation*}
\mathcal{H}f(x)
=\frac{1}{\pi}\lim_{\epsilon\rightarrow 0^{+}}\int_{|y|\geq \epsilon}\frac{f(x-y)}{y}\,\mathrm{d}y=(-i\sgn(\xi)\widehat{f}(\xi))^{\widecheck{\hspace{2mm}}}(x),
\end{equation*}
where    $\widehat{\cdot}$ denotes the Fourier transform and $\widecheck{\hspace{2mm}}$ denotes its inverse.  These equations model vorticity waves in the coastal zone, see \cite{MST} and references therein.
 
Our starting point  is a property established by Isaza, Linares and Ponce \cite{ILP1} concerning the solutions  of the IVP associated to the $k-$generalized KdV equation
\begin{equation}\label{eqdgb1}
\left\{
\begin{array}{ll}
\partial_{t}u+\partial_{x}^{3}u+u^{k}\partial_{x}u=0, & x,t\in\mathbb{R},\,k\in\mathbb{N}, \\
u(x,0)=u_{0}(x).&  
\end{array} 
\right.
\end{equation}
 It was shown   in \cite{ILP1} that  the unidirectional dispersion of the  $k-$generalized KdV equation  entails  the following propagation of regularity phenomena.
\begin{thm}[\cite{ILP1}]\label{m7}
	If $u_{0}\in H^{3/4^{+}}(\mathbb{R})$ and for some $l\in \mathbb{Z},\, l\geq 1$ and  $x_{0}\in \mathbb{R}$
	\begin{equation}\label{m9}
	\left\|\partial_{x}^{l}u_{0}\right\|_{L^{2}((x_{0},\infty))}^{2}=\int_{x_{0}}^{\infty}\left|\partial_{x}^{l}u_{0}(x)\right|^{2}\,\mathrm{d}x<\infty,
	\end{equation}
	then the solution of the IVP associated to \eqref{eqdgb1} 
	satisfies  that for any $v>0$ and $\epsilon>0$
	\begin{equation}\label{m10}
	\sup_{0\leq t\leq  T}\int_{x_{0}+\epsilon-vt}^{\infty} \left(\partial_{x}^{j}u\right)^{2}(x,t)\,\mathrm{d}x<c,
	\end{equation}
	for $j=0,1,2,\dots,l$ with $c=c\left(l;\|u_{0}\|_{H^{3/4^{+}}(\mathbb{R})};  \|\partial_{x}^{l}u_{0}\|_{L^{2}((x_{0},\infty))};v;\epsilon;T\right).$
	In particular, for  all $t\in(0,T],$ the restriction  of $u(\cdot,t)$ to any interval  $(x_{0},\infty)$ belongs to  $H^{l}((x_{0},\infty)).$
	
	Moreover, for any $v\geq 0,\, \epsilon>0$ and $R>0$
	\begin{equation*}
	\int_{0}^{T}\int_{x_{0}+\epsilon-vt}^{x_{0}+R-vt} \left(\partial_{x}^{l+1}u\right)^{2}(x,t)\,\mathrm{d}x\,\mathrm{d}t<c,
	\end{equation*}
	with $c=c\left(l;\|u_{0}\|_{H^{3/4^{+}}(\mathbb{R})};  \|\partial_{x}^{l}u_{0}\|_{L^{2}((x_{0},\infty))};v;\epsilon;R;T\right).$
\end{thm}
The proof of Theorem \ref{m7}  is based  on  weighted energy estimates. In detail, the  iterative process in the induction argument is based in a property discovered originally  by T. Kato \cite{KATO1} in the context of the KdV equation. More precisely, he showed that  solution of the KdV equation  satisfies 
\begin{equation}\label{smooth}
\int_{0}^{T}\int_{-R}^{R}\left(\partial_{x}u\right)^{2}(x,t)\,\mathrm{d}x\,\mathrm{d}t\leq c\left(R;T;\|u_{0}\|_{L^{2}_{x}}\right),
\end{equation} 
 being this  the fundamental fact in his  proof of existence  of the global  weak solutions  of \eqref{eqdgb1}, for $k=1$ and  initial data in $L^{2}(\mathbb{R}).$
 
This result  was also obtained  for the Benjamin-Ono equation \cite{ILP2} but it does not follow as the KdV case because  of the presence of the Hilbert transform.

Later on,  Kenig et al. \cite{KLPV} extended the results in Theorem \ref{m7}  to the case  when the local regularity  of the initial data $u_{0}$ in (\ref{m9}) is measured  with a fractional indices. The scope  to this case  is   quite more involved, and its  proof is  mainly based in  weighted energy estimates combined  with techniques  involving  pseudo-differential operators  and singular integrals. The property described in Theorem \ref{m7} is intrinsic to suitable solutions of some nonlinear
dispersive models (see also \cite{LPS}). In the context of 2D models, analogous results for the Kadomtsev-Petviashvili II equation \cite{ILP4}
and Zakharov-Kuznetsov \cite{LPZK} equations were proved.

Before state our main result  we will give  an overview of the local well-posedness  of the  IVP \eqref{eq7}.

Following  \cite{KATO1}  we have    that   the initial value problem  IVP (\ref{eq7}) is  \emph{locally well-posed} (LWP) in the Banach space $X$ if  for every initial condition $u_{0}\in X,$ there exists $T>0$  and a unique solution  $u(t)$ satisfying 
\begin{equation}\label{eq39}
u\in C\left([0,T]: X\right)\cap A_{T}
\end{equation}
where $A_{T}$ is an auxiliary  function space. Moreover, the solution  map $u_{0}\longmapsto u,$ is  continuous from $X$ into the class (\ref{eq39}). If $T$ can be taken arbitrarily large, one says that the IVP (\ref{eq7}) is \emph{globally well-posed} (GWP)  in the space $X.$  

It is natural to study the IVP \eqref{eq7} in the Sobolev space 
\begin{equation*}
H^{s}(\mathbb{R})=(1-\partial_{x}^{2})^{-s/2}L^{2}(\mathbb{R}),\qquad s\in \mathbb{R}.
\end{equation*}
  There exist  remarkable differences between the KdV \eqref{eqdgb1} and the IVP \eqref{eq7}. In case of KdV e.g. it posses  infinite conserved quantities, define a Hamiltonian system, have  multi-soliton solutions and is a completely integrable system by the inverse scattering  method \cite{CW}, \cite{FA}. Instead,  in the case of the IVP \eqref{eq7}  there is no integrability, but  three  conserved quantities (see \cite{SS}), specifically
\begin{equation*}
\begin{split}
& \mathrm{I}[u](t)=\int_{\mathbb{R}}u\,\mathrm{d}x,\qquad \mathrm{M}[u](t)=\int_{\mathbb{R}}u^{2}\,\mathrm{d}x,\\
&\mathrm{H}[u](t)=\frac{1}{2}\int_{\mathbb{R}}\left|D_{x}^{\frac{1+\alpha}{2}}u\right|^{2} \mathrm{d}x-\frac{1}{6}\int_{\mathbb{R}}u^{3}\,\mathrm{d}x,
\end{split}
\end{equation*}
are satisfied at least for smooth solutions.

Another property in which these two models  differ,  resides in the fact that  one can obtain a local existence  theory for the KdV equation in $H^{s}(\mathbb{R}),$  based  on the contraction principle. On the contrary,  this cannot be done in the case of the (IVP) \eqref{eq7}. This is a consequence of the  fact that  dispersion is not enough to deal with the nonlinear term. In this direction, Molinet, Saut and  Tzvetkov \cite{MST}  showed that for $0\leq\alpha<1$ the IVP (\ref{eq7}) with  the assumption  $u_{0}\in H^{s}(\mathbb{R})$ is not enough to prove  local well-posedness by using fixed point arguments or Picard iteration method.

Nevertheless, Molinet and Ribaud \cite{MR} proved  global well-posedness by considering initial data in a weighted low frequencies Sobolev space. Later,  using suitable spaces of Bourgain type, Herr \cite{Herr} proves local well-posedness for initial data in $H^{s}(\mathbb{R})\cap \dot{H}^{-\omega}(\mathbb{R})$ for  any $s>-\frac{3\alpha}{4},\, \omega=\frac{1}{\alpha+1}-\frac{1}{2},$ where $ \dot{H}^{-\omega}(\mathbb{R})$ is a a weighted low frequencies Sobolev space (for more details see \cite{Herr}), next by using  a conservation law, these results are extended to global well-posedness in $H^{s}(\mathbb{R})\cap \dot{H}^{-\omega}(\mathbb{R}),$ for $s\geq0,\,\omega =\frac{1}{\alpha+1}-\frac{1}{2}.$ In this sense, an improvement was obtained by  Herr, Ionescu, Kenig and Koch \cite{HIKK}, who 
show that the IVP \eqref{eq7} is globally  well-posed  in the space of the real-valued  $L^{2}(\mathbb{R})-$functions, by using   a renormalization method to control the strong low-high frequency interactions. However, it is not clear  that these results described  above   can be used to  establish our main result. So that,  a local theory  obtained by using energy estimates in addition to  dispersive properties of the smooth  solutions is required.

In the first step,  we obtain the following a priori estimate for  solutions of IVP \eqref{eq7}
\begin{equation*}
\|u\|_{L^{\infty}_{T}H_{x}^{s}}\lesssim \|u_{0}\|_{H_{x}^{s}}\ex^{c\|\partial_{x}u\|_{L_{T}^{1}L^{\infty}_{x}}},
\end{equation*}
part of this estimate is based on the Kato-Ponce   commutator estimate \cite{KATOP2}.

The inequality above reads as follows:  in order to    the solution $u$  abide  in the Sobolev space $H^{s}(\mathbb{R}),$ continuously in time,  we   require  to control the term  $\footnotesize{\|\partial_{x}u\|_{L_{T}^{1}L_{x}^{\infty}}.}$

 First,  we use  Kenig, Ponce and Vega in \cite{KPV4} results concerning oscillatory integrals, in order  to  obtain the classical Strichartz estimates  associated to the group $\footnotesize{S(t)=\ex^{tD_{x}^{\alpha+1}\partial_{x}},}$  corresponding  to the  linear part of the equation in (\ref{eq7}).  

 In second place, the technique introduced by  Koch and Tzvetkov \cite{Koch-Tzvetkov} related to   refined  Strichartz estimate are fundamentals in our analysis. Specifically, their   method is mainly  based in a decomposition  of  the time interval  in small pieces  whose length depends on the spatial frequencies of the solution. This approach   allowed to  Koch and Tzvetkov  to prove local well-posedness, for the Benjamin-Ono equation  in $H^{5/4^{+}}(\mathbb{R}).$ Succeeding, Kenig and Koenig \cite{KK}  enhanced  this estimate,  which led to prove local well-posedness for the Benjamin-Ono equation   in $H^{9/8^{+}}(\mathbb{R}).$  

Several issues arise when handling the nonlinear part of the equation in \eqref{eq7}, nevertheless,  following the work of Kenig, Ponce and Vega \cite{KPV1}, we manage the loss of derivatives  by  means of combination of the local smoothing effect and  a maximal  function estimate of  the group $\footnotesize{S(t)=\ex^{tD_{x}^{\alpha+1}\partial_{x}}}.$

These observations  lead us to present our first result.
\begin{thmx}\label{B.1}
	Let $0<\alpha<1.$  Set $s(\alpha)=\frac{9}{8}-\frac{3\alpha}{8}$ and assume that $s>s(\alpha).$  Then, for any $u_{0}\in H^{s}(\mathbb{R}),$ there exists a positive time  $T=T\left(\|u_{0}\|_{H^{s}(\mathbb{R})}\right)>0$ and a unique solution  $u$ satisfying  
	\eqref{eq7} such that 
	\begin{equation}\label{eq1.16}
	u\in C\left([0,T]: H^{s}(\mathbb{R})\right)\quad \mbox{and}\quad  \partial_{x}u\in L^{1}\left([0,T]: L^{\infty}(\mathbb{R})\right).	
	\end{equation}
	Moreover, for any $r>0,$ the map  $u_{0}\mapsto u(t)$ is continuous from the ball\\ $\left\{u_{0}\in H^{s}(\mathbb{R}):\, \|u_{0}\|_{H^{s}(\mathbb{R})}<r \right\}$ \quad to  $C\left([0,T]: H^{s}(\mathbb{R})\right).$
\end{thmx}
 Theorem \ref{B.1} is the base result  to  describe the propagation of regularity phenomena.  
 As we mentioned above   the propagation of regularity phenomena is satisfied by the  BO and KdV equations respectively.  These two models correspond to particular cases of the IVP \eqref{eq7}, specifically  by taking  $\alpha=0$ and $\alpha=1.$ 
 
 A  question  that arises naturally  is to determine whether the propagation of regularity phenomena is  satisfied  for a model with   an intermediate dispersion  between these two  models mentioned above. 
 
 Our main result  give answer to this problem and it  is  summarized in the following:
\begin{thmx}\label{A}
	Let $u_{0}\in H^{s}(\mathbb{R})$ with $s=\frac{3-\alpha}{2},$\, 
	and $u=u(x,t),$  be the corresponding solution  of the IVP (\ref{eq7}) provided by Theorem \ref{B.1}. 
	
	If  for some $x_{0}\in \mathbb{R}$  and for some $m\in\mathbb{Z}^{+},\,m\geq  2,$
	\begin{equation}\label{eq106}
	\partial_{x}^{m}u_{0}\in L^{2}\left(\left\{x\geq x_{0}\right\}\right),
	\end{equation}
	then for any $v\geq0,T>0,\epsilon>0$ and    \,$\tau>\epsilon$
	\begin{equation}\label{eqc1}
	\begin{split}
	&\sup_{0\leq t\leq T}\int_{x_{0}+\epsilon-vt}^{\infty}(\partial_{x}^{j}u)^{2}(x,t)\mathrm{d}x
	+ \int_{0}^{T}\int_{x_{0}+\epsilon-vt}^{x_{0}+\tau-vt}\left(D_{x}^{\frac{\alpha+1}{2}}\partial_{x}^{j}u\right)^{2}(x,t)\mathrm{d}x\,\mathrm{d}t\\
	&		+\int_{0}^{T}\int_{x_{0}+\epsilon-vt}^{x_{0}+\tau-vt}\left(D_{x}^{\frac{\alpha+1}{2}}\mathcal{H}\partial_{x}^{j}u\right)^{2}(x,t)\mathrm{d}x\,\mathrm{d}t\leq c 
	\end{split}
	\end{equation}	
	for $j=1,2,\dots,m$ with $c=c\left(T;\epsilon;v;\alpha;\|u_{0}\|_{H^{s}};\left\|\partial_{x}^{m}u_{0}\right\|_{L^{2}((x_{0},\infty))}\right)>0.$
	
	If in addition to (\ref{eq106}) there exists $x_{0}\in \mathbb{R}^{+}$ 
	\begin{equation}\label{clave1}
	D_{x}^{\frac{1-\alpha}{2}}\partial_{x}^{m}u_{0}\in L^{2}\left(\left\{x\geq x_{0}\right\}\right)
	\end{equation}
	then for any $v\geq 0,\,\epsilon>0$ and $\tau>\epsilon$
	\begin{equation}\label{l1}
	\begin{split}
	&\sup_{0\leq t\leq  T}\int_{x_{0}+\epsilon-vt}^{\infty}\left(D_{x}^{\frac{1-\alpha}{2}}\partial_{x}^{m}u\right)^{2}(x,t)\,\mathrm{d}x
	+\int_{0}^{T}\int_{x_{0}+\epsilon-vt}^{x_{0}+\tau-vt}\left(\partial_{x}^{m+1}u\right)^{2}(x,t)\,\mathrm{d}x\,\mathrm{d}t\\
	&+\int_{0}^{T}\int_{x_{0}+\epsilon-vt}^{x_{0}+\tau-vt}\left(\partial_{x}^{m+1}\mathcal{H}u\right)^{2}(x,t)\,\mathrm{d}x\,\mathrm{d}t\leq c
	\end{split}
	\end{equation}
	with $c=c\left(T;\epsilon;v;\alpha;\|u_{0}\|_{H^{s}};\left\|D_{x}^{\frac{1-\alpha}{2}}\partial_{x}^{m}u_{0}\right\|_{L^{2}((x_{0},\infty))}\right)>0.$
\end{thmx}
Although  the  argument of the proof of Theorem \ref{A} follows in spirit that of  KdV i.e. an induction process combined with  weighted energy estimates. The   presence of  the  non-local operator $\footnotesize{D_{x}^{\alpha+1}\partial_{x}}$, in the term  providing  the dispersion,  makes the proof  much harder. More precisely, two  difficulties appear, in the first place  and the most important is to  obtain explicitly  the Kato smoothing effect as in \cite{KATO1}, that as  in the proof of Theorem \ref{m7} is fundamental.

 In contrast to KdV equation,  the gain of the local smoothing  in    solutions  of the    dispersive generalized  Benjamin-Ono  equation is just $\footnotesize{\frac{\alpha+1}{2}}$ derivatives, so as occurs in the case of the Benjamin-Ono equation \cite{ILP2},  the iterative  argument in the induction process  is   carried out in two steps, one for  positive integers $m$ and  another one for $m+\frac{1-\alpha}{2}$ derivative. 

In the case of the BO equation \cite{ILP2}, the authors obtain the smoothing effect  basing  their analysis  on  several commutator estimates,  such as  the extension of  the first  Calderon's commutator for the Hilbert transform \cite{BC}.  However, their method of proof do not allow them to  obtain  explicitly the local smoothing as in \cite{KATO1}. 

   The advantage of our method is that it allows obtain explicitly the smoothing effect  for any $\alpha\in(0,1)$ in the IVP \eqref{eq7}. Roughly,  we rewrite the term modeling the dispersive part of the equation in \eqref{eq7}, in terms of an expression involving $\footnotesize{\left[\mathcal{H}
   	D_{x}^{\alpha+2}; \chi_{\epsilon, b}^{2}\right]}$. At this point, we  incorporate   Ginibre and Velo  \cite{GV2}   results about commutator decomposition. This, allows us to   obtain explicitly   the smoothing effect as in \cite{KATO1}, at every step of the induction process in the energy estimate. Besides, this approach   allow us  to study  the propagation of regularity phenomena in models where the dispersion is lower in comparison  with  that  of  IVP \eqref{eq7}. We address this issue in a forthcoming work, specifically we study the propagation of regularity phenomena in real solutions of the model 
 \begin{equation*}
 \partial_{t}u-D_{x}^{\alpha}\partial_{x}u+u\partial_{x}u=0,\quad x,t\in \mathbb{R},\quad 0<\alpha<1.
 \end{equation*}
As a direct consequence  of the Theorem \ref{A} 
one has that for an appropriate  class of initial data, the singularity of the solution  travels with infinity speed to the left as time evolves. Also, the time reversibility  property implies that the solution  cannot have  had  some regularity  in the past.

Concerning the nonlinear part of IVP \eqref{eq7}  into the weighted energy estimate, several issues arises. Nevertheless,  following  Kenig et al. \cite{KLPV}  approach, combined with the works of   Kato-Ponce \cite{KATOP2}, and the recent work D. Li \cite{Dongli} on the generalization of several commutators estimate, allow us to overcome  these difficulties. 
\begin{rem}
	\begin{itemize}
	\item[(I)]  It will be clear from our proof  that the requirement on the initial data, that is   $u_{0}\in H^{\frac{3-\alpha}{2}}(\mathbb{R})$ in  Theorem \ref{A} can be lowered  to $H^{\frac{9-3\alpha}{8}+}(\mathbb{R}).$
	\item[(II)] Also it is worth highlighting that  the proof of Theorem \ref{A} can be extended to  solutions of the the IVP
		\begin{equation}\label{eq7.1}
		\left\{
		\begin{array}{ll}
		\partial_{t}u-D_{x}^{\alpha+1}\partial_{x}u+u^{k}\partial_{x}u=0, & x,t\in\mathbb{R},\,0<\alpha<1,\, k\in \mathbb{Z}^{+}, \\
		u(x,0)=u_{0}(x).&  \\
		\end{array} 
		\right.
		\end{equation}
	\item[(III)] The results in Theorem \ref{A}  still holds for solutions  of the defocussing  generalized   dispersive Benjamin-Ono equation 
	\begin{equation*}\label{eq7.2}
\left\{
\begin{array}{ll}
\partial_{t}u-D_{x}^{\alpha+1}\partial_{x}u-u\partial_{x}u=0, & x,t\in\mathbb{R},\,0<\alpha<1, \\
u(x,0)=u_{0}(x).&  \\
\end{array} 
\right.
\end{equation*}
This can be seen applying Theorem \ref{A} to the function $v(x,t)=u(-x,-t),$ where $u(x,t)$ is a solution  of \eqref{eq7}. In short, Theorem \ref{A} remains  valid, backward in time  for initial data $u_{0},$ satisfying \eqref{eq106} and \eqref{clave1}.
	\end{itemize}
		\end{rem}
Next, we present some immediate consequences of Theorem \ref{A}.
\begin{cor}
	Let $u\in C\left([-T,T]:H^{\frac{3-\alpha}{2}}(\mathbb{R})\right)$  be a solution  of the equation in (\ref{eq7}) described by Theorem \ref{A}. If there exist  $n,m\in \mathbb{Z}^{+}$  with $m\leq n$ such that for some $\tau_{1},\tau_{2}\in \mathbb{R}$ with $\tau_{1}<\tau_{2}$
	\begin{equation*}
	\int_{\tau_{2}}^{\infty} |\partial_{x}^{n}u_{0}(x)|^{2}\,\mathrm{d}x<\infty\quad \mbox{but}\quad \partial_{x}^{m}u_{0}\notin L^{2}((\tau_{1},\infty)),
	\end{equation*} 
	then for any $t\in (0,T)$ and any $v>0$ and $\epsilon>0$
	\begin{equation*}
	\int_{\tau_{2}+\epsilon-vt}^{\infty}|\partial_{x}^{n}u(x,t)|^{2}\mathrm{d}x<\infty,
	\end{equation*}
	and for any $t\in (-T,0)$ and any $\tau_{3}\in \mathbb{R}$
	\begin{equation*}
	\int_{\tau_{3}}^{\infty}|\partial_{x}^{m}u(x,t)|^{2}\mathrm{d}x=\infty.
	\end{equation*}
	\end{cor} 

		The rest    of the paper  is organized  as follows:   in the section 2 we fix the notation to be used throughout the document. Section 3 contains a brief summary of  commutators estimates   involving fractional  derivatives. The section 4 deals with the   local well-posedness. Finally, the section 5 is   devoted   to the proof of Theorem \ref{A} .  
	\section{Notation}
	 The following notation  will be used extensively  throughout this article. The operator  $J^{s}=(1-\partial_{x}^{2})^{s/2}$ denotes the   Bessel potentials of order $-s.$  
	 
	 For $1\leq p\leq \infty,$\, $L^{p}(\mathbb{R})$  is the usual  Lebesgue space  with the norm  $\|\cdot\|_{L^{p}},$  besides for $s\in \mathbb{R},$  we consider the Sobolev space $H^{s}(\mathbb{R})$  is defined  via its usual norm $\|f\|_{H^{s}}=\|J^{s}f\|_{L^{2}}.$ In this contex,  we define   $${\displaystyle H^{\infty}(\mathbb{R})=\bigcap_{s\geq0} H^{s}(\mathbb{R}).}$$
	 
	 Let $f=f(x,t)$  be a function  defined for $x\in \mathbb{R}$ and $t$ in the time interval $[0,T],$ with $T>0$  or in the hole line $\mathbb{R}$. Then if $A$ denotes any of the spaces defined above, we define  the spaces  $L^{p}_{T}A_{x}$ and $L_{t}^{p}A_{x}$ by the norms 
	 \begin{equation*}
	 \|f\|_{L^{p}_{T}A_{x}}=\left(\int_{0}^{T}\|f(\cdot,t)\|_{A}^{p}\,\mathrm{d}t\right)^{1/p}\quad \mbox{and}\quad \|f\|_{L^{p}_{t}A_{x}}=\left(\int_{\mathbb{R}} \|f(\cdot,t)\|_{A}^{p}\,\mathrm{d}t\right)^{1/p},
	 \end{equation*}
	  for $1\leq p\leq \infty$ with the  natural modification in the case $p=\infty.$
	  Moreover, we use similar definitions for the  mixed spaces $L_{x}^{q}L_{t}^{p}$ and $L_{x}^{q}L_{T}^{p}$  with $1\leq p,q\leq \infty.$
	  
	 For two quantities  $A$ and $B$, we denote  $A\lesssim B$  if $A\leq cB$ for some constant $c>0.$ Similarly, $A\gtrsim B$  if  $A\geq cB$ for some $c>0.$  We denote  $A \sim B$  if $A\lesssim B$ and $B\lesssim A.$  The  dependence of the constant $c$  on other parameters  or constants  are usually clear  from the context  and we will often suppress this dependence whenever possible.  
	 
	 For a real number $a$ we will  denote  by $a^{+}$	 instead of $a+\epsilon$, whenever $\epsilon$ is a positive number whose value is small enough.
	\section{Preliminary}
	In this section, we state  several inequalities  to be  used in the next sections.
	
First,  we have     an extension of the Calderon commutator  theorem \cite{calderon} established by B. Baj\v{s}anski et al. \cite{BC}. 
	\begin{thm}
		For any  $p\in (1,\infty)$ and any  $l,m\in \mathbb{Z}^{+}\cup\{0\}$ there exists  $c=c(p;l;m)>0$ such that 
		\begin{equation}\label{eq31}
		\left\|\partial_{x}^{l}\left[\mathcal{H}; \psi\right]\,\partial_{x}^{m}f\right\|_{L^{p}}\leq c\|\partial_{x}^{m+l}\psi\|_{L^{\infty}}\|f\|_{L^{p}}.
		\end{equation}
		\end{thm}
For a different proof see \cite{DMP} Lemma 3.1.
		
	 In our analysis the Leibniz rule for fractional derivatives, established in \cite{{GRF},{KATOP2},{KPV2}} will be crucial. Even though most of these estimates are valid  in  several dimensions, we  will restrict our attention   to the one-dimensional case. 
	\begin{lem}\label{lema1}
		For $s>0,\, p\in [1,\infty)$
		\begin{equation}\label{eq5}
		\left\|D^{s}(fg)\right\|_{L^{p}} \lesssim \left\|f\right\|_{L^{p_{1}}}\left\|D^{s}g\right\|_{L^{p_{2}}}+\left\|g\right\|_{L^{p_{3}}}\left\|D^{s}f\right\|_{L^{p_{4}}}
		\end{equation}
		with
		\begin{equation*}
		\frac{1}{p}=\frac{1}{p_{1}}+\frac{1}{p_{2}}=\frac{1}{p_{3}}+\frac{1}{p_{4}},\quad p_{j}\in(1,\infty],\quad j=1,2,3,4.
		\end{equation*}
	\end{lem}
		Also,  we  will  state the fractional Leibniz rule proved  by Kenig, Ponce and Vega \cite{KPV1}.
		\begin{lem}\label{lem5}
			Let $s=s_{1}+s_{2}\in (0,1)$  with $s_{1},s_{2}\in (0,s),$  and $p,p_{1},p_{2}\in (1,\infty)$ satisfy
			\begin{equation*}
			\frac{1}{p}=\frac{1}{p_{1}}+\frac{1}{p_{2}}.
			\end{equation*} 
			Then,
			\begin{equation}\label{eq33}
			\|D^{s}(fg)-fD^{s}g-gD^{s}f\|_{L^{p}}\lesssim \|D^{s_{1}}f\|_{L^{p_{1}}}\|D^{s_{2}}g\|_{L^{p_{2}}}.
			\end{equation}
			Moreover, the case $s_{2}=0$ and $p_{2}=\infty$ is allowed.
		\end{lem}
		A natural question about Lemma \ref{lem5} is to investigate the   possible   generalization  of the estimate   (\ref{eq33}) when $s\geq 1.$ The answer to this question was given recently by D.Li \cite{Dongli}, where he establishes  new fractional Leibniz rules  for the nonlocal operator $D^{s},\,s>0,$  and related ones, including  various end-point situations.	
		\begin{thm}\label{thm11}
			
			Let $s>0$ and $1<p<\infty.$ Then for any  $s_{1},s_{2}\geq 0$ with $s=s_{1}+s_{2},$  and any $f,g\in \mathcal{S}(\mathbb{R^{n}}),$ the following  hold:
			\begin{enumerate}
				\item If $1<p_{1},p_{2}<\infty$ with $\frac{1}{p}=\frac{1}{p_{1}}+\frac{1}{p_{2}},$ then
				\begin{equation}\label{commutator1}
				\begin{split}
				&\left\|D^{s}(fg)-\sum_{\alpha\leq s_{1}}\frac{1}{\alpha!}\partial^{\alpha}_{x}f D^{s,\alpha}g-\sum_{\beta\leq s_{2}}\frac{1}{\beta!}\partial^{\beta}_{x}g D^{s,\beta}f\right\|_{L^{p}}\lesssim \|D^{s_{1}}f\|_{L^{p_{1}}}\|D^{s_{2}}g\|_{L^{p_{2}}}.
				\end{split}
				\end{equation}
				\item If $p_{1}=p,\, p_{2}=\infty,$\, then
			\begin{equation*}
			\begin{split}
			&\left\|D^{s}(fg)-\sum_{\alpha< s_{1}}\frac{1}{\alpha!}\partial^{\alpha}_{x}f D^{s,\alpha}g-\sum_{\beta\leq s_{2}}\frac{1}{\beta!}\partial^{\beta}_{x}g D^{s,\beta}f\right\|_{L^{p}}\lesssim \|D^{s_{1}}f\|_{L^{p}}\|D^{s_{2}}g\|_{\mathrm{BMO}},
			\end{split}
			\end{equation*}
			where $\|\cdot\|_{\mathrm{BMO}}$  denotes the norm in the BMO space\footnote{For any $f\in L^{1}_{loc}(\mathbb{R}^{n})$, the BMO semi-norm is given by 
				\begin{equation*}
				\|f\|_{\mathrm{BMO}}=\sup_{Q}\frac{1}{|Q|}\int_{Q}|f(y)-(f)_{Q}|\,\mathrm{d}y,
				\end{equation*}
				where $(f)_{Q}$ is the average of $f$ on $Q,$ and the supreme is taken over all cubes $Q$ in $\mathbb{R}^{n}.$ }.
			\item If $p_{1}=\infty,\, p_{2}=p,$\, then
				\begin{equation*}
				\begin{split}
				&\left\|D^{s}(fg)-\sum_{\alpha\leq s_{1}}\frac{1}{\alpha!}\partial^{\alpha}f D^{s,\alpha}g-\sum_{\beta< s_{2}}\frac{1}{\beta!}\partial^{\beta}g D^{s,\beta}f\right\|_{L^{p}}\lesssim \|D^{s_{1}}f\|_{\mathrm{BMO}}\|D^{s_{2}}g\|_{L^{p}}.
				\end{split}
				\end{equation*}
				The operator  $D^{s,\alpha}$ is defined  via Fourier transform\footnote{The precise  form of the Fourier transform  does not matter.}
				\begin{equation*}
				\begin{split}
				&\widehat{D^{s,\alpha}g}(\xi)=\widehat{D^{s,\alpha}}(\xi)\widehat{g}(\xi),\\
				&\widehat{D^{s,\alpha}}(\xi)= i^{-\alpha}\partial_{\xi}^{\alpha}\left(|\xi|^{s}\right).
				\end{split}
				\end{equation*}
			\end{enumerate}
%
			\end{thm}
			\begin{rem}
		As usual  empty summation  (such as $\sum_{0\leq\alpha<0}$) is defined as zero.
	\end{rem}
	\begin{proof}
			For a detailed proof of this Theorem and related results, see \cite{Dongli}.
	\end{proof}
		Next we have  the following commutator  estimates involving  non-homogeneous fractional derivatives, established by Kato and Ponce  .
		\begin{lem}[\cite{KATOP2}]
		Let $s>0$ and $p,p_{2},p_{3}\in (1,\infty)$ and $p_{1},p_{4}\in (1,\infty]$	 be such that 
		\begin{equation*}
		\frac{1}{p}=\frac{1}{p_{1}}+\frac{1}{p_{2}}=\frac{1}{p_{3}}+\frac{1}{p_{4}}.
		\end{equation*}
		Then,
		\begin{equation}\label{eq91}
		\|[J^{s}; f]g\|_{L^{p}}\lesssim \|\partial_{x}f\|_{L^{p_{1}}}\|J^{s-1}g\|_{L^{p_{2}}}+\|J^{s}f\|_{L^{p_{3}}}\|g\|_{L^{p_{4}}}
		\end{equation}
		and 
		\begin{equation}\label{eq90}
		\|J^{s}(fg)\|_{L^{p}}\lesssim \|J^{s}f\|_{L^{p_{1}}}\|g\|_{L^{p_{2}}}+\|J^{s}g\|_{L^{p_{3}}}\|f\|_{L^{p_{4}}}.
		\end{equation}
	\end{lem}
There are many  other reformulations and generalizations of the  Kato-Ponce  commutator  inequalities (cf. \cite{Benyi} and the references therein). Recently  D. Li\,\cite{Dongli}, has obtained a   family of refined Kato-Ponce type inequalities for the operator $D^{s}.$  In particular  he showed that  

\begin{lem}\label{dlkp}
	Let $1<p<\infty.$  Let $1<p_{1},p_{2},p_{3},p_{4}\leq \infty$ satisfy 
	\begin{equation*}
	\frac{1}{p}=\frac{1}{p_{1}}+\frac{1}{p_{2}}=\frac{1}{p_{3}}+\frac{1}{p_{4}}.
	\end{equation*}
	Therefore,
\begin{itemize}
	\item[(a)] If $0< s\leq 1,$  then 
	\begin{equation*}
	\|D^{s}(fg)-fD^{s}g \|_{L^{p}}\lesssim \|D^{s-1}\partial_{x}f \|_{L^{p_{1}}}\|g\|_{L^{p_{2}}}.
	\end{equation*}
		\item[(b)] If $s>1,$  then 
	\begin{equation}\label{kpdl}
	\|D^{s}(fg)-fD^{s}g\|_{L^{p}}\lesssim \|D^{s-1}\partial_{x}f\|_{L^{p_{1}}}\|g\|_{L^{p_{2}}}+\|\partial_{x}f\|_{L^{p_{3}}}\|D^{s-1}g\|_{L^{p_{4}}}.
	\end{equation}
\end{itemize}
		\end{lem}
	For a more detailed exposition on these estimates see section 5 in  \cite{Dongli}.
	
	In addition,  we have the following inequality of  Gagliardo-Nirenberg type:
	\begin{lem}\label{lema2}
		Let  $1<q,p<\infty,\, 1<r\leq \infty$ and $0<\alpha<\beta.$  Then, 
		\begin{equation*}
		\left\|D^{\alpha}f\right\|_{L^{p}}\lesssim c \|f\|_{L^{r}}^{1-\theta}\|D^{\beta}f\|_{L^{q}}^{\theta}
		\end{equation*}
		with
		\begin{equation*}
		\frac{1}{p}-\alpha=(1-\theta)\frac{1}{r}+\theta(\frac{1}{q}-\beta),\quad \theta\in [\alpha/\beta,1].
		\end{equation*}
	\end{lem}
	\begin{proof}
		See \cite{BL} chapter 4.	
	\end{proof}
	Now, we present a result that will help us  to establish the propagation of regularity of solutions of (\ref{eq7}). A previous result was proved  by  Kenig et al.(c.f \cite{KLPV}, Corollary 2.1)  using  the fact  that $J^{r}$ ($r\in \mathbb{R}$) can be seen as a pseudo-differential operator. Thus,  this approach allows to obtain an expression for $J^{r}$ in terms of a convolution with a certain kernel $k(x,y)$ which enjoys some properties on localized regions in $\mathbb{R^{2}}.$ In fact, this is known as the singular integral realization  of a pseudo-differential operator, whose proof can be found in \cite{stein2} Chapter 4.
		
	The estimate  we  consider here involves  the non-local operator $D^{s}$ instead of $J^{s}$. 
		\begin{lem}\label{lemma1}
		Let $m\in \mathbb{Z}^{+}$ and $s\geq0.$   If  $f\in L^{2}(\mathbb{R})$ and $g\in L^{p}(\mathbb{R}),\,\,2\leq p\leq \infty,$ with
		\begin{equation}\label{cond1}
		\mathrm{dist}\left(\supp(f),\supp(g)\right)\geq \delta>0.
		\end{equation}
				Then
		\begin{equation*}
		\left\|g\,\partial^{m}_{x}D^{s}f\right\|_{L^{2}}\lesssim\|g\|_{L^{p}}\|f\|_{L^{2}}.
		\end{equation*}
	\end{lem}
	\begin{proof}
			Let  $f,g$ be  functions in the Schwartz class  satisfying (\ref{cond1}).
	
			Notice that 
		\begin{equation}\label{gp5}
		\begin{split}
		g(x)\left(D_{x}^{s}\partial_{x}^{m}f\right)(x)&=\frac{g(x)}{(2\pi)^{1/2}}\int_{\mathbb{R}}\ex^{ix\xi} |\xi|^{s}\widehat{\partial_{x}^{m}f}(\xi)\,\mathrm{d}\xi\\
		&=\frac{g(x)}{(2\pi)^{1/2}}\int_{\mathbb{R}}|\xi|^{s}\widehat{\left(\tau_{-x}\partial_{x}^{m}f\right)}(\xi)\,\mathrm{d}\xi.
		\end{split}
		\end{equation}
		where $\tau_{h}$ is the translation operator.\footnote{For $h\in\mathbb{R}$  the translation operator $\tau_{h}$ is defined as $\left(\tau_{h}f\right)(x)=f(x-h).$}
		
		Moreover, the last expression in (\ref{gp5}) defines a tempered distribution for $s$ in a suitable class, that will be specified later.  Indeed, for $z\in \mathbb{C}$ with $-1<\mathrm{Re}(z)<0$  
		\begin{equation}\label{gp1}
		\begin{split}
		\frac{1}{(2\pi)^{1/2}}\int_{\mathbb{R}}|\xi|^{z}\widehat{\left(\tau_{-x}\partial_{x}^{m}\varphi\right)}(\xi)\,\mathrm{d}\xi&=c(z)\int_{\mathbb{R}}\frac{\left(\tau_{-x}\partial_{x}^{m}\varphi\right)(y)}{|y|^{1+z}}\,\mathrm{d}y,\quad \forall \varphi\in \mathcal{S}(\mathbb{R})
		\end{split}
		\end{equation}
		with $c(z)$  is independent of $\varphi.$ In fact, evaluating  $\varphi(x)=\ex^{-x^{2}/2}$ in (\ref{gp1})
		yields
		\begin{equation*}
		c(z)=\frac{2^{z}\,\Gamma\left(\frac{z+1}{2}\right)}{\pi^{1/2}\Gamma\left(-\frac{z}{2}\right)}.
		\end{equation*}  		
		Thus, for every  $\varphi\in \mathcal{S}(\mathbb{R})$ the right hand side in (\ref{gp1}) defines a meromorphic function for every test function, which can be extended  analytically to a wider range of complex numbers z's, specifically   $z$ with $\mathrm{Im}(z)=0$ and $\mathrm{Re}(z)=s>0$ that is the case that attains us. By an abuse of notation, we will denote the meromorphic extension and the original as the same.
		
	 Thus,  combining  \eqref{cond1}, \eqref{gp5} and  \eqref{gp1}   it  follows that
	  \begin{equation*}
	  \begin{split}
	  g(x)\left(D_{x}^{s}\partial_{x}^{m}f\right)(x)&=c(s)\int_{\mathbb{R}}\frac{g(x)\left(\tau_{-x}\partial_{x}^{m}f\right)(y)}{|y|^{1+s}}\,\mathrm{d}y\\
&=c(s)g(x)\left(f*\frac{\mathbb{1}_{\{|y|\geq \delta\}}}{|y|^{s+m+1}}\right)(x)
	  	  \end{split}
	 	 \end{equation*}   
	Notice that the kernel in the integral expression is not anymore singular  due to the condition (\ref{cond1}). In fact, in the particular case that $m$ is even, we obtain after apply  integration by parts 
		 \begin{equation*}
	\begin{split}
 g(x)\left(D_{x}^{s}\partial_{x}^{m}f\right)(x)& 
=c(s,m)g(x)\left(f*\frac{\mathbb{1}_{\{|y|\geq \delta\}}}{|y|^{s+m+1}}\right)(x)
	\end{split}
	\end{equation*} 
	and in the case $m$ being odd
	\begin{equation*}
	\begin{split}
	g(x)\left(D_{x}^{s}\partial_{x}^{m}f\right)(x) 
	&=c(s,m)g(x)\left(f*\frac{y\mathbb{1}_{\{|y|\geq \delta\}}}{|y|^{s+m+2}}\right)(x).
	\end{split}
	\end{equation*} 
	 Finally,  in both cases combining  Young's inequality and H\"{o}lder's inequality one gets  
		\begin{equation*}
		\begin{split}
		\|g\,\partial_{x}^{m}D_{x}^{s}f\|_{L^{2}}
		&\lesssim \|g\|_{L^{p}}\| f\|_{L^{2}}  \left\|\frac{\mathbb{1}_{\{|y|\geq \delta\}}}{|\cdot|^{s+m+1}}\right\|_{L^{r}}\\
		&\lesssim \|g\|_{L^{p}}\|f\|_{L^{2}}
		\end{split}
		\end{equation*}
		where the index $p$  satisfies $\frac{1}{2}=\frac{1}{p}+\frac{1}{r},$\,   
		which clearly implies  $p\in [2,\infty],$ as was required. 
	\end{proof}
Further, in the paper  we will use extensively some results about commutator additionally to those presented in previous section. 
Next,  we  will study the smoothing effect for solutions of the dispersive generalized Benjamin-Ono equation \eqref{eq7} following Kato's ideas \cite{KATO1}.
	\subsection{Commutator Expansions }
		In this section  we present  several new main tools obtained by Ginibre and Velo \cite{{GV1}}, \cite{GV2}   which  will be the cornerstone  in the proof of Theorem \ref{A}. They include  commutator expansions together with their estimates. The basic problem is to handle  the non-local operator $D^{s}$ for non-integer $s$ and in particular  to obtain representations  of its commutator with multiplication operators by functions that exhibit as much locality as possible. 
	
Let $a=2\mu+1>1,$   let  $n$ be a non-negative integer and $h$ be a smooth  function with suitable  decay at infinity, for instance with $h'\in C^{\infty}_{0}(\mathbb{R}).$ 
	
	We define the operator 
	\begin{equation}\label{eq8}
	R_{n}(a)=\left[H D^{a}; h\right]-\frac{1}{2}\left(P_{n}(a)-H P_{n}(a)H\right),
	\end{equation}
	\begin{equation}\label{eq105}
	P_{n}(a)=a\sum_{0\leq j\leq n}c_{2j+1}(-1)^{j}4^{-j}D^{\mu-j}\left(h^{(2j+1)}D^{\mu-j}\right)
		\end{equation}
	where
	\begin{equation*}
	c_{1}=1,\quad c_{2j+1}=\frac{1}{(2j+1)!}\prod_{0\leq k<j}\left(a^{2}-(2k+1)^{2}\right)\quad\mbox{and}\quad H=-\mathcal{H}.
	\end{equation*}
 	It was shown in \cite{GV1} that the operator $R_{n}(a)$ can be represented in terms of anti-commutators \footnote{For any two  operators $P$ and $Q$ we denote the anti-commutator by $[P;Q]_{+}=PQ+QP.$}   as follows
	\begin{equation}\label{aj1}
	R_{n}(a)=\frac{1}{2}([H; Q_{n}(a)]_{+}+[D^{a};[H; h]]_{+}),
	\end{equation}
	where  the operator $Q_{n}(a)$ is represented in the Fourier space variables by the integral kernel
	\begin{equation}\label{aj2}
	Q_{n}(a)\longrightarrow (2\pi)^{\frac{1}{2}}
\widehat{h}(\xi-\xi')|\xi\xi'|^{\frac{a}{2}}2aq_{n}(a,t),
	\end{equation}
		with $|\xi|=|\xi'|\ex^{2t}$  and 
	\begin{equation}
	q_{n}(a,t)=\frac{1}{a}(a^{2}-(2n+1)^{2})c_{2n+1}\int_{0}^{t} \sinh^{2n+1}\tau\,\sinh((a(t-\tau)))\,\mathrm{d}\tau.
	\end{equation}
%
Based on  \eqref{aj1} and \eqref{aj2}, Ginibre and Velo \cite{GV2}
obtain the following properties of boundedness and compactness of the operator $R_{n}(a).$
	 	\begin{prop}\label{propo2}
		Let $n$  be a non-negative  integer,   $a\geq 1,\,$ and   $ \sigma\geq 0,$  be such that 
		\begin{equation}\label{eq21}
		2n+1\leq a+2\sigma\leq2n+3.
		\end{equation}
			Then 
		\begin{itemize}
			\item[(a)] The operator $D^{\sigma}R_{n}(a)D^{\sigma}$ is bounded in $L^{2}$ with norm 
			\begin{equation}\label{eq98}
			\left\|D^{\sigma}R_{n}(a)D^{\sigma}f\right\|_{L^{2}}\leq C(2\pi)^{-1/2}\left\|\widehat{(D^{a+2\sigma}h)}\right\|_{L^{1}_{\xi}}\|f\|_{L^{2}}.
			\end{equation}	
			If $a\geq 2n+1,$ one can take  $C=1.$
			\item[(b)] Assume in addition  that
			\begin{equation*}
			2n+1\leq a+2\sigma<2n+3.
			\end{equation*}
			Then the operator ${\displaystyle D^{\sigma}R_{n}(a)D^{\sigma}}$ is compact in $L^{2}(\mathbb{R}).$
					\end{itemize}
	\end{prop}
	\begin{proof}
		See  Proposition 2.2 in \cite{GV2}.
	\end{proof}
		In fact the  Proposition \ref{propo2} is a generalization of a previous result, where  the derivatives of operator $R_{n}(a)$  are not considered (cf.  Proposition 1 in \cite{GV1}). 
	
		The estimative (\ref{eq98}) yields  the following identity of localization of derivatives.
	\begin{lem}\label{prop1}
		Assume $0<\alpha<1.$  Let be  $\varphi\in C^{\infty}(\mathbb{R})$   with $\varphi'\in C^{\infty}_{0}(\mathbb{R}).$ 
		
		Then,
		\begin{equation}\label{eq1}
		\begin{split}
		\int_{\mathbb{R}} \varphi f\, D^{\alpha+1}\partial_{x}f\,\mathrm{d}x&= \left(\frac{\alpha+2}{4}\right)\int_{\mathbb{R}} \left(\left|D^{\frac{\alpha+1}{2}}f\right|^{2}+\left|D^{\frac{\alpha+1}{2}}\mathcal{H}f\right|^{2}\right)\varphi' \,\mathrm{d}x\\
		&\quad+\frac{1}{2}\int_{\mathbb{R}}fR_{0}(\alpha+2)f\,\mathrm{d}x.
		\end{split}
		\end{equation}
	\end{lem}
	\begin{proof}
		The proof follows  the ideas presented in  Proposition 2.12 in \cite{LIPICLA}.
	\end{proof}
	%
	%
	%
\section{ The Linear Problem.}
The aim of this section is  to obtain Strichartz estimates associated to solutions of the IVP (\ref{eq7}).
 
 First,  consider the linear problem
 \begin{equation}\label{HoP}
 \left\{
 \begin{array}{ll}
 \partial_{t}u-D_{x}^{\alpha+1}\partial_{x}u=0, & x,t\in\mathbb{R},\,0<\alpha<1, \\
 u(x,0)=u_{0}(x),&  \\
 \end{array} 
 \right.
 \end{equation}
 whose solution is given by  
 \begin{equation}\label{solh}
 u(x,t)=S(t)u_{0}=\left(\ex^{it\,|\xi|^{\alpha+1}\xi}\widehat{u_{0}}\right){\widecheck{\hspace{0,5mm}}}.
 \end{equation}  
 We begin studying estimates of the unitary group obtained in (\ref{solh}).
\begin{prop}
 Assume that   $0<\alpha<1.$  Let $q,p$ satisfy  $ \frac{2}{q}+\frac{1}{p}=\frac{1}{2} $
      with 
 $2\leq p\leq \infty.$
 
  Then 
  \begin{equation}\label{strichartz}
  \left\| D_{x}^{\frac{\alpha}{q}}S(t)u_{0}\right\|_{L^{q}_{t}L^{p}_{x}}\lesssim \|u_{0}\|_{L^{2}_{x}}\end{equation}
for all $u_{0}\in L^{2}(\mathbb{R}).$
\end{prop}
\begin{proof}
The proof follows as an application  on Theorem 2.1 in \cite{KPV4}.	
\end{proof}
\begin{rem}
	Notice that the  condition  in $p$ implies $q\in[4,\infty],$ which in one of the extremal cases $(p,q)=(\infty, 4)$
 yields 
\begin{equation*}
 \left\| D_{x}^{\frac{\alpha}{4}}S(t)u_{0}\right\|_{L^{4}_{t}L^{\infty}_{x}}\lesssim \|u_{0}\|_{L^{2}_{x}}
\end{equation*}	
	which shows the gain of $\frac{\alpha}{4}$ derivatives globally in time for solutions of (\ref{HoP}).
\end{rem}
\begin{lem}
	Assume  that $0<\alpha<1.$ Let $\psi_{k}$ be a $C^{\infty}(\mathbb{R})$ function  supported  in the interval $[2^{k-1},2^{k+1}]$  where $k\in \mathbb{Z}^{+}.$  Then, the function $H^{\alpha}_{k}$ defined as 
		\begin{equation*}
	H^{\alpha}_{k}(x)= 
	\left\{ \begin{array}{lcc}
	2^{k} &   if  & |x| \leq 1, \\
	2^{\frac{k}{2}}|x|^{-\frac{1}{2}}& if & 1\leq |x|\leq c 2^{k(\alpha+1)}, \\
	\left(1+x^{2}\right)^{-1} &  if &   |x|>c2^{k(\alpha +1)}
	\end{array}
	\right.
	\end{equation*}
	satisfies 
	\begin{equation}\label{eq18.2}
	\left|\int_{-\infty}^{\infty}\ex^{i\left(t\xi|\xi|^{\alpha+1}+x\xi\right)}\psi_{k}(\xi)\mathrm{d}\xi\right|\lesssim	H^{\alpha}_{k}(x) 
	\end{equation}
for $|t|\leq 2,$ where the constant $c$ does not depends on $t$ nor $k.$
	
	Moreover, we have that
	\begin{equation}\label{eqA.1.1.1}
	\sum_{l=-\infty}^{\infty} H^{\alpha}_{k}\left(|l|\right)\lesssim 2^{k\left(\frac{\alpha+1}{2}\right)}.
	\end{equation}
\end{lem}
\begin{proof}
	The proof of estimate (\ref{eq18.2}) is given in  Proposition 2.6 \cite{KPV3} and it  uses arguments of localization  and the classical Van der Corput's Lemma. Meanwhile,   (\ref{eqA.1.1.1}) follows exactly  that of Lemma 2.6 in \cite{LIPICLA}.
\end{proof}
\begin{thm}
	Assume $0<\alpha<1.$	Let $s>\frac{1}{2}.$
	Then,
		\begin{equation*}
	\left\|S(t)u_{0}\right\|_{L^{2}_{x}L^{\infty}_{t}([-1,1])}\leq \left(\sum_{j=-\infty}^{\infty}\sup_{|t|\leq 1}\sup_{j \leq x<j+1}\left|S(t)u_{0}(x)\right|^{2}\right)^{1/2}\lesssim \|u_{0}\|_{H_{x}^{s}}
	\end{equation*}	
	for any $u_{0}\in H^{s}(\mathbb{R}).$
\end{thm}
\begin{proof}
	See Theorem 2.7 in  \cite{KPV3}.
\end{proof}
Next, we recall a maximal  function estimate proved by Kenig, Ponce and Vega  \cite{KPV3}.
\begin{cor}\label{cor1.1}
Assume that   $0<\alpha<1.$ Then, for any $s>\frac{1}{2}$  and any $\eta >\frac{3}{4}$
	\begin{equation*}
\left(	\sum_{j=-\infty}^{\infty} \sup_{|t|\leq T} \sup_{j \leq x<j+1} |S(t)v_{0}|^{2}\right)^{1/2}\lesssim (1+T)^{\eta}\|v_{0}\|_{H^{s}_{x}}.
	\end{equation*}
\end{cor}
\begin{proof}
	See   Corollary 2.8 in \cite{KPV3}.
\end{proof}
\subsection{The Nonlinear Problem}
This section  is devoted to study  general properties of solutions of the non-linear problem
 \begin{equation}\label{HoP1}
 \left\{
 \begin{array}{ll}
 \partial_{t}u-D_{x}^{\alpha+1}\partial_{x}u+u\partial_{x}u=0, & x,t\in\mathbb{R},\,0<\alpha<1, \\
 u(x,0)=u_{0}(x).&  \\
 \end{array} 
 \right.
 \end{equation}
 We begin this section  stating the following local existence theorem proved by Kato \cite{KATO3} and  Saut,Teman \cite {Saut-teman}.
 \begin{thm}\label{sautth}
 	\begin{enumerate}
 	  \item For any $u_{0}\in H^{s}(\mathbb{R})$ with $s>\frac{3}{2}$ there exists a unique  solution  $u$ to (\ref{HoP1})  in the class  ${\displaystyle C([-T,T]: H^{s}(\mathbb{R}))}$  with  $T=T(\|u_{0}\|_{H^{s}})>0.$
 	  \item For any $T'<T$ there exists  a neighborhood  $V$ of  $u_{0}$  in $H^{s}(\mathbb{R})$ such that the map  $\widetilde{u_{0}}\longmapsto\widetilde{u}(t)$ from $V$  into $C\left([-T',T']:H^{s}(\mathbb{R})\right)$ is continuous.
 	  \item If $u_{0}\in H^{s'}(\mathbb{R})$   with $s'>s,$  then the  time of existence  $T$  can be taken  to depend only on  $\|u_{0}\|_{H^{s}}.$
 	 \end{enumerate} 		 	
 \end{thm}
Our first goal  will be obtain    some energy estimates satisfied by  smooth  solutions of the IVP   (\ref{HoP1}).

 We  firstly present a result that arises  as a consequence of commutator estimates.
	\begin{lem}\label{kpp}
		Suppose that  $0<\alpha<1.$
		 Let  $u\in C([0,T]:H^{\infty}(\mathbb{R}))$ be a smooth solution of (\ref{HoP1}). 
		If $s>0$ is given, then
		\begin{equation}\label{eq295}
		\|u\|_{L^{\infty}_{T} H^{s}_{x}}\lesssim \|u_{0}\|_{H^{s}_{x}}\ex^{c \|\partial_{x}u\|_{L^{1}_{T}L^{\infty}_{x}}}.
		\end{equation}
			\end{lem}
	\begin{proof}  
		Let $s>0.$		
		By a standard energy estimate argument we have that
		\begin{equation*}
	\frac{1}{2}\frac{\mathrm{d}}{\mathrm{d}t}\int_{\mathbb{R}}\left(J^{s}_{x}u\right)^{2}\mathrm{d}x+\int_{\mathbb{R}}\left[J_{x}^{s}; u\right]\partial_{x}u\, J_{x}^{s}u\,\mathrm{d}x+\int_{\mathbb{R}}uJ_{x}^{s}u J_{x}^{s}\partial_{x}u\,\mathrm{d}x=0.
	\end{equation*}
	Hence integration by parts, Gronwall's inequality   and commutator estimate (\ref{eq91}) lead to (\ref{eq295}).
	\end{proof}
	\begin{rem}
		In view of the energy estimate (\ref{eq295}), the key point  to obtain \emph{a priori} estimates in $H^{s}_{x}(\mathbb{R})$ is to control $\|\partial_{x}u\|_{L^{1}_{T}L^{\infty}_{x}}$ at the  $H^{s}_{x}(\mathbb{R})-$level.   	
	\end{rem}
Additionally to this estimate, we will present  the smoothing effect provided by solutions of dispersive generalized Benjamin-Ono equation. In fact, the   smoothing   effect was first  observed by Kato in the context of the Korteweg-de Vries equation (see \cite{KATO1}). Following Kato's approach joint with the commutator expansions,  we  present a result  proved by Kenig-Ponce-Vega  \cite{KPV3} (see  Lemma 5.1). 
\begin{prop}\label{sm1}
	Let $\varphi$ denote a  non-decreasing smooth function  such that  $\supp \varphi'\subset(-1,2)$  and $\varphi'|_{[0,1)}=1.$
	For $j\in \mathbb{Z},$  we define  $\varphi_{j}(\cdot)=\varphi(\cdot -j).$ 
	Let 
	$u\in C([0,T]: H^{\infty}({\mathbb{R}}))$ be a real smooth solution  of (\ref{eq7})
	with $0<\alpha<1.$  
	Assume also, that $s\geq 0$  and $r>1/2.$ 
	
	Then,
	\begin{equation}\label{eq101}
	\begin{split}
	&\left(\int_{0}^{T}\int_{\mathbb{R}}\left(\left|D_{x}^{s+\frac{\alpha+1}{2}}u(x,t)\right|^{2} +\left|D_{x}^{s+\frac{\alpha+1}{2}}\mathcal{H}u(x,t)\right|^{2}\right)\varphi'_{j}(x) \,\mathrm{d}x\,\mathrm{d}t\right)^{1/2}\\
	&\lesssim \left(1+T+\|\partial_{x}u\|_{L^{1}_{T} L^{\infty}_{x}}+T\|u\|_{L^{\infty}_{T} H^{r}_{x}}\right)^{1/2}\|u\|_{L^{\infty}_{T}H^{s}_{x}}.
	\end{split}
	\end{equation}
\end{prop}
In addition to the smoothing effect presented above,  we  will need  the following  localized version  of the $H^{s}(\mathbb{R})$-norm. For this propose we will consider a  cutoff function  $\varphi$, with the same characteristics that in Proposition \ref{sm1}.

\begin{prop}\label{propo3}
	Let $s\geq 0$.  Then,   for any $f\in H^{s}(\mathbb{R})$
	\begin{equation*}
	\|f\|_{H^{s}(\mathbb{R})}\sim \left(\sum_{j=-\infty}^{\infty} \left\|f\varphi_{j}'\right\|_{H^{s}(\mathbb{R})}^{2}\right)^{1/2}.
	\end{equation*}
\end{prop}
Hence our first goal in establishing the Local well-posedness of (\ref{HoP1}), will start off in obtain Strichartz estimates  associated to solutions of 
\begin{equation}\label{eqta}
\partial_{t}u-D_{x}^{1+\alpha}\partial_{x}u=F.
\end{equation}
\begin{prop}\label{eq50}
	Assume that $0<\alpha<1,\, T>0$  and $ \sigma\in [0,1].$  Let $u$ be a smooth solution  to (\ref{eqta}) defined on the time interval $[0,T].$  Then there exist $0\leq\mu_{1},\mu_{2}<1/2$ such that
	\begin{equation}\label{AA.2}
	\|\partial_{x}u\|_{L^{2}_{T}L^{\infty}_{x}}\lesssim T^{\mu_{1}}\left\|J^{1-\frac{\alpha}{4}+\frac{\sigma}{4}+\varepsilon} u\right\|_{L^{\infty}_{T}L^{2}_{x}} +T^{\mu_{2}}\left\|J^{1-\frac{\alpha}{4}-\frac{3\sigma}{4}+\varepsilon}F\right\|_{L^{2}_{T}L^{2}_{x}}
	\end{equation} 
	for any $\varepsilon>0.$
	\end{prop}
	\begin{rem}
		The optimal choice in the parameters present in the  estimate (\ref{AA.2})  corresponds to
		$\sigma=\frac{1-\alpha}{2}.$ Indeed,  
		as is pointed out by Kenig and Koenig in the case of the Benjamin-Ono equation (case $\alpha=0$) (see  Remarks in Proposition 2.8 \cite{KK}) given a linear estimate  of the form 
		\begin{equation*}
		\|\partial_{x}u\|_{L^{2}_{T}L^{\infty}_{x}}\lesssim T^{\mu_{1}}\left\|J^{a} u\right\|_{L^{\infty}_{T}L^{2}_{x}} +T^{\mu_{2}}\left\|J^{b}F\right\|_{L^{2}_{T}L^{2}_{x}}
		\end{equation*}
	 the idea  is to apply the smoothing effect  (\ref{eq101}) and absorb  as many as derivatives as possible  of the function $F.$  Concerning to our case, the approach   requires the choice $a=b+\frac{1-\alpha}{2};$	  
		  this particular choice,  $\sigma=\frac{1-\alpha}{2},$  in the estimate (\ref{AA.2}) provides  the regularity  $s>9/8-3\alpha/8$ in  Theorem \ref{B.1}.
	\end{rem}
\begin{proof}
		Let   $f=\sum_{k}f_{k}$ denote  the Littlewood-Paley decomposition of a function  $f.$ More precisely  we choose functions  $\eta, \chi \in C^{\infty}(\mathbb{R})$  with 
	$\supp(\eta) \subseteq \{\xi: 1/2<|\xi|<2\} $  and $\supp(\chi)\subseteq \{\xi: |\xi|<2\},$ such that 
	\begin{equation*}
	\sum_{k=1}^{\infty}\eta(\xi/2^{k})+\chi(\xi)=1
	\end{equation*}
	and  $f_{k}=P_{k}(f),$ where $\widehat{(P_{0}f)}(\xi)=\chi(\xi)\widehat{f}(\xi)$	  and        ${\displaystyle \widehat{(P_{k}f)}(\xi)=\eta(\xi/2^{k})\widehat{f}(\xi)}$  for all $k\geq 1.$ 
%

	  	Fix $\varepsilon>0.$ Let $p>\frac{1}{\varepsilon}.$
	By Sobolev embedding  and Littlewood-Paley Theorem  it follows that
		\begin{equation*}
		\begin{split}
	\|f\|_{L^{\infty}_{x}}&\lesssim\left\|J^{\varepsilon}f\right\|_{L^{p}_{x}}\sim \left\|\left(\sum_{k=0}^{\infty}\left|J^{\varepsilon}P_{k}f\right|^{2}\right)^{1/2} \right\|_{L^{p}_{x}}=\left\|\sum_{k=0}^{\infty}\left|J^{\varepsilon}P_{k}f\right|^{2}\right\|_{L^{p/2}_{x}}^{1/2}\lesssim \left(\sum_{k=0}^{\infty}\left\|J^{\varepsilon}P_{k}f\right\|_{L^{p}_{x}}^{2}\right)^{1/2}.
	\end{split}
	\end{equation*}  
	  Therefore, to obtain (\ref{AA.2}) it enough to prove that for $p>2$
	  \begin{equation*}
	  \|\partial_{x}u_{k}\|_{L^{2}_{T}L^{p}_{x}}\lesssim \left\|D_{x}^{1-\frac{\alpha}{4}+\frac{\sigma}{4}+\frac{\alpha-\sigma}{2p}}u_{k}\right\|_{L^{\infty}_{T}L^{2}_{x}}+ \left\|D_{x}^{1-\frac{\alpha}{4}-\frac{3\sigma}{4}+\frac{\alpha-\sigma}{2p}}F_{k}\right\|_{L^{2}_{T}L^{2}_{x}},\quad k\geq 1.
	  \end{equation*}
	The estimate for the case $k=0$ follows using H\"{o}lder's inequality and  (\ref{strichartz}). For such reason we fix $k\geq 1,$ and at  these level of frequencies  we have that
	\begin{equation*}
	\partial_{t}u_{k}-D^{\alpha+1}_{x}\partial_{x}u_{k}=F_{k}.
	\end{equation*}
	Consider a partition of the interval $[0,T]=\bigcup_{j\in J}I_{j}$ where $I_{j}=[a_{j},b_{j}],$ and $T=b_{j}$ for some $j.$ Indeed,  we  choose a quantity $\sim 2^{k\sigma}T^{1-\mu}$ of intervals,  with  length $|I_{j}|\sim 2^{-k\sigma}T^{\mu}, $ where $\mu$ is a positive number to be fixed.  
	
	Let $q$ be such that
	$$\frac{2}{q}+\frac{1}{p}=\frac{1}{2}.$$
	
	Using that $u$  solves the integral equation 
	\begin{equation}
u(t)=S(t)u_{0}+\int_{0}^{t}S(t-t')F(t')\,\mathrm{d}t',
\end{equation}	
we deduce that 
	 \begin{equation*}
	 \begin{split}
	 &\|\partial_{x}u_{k}\|_{L^{2}_{T}L^{p}_{x}}\\
&\lesssim \left(T^{\mu}2^{-k\sigma}\right)^{\left(\frac{1}{2}-\frac{1}{q}\right)}\left(\sum_{j\in J}\left\|S(t-a_{j})\partial_{x}u_{k}(a_{j})\right\|_{L^{q}_{I_{j}}L^{p}_{x}}^{2} +\left\|\int_{a_{j}}^{t} S(t-s) \partial_{x}F_{k}(s)\mathrm{d}s\right\|_{L^{q}_{I_{j}}L^{p}_{x}}^{2}\right)^{1/2}.
\end{split}
\end{equation*}
In this sense, it follows from \eqref{strichartz} that
\begin{equation*}
\begin{split}	
&\|\partial_{x}u_{k}\|_{L^{2}_{T}L^{p}_{x}}\\
 &\lesssim \left(T^{\mu}2^{-k\sigma}\right)^{\left(\frac{1}{2}-\frac{1}{q}\right)}\left\{\sum_{j\in J}\left\|D_{x}^{-\frac{\alpha}{q}}\partial_{x}u_{k}(a_{j})\right\|_{L^{2}_{x}}^{2}+\sum_{j\in J}\left(\int_{I_{j}}\left\|D_{x}^{-\frac{\alpha}{q}}\partial_{x}F_{k}(t)\right\|_{L^{2}_{x}}\mathrm{d}t\right)^{2} \right\}^{1/2}\\
&\lesssim \left(T^{\mu}2^{-k\sigma}\right)^{\left(\frac{1}{2}-\frac{1}{q}\right)}\left\{\left(\sum_{j\in J}\left\|D_{x}^{1-\frac{\alpha}{q}}u_{k}\right\|_{L^{\infty}_{T}L^{2}_{x}}^{2}\right)^{1/2} +\left(\sum_{j\in J}T^{\mu}2^{-k\sigma}\int_{I_{j}}\left\|D_{x}^{1-\frac{\alpha}{q}}F_{k}(t)\right\|_{L^{2}_{x}}^{2}\,\mathrm{d}t\right)^{1/2}\right\}\\
&\lesssim \left(T^{\mu}2^{-k\sigma}\right)^{\left(\frac{1}{2}-\frac{1}{q}\right)}\left(T^{1-\mu}2^{k\sigma}\right)^{1/2}\left\|D_{x}^{1-\frac{\alpha}{q}}u_{k}\right\|_{L^{\infty}_{T} L^{2}_{x}}\\
&\quad +\left(T^{\mu}2^{-k\sigma}\right)^{\left(\frac{1}{2}-\frac{1}{q}\right)}\left(T^{\mu}2^{-k\sigma}\right)^{1/2}\left(\int_{0}^{T}\left\|D_{x}^{1-\frac{\alpha}{q}}F_{k}(t)\right\|_{L^{2}_{x}}^{2}\mathrm{d}t\right)^{1/2}\\
&\lesssim T^{1/2 -\mu/q}\left\|D_{x}^{1-\frac{\alpha}{q}+\frac{\sigma}{q}}u_{k}\right\|_{L^{\infty}_{T} L^{2}_{x}}+T^{\mu(1-1/q)}\left\|D_{x}^{1-\frac{\alpha}{q}+\frac{\sigma}{q}-\sigma}F_{k}\right\|_{L^{2}_{T}L^{2}_{x}}.
\end{split}
\end{equation*}
Since,
$1-\frac{\alpha}{q}+\frac{\sigma}{q}=1-\frac{\alpha}{4}+\frac{\sigma}{4}+\frac{\alpha-\sigma}{2p}$  and $1-\frac{\alpha}{q}+\frac{\sigma}{q}-\sigma=1-\frac{\alpha}{4}-\frac{3\sigma}{4}+\frac{\alpha-\sigma}{2p}.$  We recall that  $\varepsilon>\frac{1}{p},$\,  $\sigma\in [0,1]$ and $\alpha\in (0,1),$ then $\varepsilon+\frac{\alpha-\sigma}{2p}>\frac{\alpha-\sigma+2}{2p}>0.$ Next, we choose $\mu_{1}=\frac{1}{2}-\frac{\mu}{q},\, \mu_{2}=\mu(1-\frac{1}{q})$  with the particular choice  $\mu=1/2.$

Gathering the inequalities above follows the proposition.
\end{proof}
Now we turn our attention to the proof of Theorem  \ref{B.1}. Our starting point will be the energy estimate \eqref{eq295}, that as was remarked above, the key point is to establish  a priori control  of $\|\partial_{x}u\|_{L^{1}_{T}L^{\infty}_{x}}.$
\section{Proof of Theorem \ref{B.1}}\label{proof1}
\subsection{A priori estimates} \label{seccion1.2}
First notice that by scaling, it is enough to  deal with  small initial data  in the $H^{s}-$norm.   Indeed, if $u(x,t)$ is a solution of (\ref{eq7}) defined on a time interval $[0,T],$ for some positive time $T,$ then for all  $\lambda>0,\, u_{\lambda}(x,t)=\lambda^{1+\alpha}u(\lambda x,\lambda^{2+\alpha}t)$ is also solution with initial data $u_{0,\lambda}(x)=\lambda^{1+\alpha}u_{0}(\lambda x),$  and    time interval  $\left[0,T/\lambda^{2+\alpha}\right].$

For any $\delta>0,$ we define $B_{\delta}(0)$ as the ball with center at the origin in $H^{s}(\mathbb{R})$ and radius $\delta.$

Since 
\begin{equation*}
 \|u_{0,\lambda}\|_{L^{2}_{x}}=\lambda^{\frac{1+2\alpha}{2}}\|u_{0}\|_{L^{2}_{x}}
\quad  \mbox{and} \quad\left\|D_{x}^{s}u_{0,\lambda}\right\|_{L^{2}_{x}}=\lambda^{\frac{1}{2}+\alpha+s}\left\|D_{x}^{s}u_{0}\right\|_{L^{2}_{x}},
\end{equation*}
then 
$$ \|u_{0,\lambda}\|_{H^{s}_{x}}\lesssim \lambda^{\frac{1}{2}+\alpha}(1+\lambda^{s})\|u_{0}\|_{H^{s}_{x}},$$
so we can force   $u_{\lambda}(\cdot,0)$ to belong to the ball $B_{\delta}(0)$ by choosing
the parameter $\lambda$ with the condition
\begin{equation*}
	\lambda\sim\min\left\{\delta^{\frac{2}{1+2\alpha}}\|u_{0}\|_{H^{s}_{x}}^{-\frac{2}{1+2\alpha}},1\right\}.
\end{equation*}
Thus, the existence  and uniqueness of a solution to (\ref{eq7}) on the time interval $[0,1]$ for small initial data $\|u_{0}\|_{H^{s}_{x}}$ will ensure  the existence  and uniqueness  of a solution  to (\ref{eq7}) for arbitrary large  initial data on a time interval $[0,T]$  with
\begin{equation*}
	T\sim \min\left\{1,\|u_{0}\|_{H^{s}_{x}}^{-\frac{2(2+\alpha)}{1+2\alpha}}\right\}.
\end{equation*}
Thus, without loss of generality  we will assume that $T\leq 1,$  and that  
$${\displaystyle 
	\Lambda:=\|u_{0}\|_{L^{2}}+\|D^{s}u_{0}\|_{L^{2}}\leq \delta,}$$
where  $\delta$ is a small positive number to be fixed later.

We fix $s$ such  that  $s(\alpha)=\frac{9}{8}-\frac{3\alpha}{8}<s<\frac{3}{2}-\frac{\alpha}{2}$  and set 
$\varepsilon=s-s(\alpha)>0.$ 

Next,  taking    $\sigma=\frac{1-\alpha}{2}>0,$\,   $F=-u\partial_{x}u$ in   \eqref{AA.2}  together with \eqref{eq295} yields
\begin{equation}\label{eqta2}
\begin{split}
\|\partial_{x}u\|_{L^{2}_{T}L^{\infty}_{x}}&\lesssim T^{\mu_{1}} \|J^{s} u\|_{L^{\infty}_{T}L^{2}_{x}} +T^{\mu_{2}}\left\|J^{1-\frac{\alpha}{4}-\frac{3\sigma}{4}+\varepsilon}(u\partial_{x}u)\right\|_{L^{2}_{T}L^{2}_{x}}\\
&\lesssim \Lambda+\Lambda\ex^{c\|\partial_{x}u\|_{L^{2}_{T}L_{x}^{\infty}}}+\left\|D_{x}^{s+\frac{\alpha-1}{2}}(u\partial_{x}u)\right\|_{L_{T}^{2}L_{x}^{2}}.
\end{split}
\end{equation}
Now, to analyze the product  coming from the nonlinear term  we use the Leibniz  rule for fractional derivatives \eqref{eq33} joint with the energy estimate \eqref{eq295} as follows
\begin{equation}\label{eqta3}
\begin{split}
\left\|D_{x}^{s+\frac{\alpha-1}{2}}(u\partial_{x}u)\right\|_{L_{T}^{2}L_{x}^{2}}&\lesssim \left\|u D_{x}^{s+\frac{\alpha-1}{2}}\partial_{x}u\right\|_{L_{T}^{2}L_{x}^{2}}+\left\| \|\partial_{x}u(t)\|_{L_{x}^{\infty}}\left\|D_{x}^{s+\frac{\alpha-1}{2}}u(t)\right\|_{L_{x}^{2}}\right\|_{L_{T}^{2}}\\
& \lesssim \left\|u D_{x}^{s+\frac{\alpha-1}{2}}\partial_{x}u\right\|_{L_{T}^{2}L_{x}^{2}}+\Lambda\|\partial_{x}u\|_{L^{2}_{T}L^{\infty}_{x}}\ex^{c\|\partial_{x}u\|_{L_{T}^{2}L_{x}^{\infty}}}.
\end{split}
\end{equation}
To handle the first term in the right hand side above, we incorporate  Kato's smoothing effect estimate  obtained in (\ref{eq101}) in the following way 
\begin{equation}\label{eqta4}
\begin{split}
\left\|uD_{x}^{s+\frac{\alpha-1}{2}}\partial_{x}u\right\|_{L_{T}^{2}L_{x}^{2}} 
&\leq\left(\sum_{j=-\infty}^{\infty}\int_{0}^{T}\|u(t)\|_{L^{\infty}_{[j,j+1)}}^{2}\left\|D_{x}^{s+\frac{\alpha+1}{2}}\mathcal{H}u(t)\right\|_{L_{[j,j+1)}^{2}}^{2}\mathrm{d}t\right)^{1/2}\\
&\lesssim \left(\sum_{j=-\infty}^{\infty}\|u\|_{L_{T}^{\infty}L_{[j, j+1)}^{\infty}}^{2}\right)^{1/2}\left(1+\Lambda\right)\Lambda \ex^{c\|\partial_{x}u\|_{L_{T}^{2}L_{x}^{\infty}}}
\end{split}
\end{equation} 
In summary, after  gathering the estimates (\ref{eqta2})-(\ref{eqta4}) yields
\begin{equation}\label{eqta5}
\begin{split}
\|\partial_{x}u\|_{L^{2}_{T}L^{\infty}_{x}}&\lesssim
\Lambda(1+\Lambda)\ex^{c\|\partial_{x}u\|_{L^{2}_{T}L_{x}^{\infty}}}\left(\sum_{j=-\infty}^{\infty}\|u\|_{L_{T}^{\infty}L_{[j, j+1)}^{\infty}}^{2}\right)^{1/2}\\
&\quad + \Lambda +\Lambda\ex^{c\|\partial_{x}u\|_{L^{2}_{T}L_{x}^{\infty}}}.
\end{split}
\end{equation}
Since $u$ is a solution to (\ref{HoP1}), then by 
 Duhamel's formula it follows that
\begin{equation*}
u(t)=S(t)u_{0}-\int_{0}^{t}S(t-s)\left(u\partial_{x}u\right)(s)\,\mathrm{d}s
\end{equation*}
where $S(t)=\ex^{tD_{x}^{\alpha+1}\partial_{x}}.$

Now, we  fix $\eta>0$ such that  $\eta<\frac{1+\alpha}{8};$ this choice implies that $\eta+\frac{1}{2}<s+\frac{\alpha-1}{2}.$  Hence, Sobolev's embedding, H\"{o}lder's inequality and  Corollary  \ref{cor1.1} produce
\begin{equation}\label{eqta6}
\begin{split}
&\left(\sum_{j=-\infty}^{\infty}\|u\|_{L_{T}^{\infty}L_{[j, j+1)}^{\infty}}^{2}\right)^{1/2}\\  
&\lesssim \left(\sum_{j=-\infty}^{\infty} \left\|S(t)u_{0}\right\|_{L_{T}^{\infty}L_{[j, j+1)}^{\infty}}^{2}\right)^{1/2}+ \left(\sum_{j=-\infty}^{\infty}\left\|\int_{0}^{t}S(t-s)(u\partial_{x}u)(s)\mathrm{d}s\right\|_{L_{T}^{\infty}L_{[j, j+1)}^{\infty}}^{2}\right)^{1/2}\\
&\lesssim (1+T)\Lambda+(1+T)\|u\partial_{x}u\|_{L_{T}^{1}H_{x}^{\eta+1/2}}\\
&\lesssim \Lambda+\left\|u\partial_{x}u\right\|_{L_{T}^{1}L_{x}^{2}}+\left\|D_{x}^{\eta+1/2}(u\partial_{x}u)\right\|_{L_{T}^{1}L^{2}_{x}}\\
&\lesssim \Lambda+\Lambda\|\partial_{x}u\|_{L^{2}_{T}L_{x}^{\infty}}+\left\|D_{x}^{\eta+1/2}(u\partial_{x}u)\right\|_{L^{2}_{T}L_{x}^{2}}.
\end{split}
\end{equation}
 Employing an  argument similar to the one  applied in (\ref{eqta3}) and  (\ref{eqta5})  it is  possible to bound the last term in the right hand side as follows
 \begin{equation}\label{eqta7}
 \begin{split}
 \left\|D_{x}^{\eta+1/2}(u\partial_{x}u)\right\|_{L^{2}_{T}L_{x}^{2}}&\lesssim \left(\sum_{j=-\infty}^{\infty} \|u\|_{L_{T}^{\infty}L_{[j,j+1)}^{\infty}}^{2}\right)^{1/2}\Lambda(\Lambda+1)\ex^{c\|\partial_{x}u\|_{L_{T}^{2}L_{x}^{\infty}}}\\
 &\quad + \Lambda\ex^{c\|\partial_{x}u\|_{L_{T}^{2}L_{x}^{\infty}}}.
 \end{split}
 \end{equation}
Next, 
we define 
\begin{equation*}
\phi(T)=\left(\int_{0}^{T}\left\|\partial_{x}u(s)\right\|_{L_{x}^{\infty}}^{2}\,\mathrm{d}s\right)^{1/2}+\left(\sum_{j=-\infty}^{\infty}\|u\|_{L_{T}^{\infty}L_{[j, j+1)}^{\infty}}^{2}\right)^{1/2}
\end{equation*}
which is a continuous, non-decreasing  function of $T$.

From obtained in (\ref{eqta5}), (\ref{eqta6}) and (\ref{eqta7}) follows that
\begin{equation*}
\begin{split}
\phi(T)&\lesssim  \Lambda(\Lambda+1)\phi(T)\ex^{c \,\phi(T)}+\Lambda\ex^{c\,\phi(T)}\phi(T)+\Lambda \ex^{c\,\phi(T)}+\Lambda+\Lambda\phi(T).
\end{split}
\end{equation*}
Now, if we suppose that $\Lambda\leq \delta\leq 1$ we obtain 
\begin{equation*}
\phi(T)\leq c\Lambda+c\Lambda\ex^{c\phi(T)}
\end{equation*}
for some constant $c>0.$

To complete the proof
we will show that  there exists $\delta>0,$ such that  $\Lambda \leq \delta,$    then  
$\phi(1)\leq A,$ \; for  some constant $A>0.$

To do this, we define the function 
 \begin{equation}\label{eq32}
\Psi(x,y)=x-cy-cy\ex^{cx}.
\end{equation}
First notice that  $\Psi(0,0)=0$\, and $\partial_{x}\Psi(0,0)=1.$
Then the  Implicit Function Theorem asserts that  there exists $\delta>0,$  and a  smooth function $\xi(y)$ such that  $\xi(0)=0,$ and $\Psi(\xi(y),y)=0$\, for  \,$|y|\leq \delta.$ 

Notice that the condition $\Psi(\xi(y),y)=0$ implies that $\xi(y)>0$ for  $y>0.$ Moreover, since  $\partial_{x}\Psi(0,0)=1,$ then the function $\Psi(\cdot,y)$ is increasing close to  $\xi(y),$ whenever $\delta$  is chosen sufficiently small.

Let us suppose that  $\Lambda\leq \delta,$ and set  $\lambda=\xi(\Lambda).$ 
Then, 
combining interpolation and  Proposition \ref{propo3} we obtain 
\begin{equation*}
\begin{split}
\phi(0)&=\left(\sum_{j=-\infty}^{\infty} \left(\sup_{x\in [j,j+1)}|u(x,0)|\right)^{2}\right)^{1/2} 
\lesssim \|u_{0}\|_{H^{s}(\mathbb{R})}\leq c_{1}\|u_{0}\|_{L^{2}}+c_{1}\|D_{x}^{s}u_{0}\|_{L_{x}^{2}}
\end{split}
\end{equation*}
where  we take $c>c_{1}.$  

Therefore 
\begin{equation*}
\phi(0)\leq c_{1}\Lambda<c\Lambda+c\Lambda\ex^{c\xi(\Lambda)}=\lambda.
\end{equation*}
Suppose that  $\phi(T)>\lambda$ \, for some $T\in (0,1)$ and define
\begin{equation*}
 T_{0}=\inf\left\{T\in (0,1)\,|\,\phi(T)>\lambda\right\}.
  \end{equation*}
   Hence, $T_{0}>0$ and $ \phi(T_{0})=\lambda,$ besides, there exists a decreasing sequence $\{T_{n}\}_{n\geq 1}$  converging to $T_{0}$ such that $\phi(T_{n})>\lambda.$   In addition,  notice that  \eqref{eq32} implies   $\Psi(\phi(T), \Lambda)\leq 0$  for all $T\in [0,1].$ 

Since the function  $\Psi(\cdot,\Lambda)$ is  increasing  near $\lambda$ it  implies  that 
$$\Psi(\phi(T_{n}),\Lambda)>\Psi(\phi(T_{0}),\Lambda)=\Psi(\lambda,\Lambda)=\Psi(\xi(\Lambda),\Lambda)=0,$$ for $n$ sufficiently large.

This is a contradiction with the fact that $\phi(T)>\lambda$. So we conclude $\phi(T)\leq A$ for all $T\in (0,1),$ as was claimed. Thus,  $\phi(1)\leq A.$

In conclusion we have proved that
\begin{equation}\label{eq1.15}
\phi(T)=\left(\int_{0}^{T}\left\|\partial_{x}u(s)\right\|_{L_{x}^{\infty}}^{2}\,\mathrm{d}s\right)^{1/2}+\left(\sum_{j=-\infty}^{\infty}\|u\|_{L_{T}^{\infty}L_{[j, j+1)}^{\infty}}^{2}\right)^{1/2}\lesssim \|u_{0}\|_{H^{s}_{x}},\quad \forall T\in[0,1].
\end{equation}
At this stage, the existence, uniqueness, and  continuous dependence  on the initial data  follows from  the standard  compactness and Bona-Smith approximation  arguments (see for example \cite{KPV3} and  \cite{Ponce1}).

\section{Proof of Theorem \ref{A}}\label{ppgrg}
The aim of this section is to prove  Theorem \ref{A}. To achieve this goal is necessary  to take into account  two important aspects  of our analysis. First,  the ambient space, that in our case  is the   Sobolev space where the theorem is valid together with the properties satisfied by the real  solutions of the dispersive generalized Benjamin-Ono  equation. In second place, the auxiliaries weights functions  involved in the energy estimates that  we will describe in detail.  

The following is a summary of the local well-posedness and  Kato's smoothing effect presented in the previous sections.
\begin{thmx}\label{C}
	If $u_{0}\in H^{s}(\mathbb{R}),$\, $s\geq\frac{3-\alpha}{2},$\, $ \alpha\in(0,1),$ then there exist a positive time  \linebreak 
	$T=T(\|u_{0}\|_{H^{s}})>0$ and a unique  solution of the IVP \eqref{eq7} such that
	\begin{itemize}
	\item[(a)] $u\in C([-T,T]: H^{s}(\mathbb{R})),$
	\item[(b)] $\partial_{x}u\in L^{1}([-T,T]: L^{\infty}(\mathbb{R})),$ \qquad (Strichartz),
	\item[(c)] Smoothing effect: for $R>0,$
		\begin{equation}\label{t1}
	\int_{-T}^{T}\int_{-R}^{R} \left(\left|\partial_{x}D_{x}^{r+\frac{\alpha+1}{2}}u\right|^{2}+\left|\mathcal{H}\partial_{x}D_{x}^{r+\frac{\alpha+1}{2}}u\right|^{2}\right)\,\mathrm{d}x\,\mathrm{d}t\leq C \quad 
	\end{equation} 
with\quad  $r\in \left(\frac{9-3\alpha}{8},s\right]$ and  $C=C(\alpha;R;T;\|u_{0}\|_{H_{x}^{s}})>0.$ 
		\end{itemize}
\end{thmx}
Since we have set the Sobolev space where we will work, the next step is the description of the cutoff functions to be used in the proof.

In this part we consider   families of cutoff functions that will be used systematically  in the proof of Theorem \ref{A}. This collection of weights functions   were  constructed originally in \cite{ILP1} and \cite{KLPV} in the proof of Theorem  \ref{m7}.

More precisely,   for $\epsilon>0$ and $b\geq 5\epsilon$ define the  families of functions 
\begin{equation*}
\chi_{\epsilon,b},\;  \phi_{\epsilon,b},\; \widetilde{\phi_{\epsilon,b}},\; \psi_{\epsilon},\eta_{\epsilon,b}\in C^{\infty}(\mathbb{R})
\end{equation*}
satisfying the following properties:
\begin{enumerate}
	\item ${\displaystyle 
	 \chi_{\epsilon,b}'\geq 0,
}$	
\item ${\displaystyle 
\chi_{\epsilon,b}(x)=
\left\{
\begin{array}{ll}
0, & x\leq \epsilon \\
1, & x\geq b,
\end{array} 
\right. 
}$
\item ${\displaystyle  \supp(\chi_{\epsilon, b})\subseteq [\epsilon,\infty);}$
\item ${\displaystyle \chi_{\epsilon, b}'(x)\geq\frac{1}{10(b-\epsilon)}\mathbb{1}_{[2\epsilon,b-2\epsilon]}(x),}$
\item ${\displaystyle \supp(\chi_{\epsilon,b}') \subseteq [\epsilon,b];}$ 
\item There exists real numbers $c_{j}$ such that
\begin{equation*}
\left|\chi_{\epsilon, b}^{(j)}(x)\right|\leq c_{j}\chi_{\frac{\epsilon}{3}, b+\epsilon}'(x),\quad \forall x\in \mathbb{R},\,j\in \mathbb{Z}^{+}.
\end{equation*} 
\item For $x\in (3\epsilon,\infty)$ 
 \begin{equation*}
 \chi_{\epsilon, b}(x)\geq\frac{1}{2}\frac{\epsilon}{b-3\epsilon}.
 \end{equation*}
\item For $x\in \mathbb{R}$
\begin{equation*}
 \chi_{\frac{\epsilon}{3},b+\epsilon}'(x)\leq \frac{\epsilon}{b-3\epsilon}.
\end{equation*}
\item Given $\epsilon>0$ and $b\geq 5\epsilon$ there exist $c_{1},c_{2}>0$ such that 
\[ \begin{array}{lcr}
\chi_{\epsilon,b}'(x)\leq c_{1}\chi_{\epsilon/3,b+\epsilon}'(x)\chi_{\epsilon/3,b+\epsilon}(x), & \\
\chi_{\epsilon,b}'(x)\leq c_{2} \chi_{\epsilon/5, \epsilon}(x). & \\
\end{array}\] 
\item For $\epsilon>0 $ given and $b\geq 5\epsilon,$ we define the function 
\begin{equation*}
\eta_{\epsilon,b}=\sqrt{\chi_{\epsilon, b}\chi_{\epsilon,b }'}.
\end{equation*}
\item ${\displaystyle \supp(\phi_{\epsilon,b}),\,\supp(\widetilde{\phi_{\epsilon, b}})\subset [\epsilon/4,b],}$
\item ${\displaystyle \phi_{\epsilon}(x)=\widetilde{\phi_{\epsilon, b}}(x)=1,\quad x\in [\epsilon/2,\epsilon],}$
\item ${\displaystyle \supp(\psi_{\epsilon})\subseteq(-\infty,\epsilon/2]},$
\item for $x\in \mathbb{R}$
\begin{equation*}
\chi_{\epsilon, b}(x)+\phi_{\epsilon, b}(x)+\psi_{\epsilon}(x)=1,
\end{equation*}
and 
\begin{equation*}
\chi_{\epsilon, b}^{2}(x)+\widetilde{\phi_{\epsilon, b}}^{2}(x)+\psi_{\epsilon}(x)=1.
\end{equation*}
\end{enumerate}
The family ${\displaystyle \{\chi_{\epsilon, b}: \epsilon>0,\, b\geq 5\epsilon\}}$ is constructed as follows:
  let  $\rho\in C^{\infty}_{0}(\mathbb{R}),\, \rho(x)\geq 0,$  even, with $\supp(\rho)\subseteq(-1,1)$ and $ \|\rho\|_{L^{1}}=1.$
  
  Then defining
  \begin{equation*}
  \nu_{\epsilon,b}(x)=\left\{
  \begin{array}{lll}
  0, & x\leq 2\epsilon ,\\
  \frac{x}{b-3\epsilon}-\frac{2\epsilon}{b-3\epsilon}, & 2\epsilon\leq x\leq b-\epsilon,\\
  1,& x\geq b-\epsilon,
  \end{array} 
  \right. 
  \end{equation*}
and
\begin{equation*}
\chi_{\epsilon,b}(x)=\rho_{\epsilon}*\nu_{\epsilon,b}(x)
\end{equation*}
where $\rho_{\epsilon}(x)=\epsilon^{-1}\rho(x/\epsilon).$

Now that it has been described all the required estimates and tools necessary, we present the proof of our main result.	
	\begin{proof}[Proof of Theorem \ref{A}]
	Since the argument is   translation invariant, without loss of generality we will consider  the case  $x_{0}=0.$
	
	First, we will describe the formal calculations  assuming as much as regularity as possible, later we provide the justification using a limiting process. 
	
	The proof will be established  by induction, however in every step of induction we will subdivide every  case in two steps, due to the non-local  nature of the operator involving the dispersive part in the  equation in   (\ref{eq7}).
\begin{flushleft}
	\textbf{\underline{Case $j=1$}}\\
\end{flushleft}
	\begin{flushleft}
		\textit{Step 1.}
	\end{flushleft}	
First we  apply one spatial derivative  to the equation in  (\ref{eq7}), after that we multiply by $\partial
_{x}u(x,t)\chi_{\epsilon, b}^{2}(x+vt),$ and finally we integrate in the $x-$variable 
to obtain the identity 
\begin{equation*}
\begin{split}
&\frac{1}{2}\frac{\mathrm{d}}{\mathrm{d}t}\int_{\mathbb{R}}(\partial_{x}u)^{2}\chi_{\epsilon, b}^{2}\,\mathrm{d}x\underbrace{-\frac{v}{2}\int_{\mathbb{R}}(\partial_{x}u)^{2}(\chi_{\epsilon, b}^{2})'\,\mathrm{d}x}_{A_{1}(t)}\underbrace{-\int_{\mathbb{R}}(\partial_{x}D_{x}^{\alpha+1}\partial_{x}u)\partial_{x}u\chi_{\epsilon, b}^{2}\,\mathrm{d}x}_{A_{2}(t)}\\
&+\underbrace{\int_{\mathbb{R}} \partial_{x}\left(u\partial_{x}u\right) \partial_{x}u\chi_{\epsilon, b}^{2}\,\mathrm{d}x}_{A_{3}(t)}=0.
\end{split}
\end{equation*} 
\S.1\quad Combining the  local theory  we obtain
the following 
\begin{equation*}
\begin{split}
\int_{0}^{T}|A_{1}(t)|\;\mathrm{d}t&\leq \frac{v}{2}\int_{0}^{T}\int_{\mathbb{R}} (\partial_{x}u)^{2}(\chi_{\epsilon, b}^{2})'\;\mathrm{d}x\;\mathrm{d}t
\lesssim \|u\|_{L_{T}^{\infty}H^{\frac{3-\alpha}{2}}_{x}}.
\end{split}
\end{equation*}	
\S.2\quad Integration by parts and Plancherel's identity allow us rewrite the term $A_{2}$ as follows 
	\begin{equation}\label{n1}
	\begin{split}
	A_{2}(t)=\frac{1}{2}\int_{\mathbb{R}} \partial_{x}u\left[D_{x}^{\alpha+1}\partial_{x};\chi_{\epsilon, b}^{2}\right]\partial_{x}u\,\mathrm{d}x=-\frac{1}{2}\int_{\mathbb{R}}\partial_{x}u\left[\mathcal{H}D_{x}^{\alpha+2};\chi_{\epsilon, b}^{2}\right]\partial_{x}u\,\mathrm{d}x.
	\end{split}
	\end{equation}
	Since $\alpha+2>1,$ we have by (\ref{eq8}) that the commutator ${\displaystyle \left[\mathcal{H}D_{x}^{\alpha+2};\chi_{\epsilon, b}^{2}\right]}$ can be decomposed as  
 \begin{equation}\label{n2}
 \left[\mathcal{H}D_{x}^{\alpha+2};\chi_{\epsilon, b}^{2}\right]=-\frac{1}{2}P_{n}(\alpha+2)+\frac{1}{2}\mathcal{H}P_{n}(\alpha+2)\mathcal{H}-R_{n}(\alpha+2)
 \end{equation}
 for some positive integer $n,$ that will be fixed later.
 
 Inserting (\ref{n2}) into (\ref{n1}) 
	\begin{equation*}\label{}
	\begin{split}
	A_{2}(t)&=\frac{1}{2}\int_{\mathbb{R}}\partial_{x}uR_{n}(\alpha+2)\partial_{x}u\,\mathrm{d}x+\frac{1}{4}\int_{\mathbb{R}}\partial_{x}uP_{n}(\alpha+2)\partial_{x}u\,\mathrm{d}x\\
	&\quad -\frac{1}{4}\int_{\mathbb{R}} \partial_{x}u\mathcal{H}P_{n}(\alpha+2)\mathcal{H}\partial
	_{x}u\,\mathrm{d}x\\
	&=A_{2,1}(t)+A_{2,2}(t)+A_{2,3}(t).
		\end{split}
	\end{equation*}
Now, we proceed to fix the value of $n$ present in the terms  $A_{2,1},A_{2,2}$ and $A_{2,3},$ according to a determinate condition. 

First, notice that    
\begin{equation*}
\begin{split}
A_{2,1}(t)&=\frac{1}{2}\int_{\mathbb{R}}D_{x}\mathcal{H}uR_{n}(\alpha+2)D_{x}\mathcal{H}u\,\mathrm{d}x
=\frac{1}{2}\int_{\mathbb{R}}\mathcal{H}uD_{x}\left\{R_{n}(\alpha+2)D_{x}\mathcal{H}u\right\}\,\mathrm{d}x\\
\end{split}
\end{equation*}
Then  we   fix $n$ such that 
$
2n+1\leq a+2\sigma\leq 2n+3,
$ that according to the case we are studying   ($j=1$), corresponds to  $a=\alpha+2$ and $\sigma=1.$ This   produces $n=1.$  

For this  $n$  in particular we have by Proposition \ref{propo2} that  $R_{1}(\alpha+2)$  maps    $L^{2}_{x}$ into $L^{2}_{x}.$

Hence,
\begin{equation*}
\begin{split}
A_{2,1}(t)&\lesssim \|\mathcal{H}u(t)\|_{L^{2}_{x}}^{2}\left\|\widehat{D_{x}^{4+\alpha}\chi_{\epsilon, b}^{2}}\right\|_{L^{1}_{\xi}}=c \|u_{0}\|_{L^{2}_{x}}^{2}\left\|\widehat{D_{x}^{4+\alpha}\chi_{\epsilon, b}^{2}}\right\|_{L^{1}_{\xi}},
\end{split}
\end{equation*}
which  after integrating in time yields
\begin{equation*}
\int_{0}^{T}|A_{2,1}(t)|\,\mathrm{d}t\lesssim\|u_{0}\|_{L^{2}_{x}}^{2}\sup_{0\leq t\leq T}\left\|\widehat{D_{x}^{4+\alpha}\chi_{\epsilon, b}^{2}}\right\|_{L^{1}_{\xi}}.
\end{equation*}
Next,
 we turn our attention  to $A_{2,2}.$ 
Replacing  $P_{1}(\alpha+2)$ into $A_{2,2}$ 
\begin{equation*}
\begin{split}
A_{2,2}(t)&
=\widetilde{c_{1}}\int_{\mathbb{R}} \left(D_{x}^{\frac{\alpha+1}{2}}\partial_{x}u\right)^{2}(\chi_{\epsilon, b}^{2})'\,\mathrm{d}x-\widetilde{c_{3}}\int_{\mathbb{R}}\left(D_{x}^{\frac{\alpha+1}{2}}\mathcal{H}u\right)^{2}(\chi_{\epsilon, b}^{2})'''\,\mathrm{d}x\\
&=A_{2,2,1}(t)+A_{2,2,2}(t).
\end{split}
\end{equation*}
We shall underline that $A_{2,2,1}(t)$ is positive, besides it  represents explicitly  the smoothing effect for the case $j=1$.

Regarding  $A_{2,2,2},$ the local theory combined with interpolation leads to  	
		\begin{equation}\label{eq8.17}
	\begin{split}
	\int_{0}^{T}|A_{2,2,2}(t)|\,\mathrm{d}t&\lesssim \left\|u\right\|_{L_{T}^{\infty}H^{\frac{3-\alpha}{2}}_{x}}.
	\end{split}
	\end{equation}
		After replacing  \eqref{eq105} into $A_{2,3}$ and using  the fact that Hilbert transform is skew-symmetric 
		
		\begin{equation*}
		\begin{split}
		A_{2,3}(t)&
		=\widetilde{c_{1}}\int_{\mathbb{R}}\left(D_{x}^{1+\frac{\alpha+1}{2}}u\right)^{2}(\chi_{\epsilon, b}^{2})'\;\mathrm{d}x-\widetilde{c_{3}}\int_{\mathbb{R}}\left(\mathcal{H}D_{x}^{\frac{\alpha+1}{2}}u\right)^{2}(\chi_{\epsilon, b}^{2})'''\,\mathrm{d}x\\
		&=A_{2,3,1}(t)+A_{2,3,2}(t).
		\end{split}
		\end{equation*}  
				Notice that the term $A_{2,3,1}$ is positive and represents the smoothing effect. In contrast, the term $A_{2,3,2}$  is estimated as we did with  $A_{2,2,2}$  in (\ref{eq8.17}). So, after integration in the time variable
		\begin{equation*}
		\int_{0}^{T}|A_{2,3,2}(t)|\,\mathrm{d}t\lesssim \left\|u\right\|_{L_{T}^{\infty}H_{x}^{\frac{3-\alpha}{2}}}.
		\end{equation*}
		Finally,  after apply integration by parts 
		\begin{equation*}
		\begin{split}
			A_{3}(t)
			&=\frac{1}{2}\int_{\mathbb{R}}\partial_{x}u(\partial_{x}u)^{2}\chi_{\epsilon, b}^{2}\;\mathrm{d}x-\frac{1}{2}\int_{\mathbb{R}}u(\partial_{x}u)^{2}(\chi_{\epsilon, b}^{2})'\;\mathrm{d}x\\
			&=A_{3,1}(t)+A_{3,2}(t).
				\end{split}
		\end{equation*}
		On  one hand,
		\begin{equation*}
		|A_{3,1}(t)|\lesssim\|\partial_{x}u(t)\|_{L^{\infty}_{x}}\int_{\mathbb{R}}(\partial_{x}u)^{2}\chi_{\epsilon, b}^{2}\;\mathrm{d}x,
				\end{equation*}
				where the integral expression on the right-hand side
is the quantity to be estimated by means of   Gronwall's inequality.

On the other hand, 
\begin{equation*}
|A_{3,2}(t)\lesssim \|u(t)\|_{L^{\infty}_{x}}\int_{0}^{T}(\partial_{x}u)^{2}(\chi_{\epsilon, b}^{2})'\;\mathrm{d}x.
\end{equation*}

By Sobolev embedding  we have  after integrating in time 
\begin{equation*}
\int_{0}^{T}|A_{3,2}(t)|\;\mathrm{d}t\lesssim\left(\sup_{0\leq t\leq T}\|u(t)\|_{H^{s(\alpha)+}_{x}}\right)\int_{0}^{T}\int_{\mathbb{R}}(\partial_{x}u)^{2}(\chi_{\epsilon, b}^{2})'\;\mathrm{d}x\;\mathrm{d}t\leq c.
\end{equation*}

  Since $\left\|\partial_{x}u\right\|_{L_{T}^{1}L_{x}^{\infty}}<\infty,$ then after gathering all estimates above and apply    Gronwall's inequality we obtain  
\begin{equation}\label{eq110}
\begin{split}
&\sup_{0\leq t\leq T}\left\|\partial_{x}u\chi_{\epsilon, b}\right\|_{L^{2}_{x}}^{2}
+\left\|D_{x}^{\frac{\alpha+1}{2}}\partial_{x}u\eta_{\epsilon,b}\right\|_{L^{2}_{T}L^{2}_{x}}^{2}+\left\|D_{x}^{1+\frac{\alpha+1}{2}}u\eta_{\epsilon,b}\right\|_{L^{2}_{T}L^{2}_{x}}^{2}\leq c^{*}_{1,1}
\end{split}
\end{equation}
where $c^{*}_{1,1}=c^{*}_{1,1}\left(\alpha;\epsilon;T; \|u_{0}\|_{H_{x}^{\frac{3-\alpha}{2}}};\|\partial_{x}u_{0}\chi_{\epsilon, b}\|_{L_{x}^{2}}\right)>0,$  for any $\epsilon>0,\, b\geq 5\epsilon$ and $v\geq 0.$

This  estimate finish the step 1 corresponding to the case $j=1.$

  The local  smoothing effect obtained  above is just $\frac{1+\alpha}{2}$  derivative (see \cite{ILP2}). So, the  iterative argument is carried out  in two steps, the first step for positive integers $m$  and the second one for $m+\frac{1-\alpha}{2}.$

	\begin{flushleft}		
		\textit{Step 2.}
\end{flushleft}
 After apply the operator $D_{x}^{\frac{1-\alpha}{2}}\partial_{x}$ to  the equation in (\ref{eq7})  and 
 multiply the resulting  by $D_{x}^{\frac{1-\alpha}{2}}\partial_{x}u\chi_{\epsilon, b}^{2}(x+vt)$ one gets 
\begin{equation*}
\begin{split}
&D_{x}^{\frac{1-\alpha}{2}}\partial_{x}\partial_{t}uD_{x}^{\frac{1-\alpha}{2}}\partial_{x}u\chi_{\epsilon, b}^{2}-D_{x}^{\frac{1-\alpha}{2}}\partial_{x}D_{x}^{1+\alpha}\partial_{x}uD_{x}^{\frac{1-\alpha}{2}}\partial_{x}u\chi_{\epsilon, b}^{2}\\
&+D_{x}^{\frac{1-\alpha}{2}}\partial_{x}(u\partial_{x}u)D_{x}^{\frac{1-\alpha}{2}}\partial_{x}u\chi_{\epsilon, b}^{2}=0,
\end{split}
\end{equation*}
which after integrate in the spatial variable it  becomes 
	\begin{equation*}
	\begin{split}
	&\frac{1}{2}\frac{\mathrm{d}}{\mathrm{d}t}\int_{\mathbb{R}} (D_{x}^{\frac{1-\alpha}{2}}\partial_{x}u)^{2}\chi_{\epsilon, b}^{2}\;\mathrm{d}x\underbrace{-v\int_{\mathbb{R}}(D_{x}^{\frac{1-\alpha}{2}}\partial_{x}u)^{2}(\chi_{\epsilon, b}^{2})^{'}\;\mathrm{d}x}_{A_{1}(t)}\\
	&\underbrace{-\int_{\mathbb{R}}(D_{x}^{\frac{1-\alpha}{2}}\partial_{x}D_{x}^{1+\alpha}\partial_{x}u)\,D_{x}^{\frac{1-\alpha}{2}}\partial_{x}u\,\chi_{\epsilon, b}^{2}\;\mathrm{d}x}_{A_{2}(t)}\underbrace{+\int_{\mathbb{R}}D_{x}^{\frac{1-\alpha}{2}}\partial_{x}(u\partial_{x}u)D_{x}^{\frac{1-\alpha}{2}}\partial_{x}u\, \chi_{\epsilon, b}^{2}\;\mathrm{d}x}_{A_{3}(t)}
	=0.
		\end{split}
	\end{equation*}
\S.1\quad  	First observe that  by the local theory 
	\begin{equation*}
	\begin{split}
	\int_{0}^{T}|A_{1}(t)|\,\mathrm{d}t&\leq |v|\int_{0}^{T}\int_{\mathbb{R}}\left(D_{x}^{\frac{1-\alpha}{2}}\partial_{x}u\right)^{2}(\chi_{\epsilon, b}^{2})'\,\mathrm{d}x\,\mathrm{d}t\lesssim \|u\|_{L_{T}^{\infty}H_{x}^{\frac{3-\alpha}{2}}}.
	\end{split}
	\end{equation*}

\S.2\quad Concerning to the  term $A_{2}$ integration by parts and  Plancherel's identity   yields
\begin{equation}\label{eq8.15}
\begin{split}
A_{2}(t)
&=-\frac{1}{2}\int_{\mathbb{R}}D_{x}^{\frac{1-\alpha}{2}+1}\mathcal{H}u\left[\mathcal{H}D_{x}^{2+\alpha}; \chi_{\epsilon, b}^{2}\right] D_{x}^{\frac{1-\alpha}{2}+1}\mathcal{H}u\,\mathrm{d}x.
\end{split}
\end{equation}
	Since $2+\alpha>1,$ we have by (\ref{eq8}) that the commutator ${\displaystyle \left[\mathcal{H}D_{x}^{\alpha+2};\chi_{\epsilon, b}^{2}\right]}$ can be decomposed as 
\begin{equation}\label{eq8.14}
\left[\mathcal{H}D_{x}^{\alpha+2};\chi_{\epsilon, b}^{2}\right]+\frac{1}{2}P_{n}(\alpha+2)+R_{n}(\alpha+2)=\frac{1}{2}\mathcal{H}P_{n}(\alpha+2)\mathcal{H}
\end{equation}
for some positive integer $n$  that as in the previous cases it will be fixed suitably.

Replacing (\ref{eq8.14}) into (\ref{eq8.15})  
\begin{equation*}
\begin{split}
A_{2}(t)
&=\frac{1}{2}\int_{\mathbb{R}}D_{x}^{\frac{3-\alpha}{2}}\mathcal{H}u\left(R_{n}(\alpha+2)D_{x}^{\frac{3-\alpha}{2}}\mathcal{H}u\right)\mathrm{d}x\\
&\quad +
\frac{1}{4}\int_{\mathbb{R}}D_{x}^{\frac{3-\alpha}{2}}\mathcal{H}u\left(P_{n}(\alpha+2)D_{x}^{\frac{3-\alpha}{2}}\mathcal{H}u\right)\mathrm{d}x\\
&\quad-\frac{1}{4}\int_{\mathbb{R}}D_{x}^{\frac{3-\alpha}{2}}\mathcal{H}u\left(\mathcal{H}P_{n}(\alpha+2)\mathcal{H}D_{x}^{\frac{3-\alpha}{2}}\mathcal{H}u\right)\mathrm{d}x\\
&=A_{2,1}(t)+A_{2,2}(t)+A_{2,3}(t).
\end{split}
\end{equation*}
Fixing the value of $n$ present in the terms $A_{2,1},A_{2,2}$ and $A_{2,3}$  requires an argument  almost  similar to that one  used in step 1.
First, we deal with  $A_{2,1}$ where  a simple computation produces
\begin{equation*}
\begin{split}
A_{2,1}(t)&=\frac{1}{2}\int_{\mathbb{R}}\mathcal{H}uD_{x}^{\frac{3-\alpha}{2}}\left\{R_{n}(\alpha+2)D_{x}^{\frac{3-\alpha}{2}}\mathcal{H}u\right\}\,\mathrm{d}x.
\end{split}
\end{equation*}
We fix $n\in\mathbb{Z^{+}}$ in such a way  
\begin{equation*}
2n+1\leq a+2\sigma\leq 2n+3
\end{equation*} 
where $a=\alpha+2$ and $\sigma=\frac{3-\alpha}{2}$ in order to obtain obtain $n=1$ or $n=2.$ For the sake of  simplicity we choose $n=1.$    
 
   Hence, by construction $R_{1}(\alpha+2)$ satisfies the  hypothesis of 
  Proposition \ref{propo2}, and 
\begin{equation*}
\begin{split}
|A_{2,1}(t)|&\lesssim\|\mathcal{H}u(t)\|_{L^{2}_{x}}\left\|\widehat{D_{x}^{5}(\chi_{\epsilon, b}^{2})}\right\|_{L^{1}_{\xi}}
\lesssim\|u_{0}\|_{L^{2}_{x}}\left\|\widehat{D_{x}^{5}(\chi_{\epsilon, b}^{2})}\right\|_{L^{1}_{\xi}}.
\end{split}
\end{equation*} 	
Thus
\begin{equation*}
\int_{0}^{T}|A_{2,1}(t)|\,\mathrm{d}t\lesssim\|u_{0}\|_{L^{2}_{x}}\sup_{0\leq t\leq  T}\left\|\widehat{D_{x}^{5}(\chi_{\epsilon, b}^{2})}\right\|_{L^{1}_{\xi}}.
\end{equation*}
Next, after replacing $P_{1}(\alpha+2)$ in $A_{2,2}$
\begin{equation*}
\begin{split}
A_{2,2}(t)
 &=\left(\frac{\alpha+2}{4}\right)\int_{\mathbb{R}}\left(\mathcal{H}\partial_{x}^{2}u\right)^{2}(\chi_{\epsilon, b}^{2})^{'}\;\mathrm{d}x-c_{3}\left(\frac{\alpha+2}{16}\right)\int_{\mathbb{R}}(\partial_{x}u)^{2}(\chi_{\epsilon, b}^{2})^{'''}\,\mathrm{d}x\\
 &=A_{2,2,1}(t)+A_{2,2,2}(t).
 \end{split}
\end{equation*}	
The smoothing effect corresponds to the term $A_{2,2,1}$ and it will be bounded after integrating in time. In contrast,  to bound $A_{2,2,2}$ is required only the local theory,  in fact 
 \begin{equation*}
 \begin{split}
 \int_{0}^{T}|A_{2,2,2}(t)|\,\mathrm{d}t&\lesssim \|u\|_{L_{T}^{\infty}H_{x}^{\frac{3-\alpha}{2}}}. 
 \end{split}
 \end{equation*}
Concerning the term $A_{2,3}$ we have after replacing $P_{1}(\alpha+2)$   
 and using the  properties of the  Hilbert transform that 
	\begin{equation*}
	\begin{split}
	A_{2,3}(t) 
	&=\left(\frac{\alpha+2}{4}\right)\int_{\mathbb{R}} (\partial_{x}^{2}u)^{2}(\chi_{\epsilon, b}^{2})'\,\mathrm{d}x-c_{3}\left(\frac{\alpha+2}{16}\right)\int_{\mathbb{R}}(D_{x}u)^{2}(\chi_{\epsilon, b}^{2})'''\,\mathrm{d}x\\
	&=A_{2,3,1}(t)+A_{2,3,2}(t).
	\end{split}
	\end{equation*}	
	As before, $A_{2,3,1}\geq0$ and represents the smoothing effect.
	Besides, the local theory and interpolation yields
\begin{equation*}
\int_{0}^{T}|A_{2,3,2}(t)|\,\mathrm{d}t\lesssim\|u\|_{L_{T}^{\infty}H_{x}^{\frac{3-\alpha}{2}}}.
\end{equation*} 
	
\S.3\quad  It only remains to handle the term $A_{3}.$

Since 	
	\begin{equation}\label{eqA1}
	\begin{split}
	&D_{x}^{\frac{1-\alpha}{2}}\partial_{x}(u\partial_{x}u)\chi_{\epsilon, b}\\ 
	&=-\frac{1}{2}\left[D_{x}^{\frac{1-\alpha}{2}}\partial_{x};\chi_{\epsilon, b}\right]\partial_{x}((\chi_{\epsilon,b}u)^{2}+(\widetilde{\phi_{\epsilon,b}}u)^{2}+(\psi_{\epsilon}u^{2}))\\
	&\quad +\left[D_{x}^{\frac{1-\alpha}{2}}\partial_{x}; u\chi_{\epsilon, b}\right]\partial_{x}((\chi_{\epsilon, b}u)+ (u\phi_{\epsilon,b})+(u\psi_{\epsilon}))
	+ u\chi_{\epsilon, b}D_{x}^{\frac{1-\alpha}{2}}\partial_{x}^{2}u\\
	&=\widetilde{A_{3,1}}(t)+\widetilde{A_{3,2}}(t)+\widetilde{A_{3,3}}(t)+\widetilde{A_{3,4}}(t)+\widetilde{A_{3,5}}(t)+\widetilde{A_{3,6}}(t)+\widetilde{A_{3,7}}(t).
	\end{split}
	\end{equation}	
First, we  rewrite   $\widetilde{A_{3,1}}$ as   follows
	 \begin{equation*}
	 \begin{split}
	 	\widetilde{A_{3,1}}(t)=c_{\alpha}\mathcal{H} \left[ D_{x}^{1+\frac{1-\alpha}{2}};\chi_{\epsilon,b}\right]\partial_{x}((u\chi_{\epsilon,b})^{2})+c_{\alpha}\left[\mathcal{H};\chi_{\epsilon,b}\right]D_{x}^{1+\frac{1-\alpha}{2}}\partial_{x}\left(\left(\chi_{\epsilon,b}u\right)^{2}\right),
	 \end{split}
	 \end{equation*} 
	 where $c_{\alpha}$ denotes a non-null  constant.
	  Next, combining \eqref{eq5}, \eqref{kpdl} and Lemma \ref{lema2}   one gets  
	 
	 \begin{equation*}
	 \begin{split}
	 \|\widetilde{A_{3,1}}(t)\|_{L^{2}_{x}}& 
	 \lesssim  \left\|D_{x}^{1+\frac{1-\alpha}{2}}(u\chi_{\epsilon,b})\right\|_{L^{2}_{x}}\|u\|_{L^{\infty}_{x}}+\|u_{0}\|_{L_{x}^{2}}\|u\|_{L^{\infty}_{x}}
	 \end{split}
	 \end{equation*}
	 and  
	 \begin{equation*}
	 \begin{split}
	 \|\widetilde{A_{3,2}}(t)\|_{L^{2}_{x}}& 
	 \lesssim   \left\|D_{x}^{1+\frac{1-\alpha}{2}}(u\widetilde{\phi_{\epsilon,b}})\right\|_{L^{2}_{x}}\|u\|_{L^{\infty}_{x}}+\|u_{0}\|_{L_{x}^{2}}\|u\|_{L^{\infty}_{x}}.
	 \end{split}
	 \end{equation*}
	 Next, we recall that  by construction
	 \begin{equation*}
	 \dist\left(\supp\left( \chi_{\epsilon,b }\right), \supp\left(\psi_{\epsilon}\right)\right)\geq \frac{\epsilon}{2},
	 \end{equation*} 
	 so, by  Lemma \ref{lemma1}
	 	 \begin{equation*}
	 	 \begin{split}
	 	 \|\widetilde{A_{3,3}}(t)\|_{L^{2}_{x}}&=\left\|\left[D_{x}^{\frac{1-\alpha}{2}}\partial_{x};\chi_{\epsilon}\right]\partial_{x}(\psi_{\epsilon}u^{2})\right\|_{L^{2}_{x}}
	 	 \lesssim\|u_{0}\|_{L^{2}_{x}}\|u\|_{L^{\infty}_{x}}.
	 	 	 	 \end{split}
	 	 \end{equation*}
	 Rewriting  
	 \begin{equation*}
	 \begin{split}
	 \widetilde{A_{3,4}}(t)=c\mathcal{H}\left[D_{x}^{1+\frac{1-\alpha}{2}};u\chi_{\epsilon,b}\right]\partial_{x}(u\chi_{\epsilon,b})-c\left[\mathcal{H};u\chi_{\epsilon,b}\right]\partial_{x}D_{x}^{1+\frac{1-\alpha}{2}}(u\chi_{\epsilon,b})
	 \end{split}
	 \end{equation*}
	 for some non-null constant $c.$
	
	 Thus, by the commutator estimates  \eqref{eq31} and   \eqref{dlkp} 
	 \begin{equation*}
	 \begin{split}
	 \|\widetilde{A_{3,4}}(t)\|_{L^{2}_{x}} 
	 &\lesssim \|\partial_{x}(u\chi_{\epsilon,b})\|_{L^{\infty}_{x}}\left\|D_{x}^{1+\frac{1-\alpha}{2}}(u\chi_{\epsilon,b})\right\|_{L^{2}_{x}}.
	 \end{split}
	 \end{equation*}
	 	 Applying the same procedure to $\widetilde{A_{3,5}}$ yields  
	\begin{equation*}
	\begin{split}
	&\|\widetilde{A_{3,5}}(t)\|_{L^{2}_{x}}\\ 
	&\lesssim \|\partial_{x}(u\chi_{\epsilon,b})\|_{L^{\infty}_{x}}\left\|D_{x}^{1+\frac{1-\alpha}{2}}(u\phi_{\epsilon,b})\right\|_{L^{2}_{x}}+\|\partial_{x}(u\phi_{\epsilon,b})\|_{L^{\infty}_{x}}\left\|D_{x}^{1+\frac{1-\alpha}{2}}(u\chi_{\epsilon,b})\right\|_{L^{2}_{x}}.
	\end{split}
	\end{equation*}
Since the supports of $\chi_{\epsilon, b}$ and $\psi_{\epsilon}$ are separated we obtain by  Lemma \ref{lemma1} 
	 \begin{equation*}
	 \begin{split}
	 \|\widetilde{A_{3,6}}(t)\|_{L^{2}_{x}}&=\left\|u\chi_{\epsilon,b}\partial_{x}^{2}D_{x}^{1+\frac{1-\alpha}{2}}(u\psi_{\epsilon})\right\|_{L^{2}_{x}}
\lesssim\left\|u_{0}\right\|_{L^{2}_{x}}\left\|u\right\|_{L^{\infty}_{x}}.
	 \end{split}
	 \end{equation*}
	 To finish with the estimates above we use the relation
	 \begin{equation*}
	 \chi_{\epsilon,b }(x)+\phi_{\epsilon, b}(x)+\psi_{\epsilon}(x)=1\quad \forall x\in \mathbb{R}.
	 \end{equation*}
	 Then
	 \begin{equation*}
	 \begin{split}
	 D_{x}^{1+\frac{1-\alpha}{2}}(u\chi_{\epsilon, b})&=D_{x}^{1+\frac{1-\alpha}{2}}u \chi_{\epsilon, b}+ \left[D_{x}^{1+\frac{1-\alpha}{2}}; \chi_{\epsilon, b}\right](u\chi_{\epsilon, b}+u\phi_{\epsilon, b}+u\psi_{\epsilon})\\
	 &=I_{1}+I_{2}+I_{3}+I_{4}.
	 \end{split}
	 	 \end{equation*}
	 	 Notice that $\|I_{1}\|_{L_{x}^{2}}$ is the quantity to estimate. In contrast,  $\|I_{2}\|_{L_{x}^{2}}$ and $\|I_{3}\|_{L_{x}^{2}}$  can be handled by Lemma \ref{dlkp}  combined with  the local theory. Meanwhile  $I_{3}$ can be bounded  by  using  Lemma \ref{lemma1}.
	 	 
	 	 We notice that the gain of regularity obtained  in the step 1 implies that \\ $\|D_{x}^{1+\frac{1+\alpha}{2}}\left(u\phi_{\epsilon,b}\right)\|_{L_{x}^{2}}<\infty.$ To show this  we use  Theorem \ref{thm11}  and  H\"{o}lder's inequality   as follows
	 	  \begin{equation}\label{eq72}
	 	 \begin{split}
	 	 \left\|D_{x}^{1+\frac{1+\alpha}{2}}(u\phi_{\epsilon,b})\right\|_{L^{2}_{x}}
	 	 &\lesssim  \|u_{0}\|_{L_{x}^{2}}+\|u\|_{L_{x}^{\infty}}+ \left\|\mathbb{1}_{[\epsilon/4,b]} D_{x}^{1+\frac{1+\alpha}{2}}u\right\|_{L_{x}^{2}}+\left\|\mathbb{1}_{[\epsilon/4,b]}\mathcal{H}D_{x}^{\frac{1+\alpha}{2}}u\right\|_{L_{x}^{2}}.
	 	 \end{split}
	 	 \end{equation}
	 	 The second term on the right hand side after integrate in time  is  controlled by using Sobolev's embedding. Meanwhile, the third term can be handled  after integrate in time and use  (\ref{eq110})  with $(\epsilon,b)=(\epsilon/24,b+7\epsilon/24).$  
	 	 
	 	  The fourth  term in the right hand side can be bounded  combining  the local theory and interpolation. 
	 	  
	 	  Hence, after integration in time 
	 	  \begin{equation}\label{eq73}
	 	  \left\|D_{x}^{1+\frac{1+\alpha}{2}}(u\phi_{\epsilon,b})\right\|_{L^{2}_{T}L^{2}_{x}}<\infty,
	 	  \end{equation}
	 	which clearly implies ${ \displaystyle \left\|D_{x}^{1+\frac{1-\alpha}{2}}(u\phi_{\epsilon,b})\right\|_{L^{2}_{T}L^{2}_{x}}<\infty},$ as was required.
	 	Analogously can be handled $\left\|D_{x}^{1+\frac{1-\alpha}{2}}(u\widetilde{\phi_{\epsilon,b}})\right\|_{L_{x}^{2}}.$
	 	  
	 	Finally, 
	 	\begin{equation*}
	 	\begin{split}
	 	A_{3,7}(t)
	 	&=-\frac{1}{2}\int_{\mathbb{R}}\partial_{x}u\chi_{\epsilon, b}^{2}\left(D_{x}^{\frac{1-\alpha}{2}}\partial_{x}u\right)^{2}\mathrm{d}x-\frac{1}{2}\int_{\mathbb{R}}u(\chi_{\epsilon, b}^{2})'\left(D_{x}^{\frac{1-\alpha}{2}}\partial_{x}u\right)^{2}\mathrm{d}x\\
	 	&=A_{3,7,1}(t)+A_{3,7,2}(t).
	 	\end{split}
	 	\end{equation*}
	 	Since 
	 	\begin{equation*}
	 	|A_{3,7,1}(t)|\lesssim\|\partial_{x}u(t)\|_{L^{\infty}_{x}}\int_{\mathbb{R}}\left(D_{x}^{\frac{1-\alpha}{2}}\partial_{x}u\right)^{2}\chi_{\epsilon, b}^{2}\,\mathrm{d}x,
	 	\end{equation*}
	 	being the last integral  the quantity to be estimated by means of  Gronwall's  inequality, and by the local theory $\|\partial_{x}u\|_{L^{1}_{T}L_{x}^{\infty}}<\infty.$
	 
	 Sobolev's embedding led us to  
	 		 	\begin{equation*}
	 	\begin{split}
	 	\int_{0}^{T}|A_{3,7,2}(t)|\,\mathrm{d}t
	 	&\lesssim \left( \sup_{0\leq t\leq T}\|u(t)\|_{H^{s(\alpha)+}_{x}}\right)\int_{0}^{T}\int_{\mathbb{R}}\chi_{\epsilon, b} \,\chi_{\epsilon, b}'\left(D_{x}^{\frac{1-\alpha}{2}}\partial_{x}u\right)^{2}\mathrm{d}x\mathrm{d}t.
	 	\end{split}
	 	\end{equation*}
	 	Gathering all the information corresponding to this step combined  with  Gronwall's inequality yields
	 	
	\begin{equation}\label{eq120}
	\begin{split}
	&\sup_{0\leq t\leq T}\left\|D_{x}^{\frac{1-\alpha}{2}}\partial_{x}u\chi_{\epsilon, b}\right\|_{L^{2}_{x}}^{2}+\left\|\partial_{x}^{2}u\eta_{\epsilon,b}\right\|_{L^{2}_{T}L^{2}_{x}}^{2}+\left\|\mathcal{H}\partial_{x}u\eta_{\epsilon, b}\right\|_{L^{2}_{T}L^{2}_{x}}^{2}\leq c^{*}_{1,2}
	\end{split}
		\end{equation}
	with $c^{*}_{1,2}=c^{*}_{1,2}\left(\alpha;\epsilon;T;v; \|u_{0}\|_{H_{x}^{\frac{3-\alpha}{2}}};\left\|D_{x}^{\frac{1-\alpha}{2}}\partial_{x}u_{0}\chi_{\epsilon, b}\right\|_{L_{x}^{2}}\right)$ for any $\epsilon>0,\,b \geq 5\epsilon$ and $v\geq0.$
	
	This finishes the step two corresponding to the  case $j=1$ in the induction process.
	
Next, we present the case $j=2,$ to show how we proceed in the case $j$ even.
	\begin{flushleft}
			\textbf{\underline{Case $j=2$}}
	\end{flushleft}

	\begin{flushleft}		
	\textit{Step 1.}
	\end{flushleft}
	First we  apply two spatial derivatives  to the equation in  (\ref{eq7}), after that we multiply by $\partial
	_{x}^{2}u(x,t)\chi_{\epsilon, b}^{2}(x+vt),$ and finally we integrate in the $x-$variable 
	to obtain the identity 
	\begin{equation}\label{eqk1}
	\begin{split}
	&\frac{1}{2}\frac{\mathrm{d}}{\mathrm{d}t}\int_{\mathbb{R}}\left(\partial_{x}^{2}u\right)^{2}\chi_{\epsilon, b}^{2}\,\mathrm{d}x-\underbrace{\frac{v}{2}\int_{\mathbb{R}}
\left(\partial_{x}^{2}u\right)^{2}(\chi_{\epsilon,b}^{2})'\,\mathrm{d}x}_{ A_{1}(t)}\underbrace{-\int_{\mathbb{R}}\left(\partial_{x}^{2}D_{x}^{\alpha+1}\partial_{x}u\right)\, \partial_{x}^{2}u \chi_{\epsilon, b}^{2}\,\mathrm{d}x}_{A_{2}(t)} \\
&+\underbrace{\int_{\mathbb{R}} \partial_{x}^{2}(u\partial_{x}u)\partial_{x}^{2}u\chi_{\epsilon,b}^{2}\,\mathrm{d}x}_{ A_{3}(t)}=0.\\
\end{split}
	\end{equation}
Similarly as was done in the previous steps we first proceed to  estimate   $A_{1}.$

\S.1  
  By (\ref{eq120}) it follows that  
  \begin{equation}\label{m1}
  \begin{split}
    \int_{0}^{T}|A_{1}(t)|\,\mathrm{d}t & \leq \int_{0}^{T}\int_{\mathbb{R}}(\partial_{x}^{2}u)^{2}(\chi_{\epsilon, b}^{2})'\,\mathrm{d}x\,\mathrm{d}t\leq c^{*}_{1,2}.\\
  \end{split}
  \end{equation} 
\S2. 
To extract information  from the term  $A_{2}$   we  use integration by parts and Plancherel's identity to obtain 
	\begin{equation}\label{eq8.10}
	\begin{split}
	A_{2}(t)=\frac{1}{2}\int_{\mathbb{R}} \partial_{x}^{2}u\left[D_{x}^{\alpha+1}\partial_{x};\chi_{\epsilon, b}^{2}\right]\partial_{x}^{2}u\,\mathrm{d}x=-\frac{1}{2}\int_{\mathbb{R}}\partial_{x}^{2}u\left[\mathcal{H}D_{x}^{\alpha+2};\chi_{\epsilon, b}^{2}\right]\partial_{x}^{2}u\,\mathrm{d}x.
	\end{split}
	\end{equation}

Although   this stage of the process is related to the  one performed in step 1 (for $j=1$),  we will use again    the commutator expansion in \eqref{eq8}, taking into account  in this case that $a=\alpha+2>1$ and $n$ is a non-negative integer whose value will be fixed later. 
  
Then,
\begin{equation*}
\begin{split}
A_{2}(t)&= \frac{1}{2}\int_{\mathbb{R}}\partial_{x}^{2}u\,R_{n}(s+2)\partial_{x}^{2}u\,\mathrm{d}x +\frac{1}{4}\int_{\mathbb{R}}\partial_{x}^{2}u\,P_{n}(s+2)\partial_{x}^{2}u\,\mathrm{d}x\\
&\quad -\frac{1}{4}\int_{\mathbb{R}}\partial_{x}^{2}u\,\mathcal{H}P_{n}(s+2)\mathcal{H}\partial_{x}^{2}u\,\mathrm{d}x\\
&=A_{2,1}(t)+A_{2,2}(t)+A_{2,3}(t).
\end{split}
\end{equation*}
Essentially, the key term which allows us  to fix the value of  $n,$ correspond to $A_{2,1}.$ Indeed, after some integration by parts  
\begin{equation*}
\begin{split}
A_{2,1}(t)&=\frac{1}{2}\int_{\mathbb{R}}u\,\partial_{x}^{2}R_{n}(\alpha+2)\partial_{x}^{2}u\,\mathrm{d}x\\
          &=\frac{1}{2}\int_{\mathbb{R}}u\,\partial_{x}^{2}\left\{R_{n}(\alpha+2)\partial_{x}^{2}u\right\}\,\mathrm{d}x.
\end{split}
\end{equation*}
We fix $n$ such that it satisfies  
\begin{equation*}
2n+1\leq a+2\sigma\leq 2n+3.
\end{equation*}
 In this case with $a=\alpha+2>1$  and  $\sigma=2,$  we  obtain  $n=2.$ 

Hence by construction the  Proposition \ref{propo2}  guarantees that $D_{x}^{2}R_{2}(\alpha+2)D_{x}^{2}$ is bounded in $L^{2}_{x}.$ 

Thereby
\begin{equation*}
\begin{split}
|A_{2,1}(t)|&\lesssim \|u(t)\|_{L^{2}_{x}}^{2}\left\|D_{x}^{2}R_{2}(\alpha+2)D_{x}^{2}u\right\|_{L^{2}_{x}}\leq c\|u_{0}\|_{L^{2}_{x}}^{2}\left\|\widehat{D_{x}^{\alpha+6}\left(\chi_{\epsilon, b}^{2}\right)}\right\|_{L_{\xi}^{1}}.
\end{split}
\end{equation*}
Since  we  fixed  $n=2,$ we proceed to handle the contribution coming from   $A_{2,2}$ and $A_{2,3}.$

Next,
\begin{equation*}
\begin{split}
A_{2,2}(t) 
&=\widetilde{c_{1}}\int_{\mathbb{R}} \left(D_{x}^{\frac{\alpha+1}{2}}\partial_{x}^{2}u\right)^{2}(\chi_{\epsilon, b}^{2})'\mathrm{d}x-\widetilde{c_{3}}\int_{\mathbb{R}}\left(D_{x}^{1+\frac{\alpha+1}{2}}u\right)^{2}(\chi_{\epsilon, b}^{2})^{(3)}\mathrm{d}x\\
&\quad +c_{5}\left(\frac{\alpha+2}{64}\right)\int_{\mathbb{R}} (D_{x}^{\frac{\alpha+1}{2}}u)^{2}(\chi_{\epsilon, b}^{2})^{(5)}\mathrm{d}x\\
&=A_{2,2,1}(t)+A_{2,2,2}(t)+A_{2,2,3}(t).
\end{split}
\end{equation*}
	Notice    that $A_{2,2,1}\geq 0$  represents the smoothing effect. 

We recall that 
\begin{equation*}
\left|\chi_{\epsilon,b}^{(j)}(x)\right|\lesssim \chi_{\frac{\epsilon}{3},b+\epsilon}'(x) \quad \forall x\in \mathbb{R},\,j\in \mathbb{Z}^{+},
\end{equation*}
then 
\begin{equation*}
\begin{split}
\int_{0}^{T}|A_{2,2,2}(t)|\,\mathrm{d}t& \lesssim \int_{0}^{T}\int_{\mathbb{R}} \left(D_{x}^{1+\frac{1+\alpha}{2}}u\right)^{2}\chi_{\epsilon/3,b+\epsilon}'\,\mathrm{d}x\,\mathrm{d}t.\\ 
\end{split}
\end{equation*}
Taking $(\epsilon,b)=(\epsilon/9, b+10\epsilon/9)$ in (\ref{eq110}) combined with  the properties of the cutoff function  we have  
\begin{equation*}
\int_{0}^{T}|A_{2,2,2}(t)|\,\mathrm{d}t \lesssim c_{1,1}^{*}.
\end{equation*}
To finish the terms that   make  $A_{2}$  we  proceed to estimate $A_{2,2,3}.$  

 As usual  the low regularity is controlled by interpolation and the local theory. Therefore
\begin{equation*}
\begin{split}
\int_{0}^{T}|A_{2,2,3}(t)|\,\mathrm{d}t&\lesssim \|u\|_{L_{T}^{\infty}H_{x}^{\frac{3-\alpha}{2}}}.
\end{split}
\end{equation*} 
Next,
\begin{equation*}
\begin{split}
 A_{2,3}(t)  
&=\widetilde{c_{1}}\int_{\mathbb{R}}\left(\mathcal{H}D_{x}^{\frac{\alpha+1}{2}}\partial_{x}^{2}u\right)^{2}(\chi_{\epsilon, b}^{2})'\mathrm{d}x
-\widetilde{c_{3}}\int_{\mathbb{R}}\left(D_{x}^{\frac{\alpha+1}{2}}\partial_{x}u\right)^{2}(\chi_{\epsilon, b}^{2})^{(3)}\mathrm{d}x\\
&\quad+\left(\frac{\alpha+2}{64}\right)c_{5}\int_{\mathbb{R}}\left(D_{x}^{\frac{\alpha+1}{2}}\mathcal{H}u\right)^{2}(\chi_{\epsilon, b}^{2})^{(5)}\mathrm{d}x\\
&=A_{2,3,1}(t)+A_{2,3,2}(t)+A_{2,3,3}(t).
\end{split}
\end{equation*}
 $A_{2,3,1}$ is positive and it will provide  the smoothing effect after   being integrated in time.

The terms $A_{2,3,2}$ and $A_{2,3,3}$  can be handled exactly in the same way that were treated  $A_{2,2,2}$ and $A_{2,2,3}$ respectively, so we will omit the proof.
\S.3   	Finally,  
	\begin{equation*}
	\begin{split}
	A_{3}(t)
	&=\frac{5}{2}\int_{\mathbb{R}} \partial_{x}u(\partial_{x}^{2}u)^{2}\chi_{\epsilon,b}^{2}\, \mathrm{d}x
	-\frac{1}{2}\int_{\mathbb{R}} u(\partial_{x}^{2}u)^{2}(\chi_{\epsilon, b}^{2})'\,\mathrm{d}x\\
	&=A_{3,1}(t)+A_{3,2}(t).
		\end{split}
	\end{equation*}	
First,
\begin{equation}\label{eq10}
\begin{split}
|A_{3,1}(t)|\lesssim \|\partial_{x}u(t)\|_{L^{\infty}_{x}}\int_{\mathbb{R}}(\partial_{x}^{2}u)^{2} \chi_{\epsilon, b}^{2}\,\mathrm{d}x,
\end{split}
\end{equation}
by the local theory 	
   $\partial_{x}u\in L^{1}\left([0,T]:L_{x}^{\infty}(\mathbb{R})\right)$  (see Theorem \ref{C}-(b)); and the integral expression  is  the quantity we want estimate.
 
 Next,
 \begin{equation}\label{eq11}
 |A_{3,2}(t)|\lesssim \|u(t)\|_{L^{\infty}_{x}}\int_{\mathbb{R}} (\partial_{x}^{2}u)^{2}(\chi_{\epsilon, b}^{2})'\,\mathrm{d}x.
 \end{equation}
	After apply the  Sobolev embedding  and  integrate in  the time variable we obtain
	\begin{equation*}
	\begin{split}
	\int_{0}^{T}|A_{3,2}(t)|\,\mathrm{d}t &\lesssim \left(\sup_{0\leq t\leq T}\|u(t)\|_{{H}_{x}^{s(\alpha)+}}\right)\int_{0}^{T}\int_{\mathbb{R}}(\partial_{x}^{2}u)^{2}(\chi_{\epsilon, b}^{2})'\,\mathrm{d}x\,\mathrm{d}t,
	\end{split}	
	\end{equation*}
	and the integral term in the right hand side  was estimated   previously in  (\ref{m1}).
	
Thus, after grouping all the terms and apply  Gronwall's inequality  we obtain
	\begin{equation}\label{eq121}
	\begin{split}
		&\sup_{0\leq t\leq T}\left\|\partial_{x}^{2}u\chi_{\epsilon, b}\right\|_{L^{2}_{x}}^{2}
		+\left\|D_{x}^{\frac{\alpha+1}{2}}\partial_{x}^{2}u\eta_{\epsilon, b}\right\|_{L^{2}_{T}L^{2}_{x}}^{2}+\left\|D_{x}^{\frac{\alpha+1}{2}}\mathcal{H}\partial_{x}^{2}u\eta_{\epsilon,b}\right\|_{L^{2}_{T}L^{2}_{x}}^{2}\leq c^{*}_{2,1}
	\end{split}
\end{equation}
where $c^{*}_{2,1}=c^{*}_{2,1}\left(\alpha;\epsilon;T;v; \|u_{0}\|_{H_{x}^{\frac{3-\alpha}{2}}};\left\|\partial_{x}^{2}u_{0}\chi_{\epsilon, b}\right\|_{L_{x}^{2}}\right),$  for any $\epsilon>0,b\geq 5 \epsilon$ and   $v\geq0.$
\begin{flushleft}
	\textit{Step 2.}
\end{flushleft}
%
%
From equation  in (\ref{eq7})  one gets  after applying the operator $D_{x}^{\frac{1-\alpha}{2}}\partial_{x}^{2}$ and  multiplying  the result by $D_{x}^{\frac{1-\alpha}{2}}\partial_{x}^{2}u\chi_{\epsilon, b}^{2}(x+vt),$ 
 \begin{equation*}
 \begin{split}
 &D_{x}^{\frac{1-\alpha}{2}}\partial_{x}^{2}\partial_{t}uD_{x}^{\frac{1-\alpha}{2}}\partial_{x}^{2}u\chi_{\epsilon,b}^{2}-D_{x}^{\frac{1-\alpha}{2}}\partial_{x}^{2}D_{x}^{1+\alpha}\partial_{x}uD_{x}^{\frac{1-\alpha}{2}}\partial_{x}^{2}u\chi_{\epsilon, b}^{2}\\
 &+D_{x}^{\frac{1-\alpha}{2}}\partial_{x}^{2}(u\partial_{x}u)D_{x}^{\frac{1-\alpha}{2}}\partial_{x}^{2}u\chi_{\epsilon, b}^{2}=0
 \end{split}
 \end{equation*}
 which after integration in the spatial variable it  becomes 
 \begin{equation*}
\begin{split}
&\frac{1}{2}\frac{\mathrm{d}}{\mathrm{d}t}\int_{\mathbb{R}}\left(D_{x}^{\frac{1-\alpha}{2}}\partial_{x}^{2}u\right)^{2}\chi_{\epsilon, b}^{2}\,\mathrm{d}x\underbrace{-\frac{v}{2}\int_{\mathbb{R}}\left(D_{x}^{\frac{1-\alpha}{2}}\partial_{x}^{2}u\right)^{2}(\chi_{\epsilon, b}^{2})'\,\mathrm{d}x}_{ A_{1}(t)}\\
&\underbrace{-\int_{\mathbb{R}}\left(D_{x}^{\frac{1-\alpha}{2}}\partial_{x}^{2}D_{x}^{1+\alpha}\partial_{x}u\right)\left(D_{x}^{\frac{1-\alpha}{2}}\partial_{x}^{2}u\right)\chi_{\epsilon, b}^{2}\,\mathrm{d}x}_{ A_{2}(t)} \\
&\underbrace{+\int_{\mathbb{R}}\left(D_{x}^{\frac{1-\alpha}{2}}\partial_{x}^{2}( u\partial_{x}u)\right)\left(D_{x}^{\frac{1-\alpha}{2}}\partial_{x}^{2}u\right)\chi_{\epsilon, b}^{2}\,\mathrm{d}x}_{ A_{3}(t)}=0.
\end{split}
\end{equation*}
	To estimate $A_{1}$  we will use different techniques to  the  ones implemented  to bound   $A_{1}$ in the previous step. The main difficulty we have to face   is to  deal with the non-local character of the operator $D_{x}^{s}$ for $s\in\mathbb{R}^{+}\backslash 2\mathbb{N},$ the case $s\in 2\mathbb{N}$ is less complicated  because $D_{x}^{s}$ becomes local, so we can integrate by parts.
	
	The strategy to solve this issue will be the following. In \eqref{eq121} we proved that $u$ has 
 a gain of $\frac{\alpha+1}{2}$ derivatives (local) which in total sum $2+\frac{1+\alpha}{2}.$ This suggests that if we can find an appropriated channel  where we can  localize the smoothing effect, we shall be able to  recover all the local derivatives $r$  with $r\leq 2+\frac{1+\alpha}{2}.$  
  
 Henceforth  we will  employ recurrently a technique of localization of commutator used  by Kenig, Linares, Ponce and Vega \cite{KLPV} in the study of  propagation of regularity (fractional) for solutions of the k-generalized KdV equation. Indeed, the idea consists in constructing an appropriate system of  smooth partition of unit, localizing the regions where is available the information obtained in the previous cases.
  
  We recall that for $\epsilon>0$ and $b\geq 5\epsilon$
  \begin{equation}\label{eq61}
  \eta_{\epsilon,b}=\sqrt{\chi_{\epsilon, b}\chi_{\epsilon,b}'}\qquad \mbox{and}\qquad \chi_{\epsilon,b}+\phi_{\epsilon,b}+\psi_{\epsilon}=1.
  \end{equation}	

\S.1\quad  Claim
			\begin{equation}\label{d10}
			\left\|D_{x}^{\frac{1+\alpha}{2}}\partial_{x}^{2}(u\eta_{\epsilon,b})\right\|_{L_{T}^{2}L_{x}^{2}}<\infty.
			\end{equation}
	Combining the commutator estimate   (\ref{kpdl}), (\ref{eq61}), H\"{o}lder's inequality  and (\ref{eq121}) 
			 yields
			\begin{equation}\label{d1}
			\begin{split}
			\left\|D_{x}^{\frac{1+\alpha}{2}}\partial_{x}^{2}(u\eta_{\epsilon,b})\right\|_{L_{T}^{2}L_{x}^{2}}
			&\leq \left\|D_{x}^{2+\frac{1+\alpha}{2}}u \eta_{\epsilon,b}\right\|_{L_{T}^{2}L_{x}^{2}}+\left\|\left[D_{x}^{2+\frac{1+\alpha}{2}}; \eta_{\epsilon,b}\right]\left(u\chi_{\epsilon, b}+u\phi_{\epsilon,b}+u\psi_{\epsilon}\right)\right\|_{L_{T}^{2}L_{x}^{2}}\\
			&\lesssim \left(c^{*}_{2,1}\right)^{2} +\underbrace{\left\|D_{x}^{1+\frac{1+\alpha}{2}}(u\chi_{\epsilon, b})\right\|_{L^{2}_{T}L_{x}^{2}}}_{B_{1}}+\underbrace{\left\|D_{x}^{1+\frac{1+\alpha}{2}}(u\phi_{\epsilon, b})\right\|_{L^{2}_{T}L_{x}^{2}}}_{B_{2}}+\|u_{0}\|_{L_{x}^{2}}\\
			&\quad+\underbrace{\left\|\eta_{\epsilon,b
				}D_{x}^{2+\frac{1+\alpha}{2}}(u\psi_{\epsilon})\right\|_{L^{2}_{T}L_{x}^{2}}}_{B_{3}}.
			\end{split}
			\end{equation}
			Since ${\displaystyle \chi_{\epsilon/5,\epsilon}=1}$ on the support of $\chi_{\epsilon, b}$ then 
			\begin{equation*}
			\chi_{\epsilon,b}(x)\chi_{\epsilon/5, \epsilon}(x)=\chi_{\epsilon, b}(x)\quad \forall x\in \mathbb{R}.
			\end{equation*}
			Thus,
			combining  Lemma \ref{lema2} and  Young's inequality we obtain
		\begin{equation}\label{eq70}
			\begin{split}
			\left\|D_{x}^{1+\frac{1+\alpha}{2}}(u\chi_{\epsilon,b})\right\|_{L_{x}^{2}} 
		&\lesssim \|\partial_{x}^{2}u\chi_{\epsilon,b}\|_{L_{x}^{2}}+\|\partial_{x}u\chi_{\frac{\epsilon}{5}, \epsilon}\|_{L_{x}^{2}}+\|u_{0}\|_{L_{x}^{2}}.
		\end{split}
			\end{equation}
Then, by an application of    (\ref{eq121}) adapted  to  every case yields 
	\begin{equation}\label{eq71}
	\begin{split}
		B_{1}&\lesssim \|\partial_{x}^{2}u\chi_{\epsilon,b}\|_{L_{T}^{\infty}L_{x}^{2}}+\|\partial_{x}u\chi_{\frac{\epsilon}{5}, \epsilon}\|_{L_{T}^{\infty}L_{x}^{2}}+\|u_{0}\|_{L_{x}^{2}}
		\lesssim c^{*}_{2,1}+c^{*}_{1,1}+\|u_{0}\|_{L_{x}^{2}}.
	\end{split}
	\end{equation}
	Notice that $B_{2}$ was estimated in the case $j=1,$ step 2 see (\ref{eq73}), so we will omit the proof.
			Next, we recall that  by construction
			\begin{equation*}
			\dist\left(\supp\left(\eta_{\epsilon,b}\right), \supp\left(\psi_{\epsilon}\right)\right)\geq \frac{\epsilon}{2}.
			\end{equation*}
			Hence by Lemma \ref{lemma1} 
			\begin{equation}\label{d4}
			\begin{split}
			B_{3}&=\left\|\eta_{\epsilon,b}D_{x}^{2+\frac{1+\alpha}{2}}(u\psi_{\epsilon})\right\|_{L_{T}^{2}L_{x}^{2}}
			\lesssim \|\eta_{\epsilon,b}\|_{L_{T}^{\infty}L_{x}^{\infty}}\|u_{0}\|_{L_{x}^{2}}.
			\end{split}
			\end{equation}
			
		The claim follows gathering the calculations above.
		
		At this point we have proved that locally in the interval $[\epsilon,b]$ there exists $2+\frac{\alpha+1}{2}$ derivatives. By Lemma \ref{lema2} we get
		 \begin{equation*}
		 \left\|D_{x}^{2+\frac{1-\alpha}{2}}(u\eta_{\epsilon,b})\right\|_{L_{T}^{2}L_{x}^{2}}\lesssim \left\|D_{x}^{2+\frac{1+\alpha}{2}}(u\eta_{\epsilon,b})\right\|_{L_{T}^{2}L_{x}^{2}}+\|u_{0}\|_{L_{x}^{2}}<\infty.
		 \end{equation*}
		 As before 
		 \begin{equation*}
		 D_{x}^{2+\frac{1-\alpha}{2}}u\eta_{\epsilon,b}=D_{x}^{2+\frac{1-\alpha}{2}}\left(u\eta_{\epsilon,b}\right)-\left[D_{x}^{2+\frac{1-\alpha}{2}}; \eta_{\epsilon,b}\right] \left(u\chi_{\epsilon,b}+u \phi_{\epsilon,b}+ u\psi_{\epsilon}\right).
		 \end{equation*}
		  The argument  used in the proof of the claim  yields
		  \begin{equation*}
		 \left\|D_{x}^{2+\frac{1-\alpha}{2}}u\eta_{\epsilon,b}\right\|_{L^{2}_{T}L_{x}^{2}}<\infty.
		 \end{equation*}
		 Therefore,
		\begin{equation}\label{d16}
		\begin{split}
		\int_{0}^{T}|A_{1}(t)|\,\mathrm{d}t&\leq |v|\int_{0}^{T}\int_{\mathbb{R}}(D_{x}^{\frac{1-\alpha}{2}}\partial_{x}^{2}u)^{2}(\chi_{\epsilon, b}^{2})'\mathrm{d}x\,\mathrm{d}t\lesssim\left\|D_{x}^{2+\frac{1-\alpha}{2}}u\eta_{\epsilon,b}\right\|_{L_{T}^{2}L_{x}^{2}}^{2}<\infty.
		\end{split}
		\end{equation}
%
\S{.2} \quad Now we focus our attention in the term $A_{2}.$ 	Notice that  after integration by parts and Plancherel's identity
	\begin{equation}\label{eq8.15.1}
	\begin{split}
	A_{2}(t)
	&=-\frac{1}{2}\int_{\mathbb{R}}D_{x}^{\frac{5-\alpha}{2}}u\left[\mathcal{H}D_{x}^{2+\alpha}; \chi_{\epsilon, b}^{2}\right] D_{x}^{\frac{5-\alpha}{2}}u\,\mathrm{d}x.
	\end{split}
	\end{equation}
		The procedure to decompose the commutator will be almost  similar to the introduced in the previous step, the main difference  relies on the fact that the quantity of derivatives is higher in comparison with the step 1.
	
	Concerning this, we  notice that  $2+\alpha>1$ and  by (\ref{eq8}) the commutator  ${\displaystyle \left[\mathcal{H}D_{x}^{\alpha+2};\chi_{\epsilon, b}^{2}\right]}$ can be decomposed as
	\begin{equation}\label{eq8.14.1}
	\left[\mathcal{H}D_{x}^{\alpha+2};\chi_{\epsilon, b}^{2}\right]+\frac{1}{2}P_{n}(\alpha+2)+R_{n}(\alpha+2)=\frac{1}{2}\mathcal{H}P_{n}(\alpha+2)\mathcal{H}
	\end{equation}
	for some positive integer $n.$ We shall fix the value of $n$
satisfying a suitable condition. 
	
	Replacing (\ref{eq8.14.1}) into (\ref{eq8.15.1})  produces
	\begin{equation}\label{eq8.14.2}
	\begin{split}
	A_{2}(t)&=\frac{1}{2}\int_{\mathbb{R}}D_{x}^{\frac{5-\alpha}{2}}u\left(R_{n}(\alpha+2)D_{x}^{\frac{5-\alpha}{2}}u\right)\,\mathrm{d}x+
	\frac{1}{4}\int_{\mathbb{R}}D_{x}^{\frac{5-\alpha}{2}}u\left(P_{n}(\alpha+2)D_{x}^{\frac{5-\alpha}{2}}u\right)\,\mathrm{d}x\\
	&\quad-\frac{1}{4}\int_{\mathbb{R}}D_{x}^{\frac{5-\alpha}{2}}u\left(\mathcal{H}P_{n}(\alpha+2)\mathcal{H}D_{x}^{\frac{5-\alpha}{2}}u\right)\,\mathrm{d}x\\
	&=A_{2,1}(t)+A_{2,2}(t)+A_{2,3}(t).
	\end{split}
	\end{equation}
	Now we proceed to fix the value of $n$ present in  $A_{2,1},A_{2,2}$ and $A_{2,3}.$ 

	First we deal with the term  that determine the value $n$    in the decomposition 
	associated to $A_{2}.$ In this case it corresponds to   $A_{2,1}.$ 
	
	Applying  Plancherel's identity, $A_{2,1}$ becomes 
	\begin{equation*}
	\begin{split}
	A_{2,1}(t)&=\frac{1}{2}\int_{\mathbb{R}}uD_{x}^{\frac{5-\alpha}{2}}\left\{R_{n}(\alpha+2)D_{x}^{\frac{5-\alpha}{2}}u\right\}\,\mathrm{d}x
	\end{split}
	\end{equation*}
	We   fix $n$ such that  it satisfies (\ref{eq21}) i.e.,
	\begin{equation*}
	2n+1\leq a+2\sigma\leq 2n+3
	\end{equation*} 
 with $a=\alpha+2$ and $\sigma=\frac{5-\alpha}{2},$ which  produces $n=2$ or $n=3.$  Nevertheless, for the sake of simplicity we take  $n=2.$ 
 
Hence, by construction $R_{2}(\alpha+2)$ is bounded in $L^{2}_{x}$ (see Proposition \ref{propo2}).

	Thus,
	\begin{equation*}
	\begin{split}
	\int_{0}^{T}|A_{2,1}(t)|\,\mathrm{d}t&\leq c \int_{0}^{T}\|u(t)\|_{L^{2}_{x}}^{2}\left\|\widehat{D_{x}^{7}(\chi_{\epsilon, b}^{2}(\cdot+vt))}\right\|_{L^{1}_{\xi}}\,\mathrm{d}t
	\lesssim\|u_{0}\|_{L^{2}_{x}}^{2}\sup_{0\leq t\leq T}\left\|\widehat{D_{x}^{7}(\chi_{\epsilon, b}^{2})}\right\|_{L^{1}_{\xi}}.
	\end{split}
	\end{equation*} 	
	Since 	we have   fixed $n=2,$ we obtain after replace $P_{2}(\alpha+2)$  into $A_{2,2}$  
	\begin{equation*}
	\begin{split}
	A_{2,2}(t) 
	&=\widetilde{c_{1}}\int_{\mathbb{R}}\left(\mathcal{H}\partial_{x}^{3}u\right)^{2}(\chi_{\epsilon, b}^{2})'\;\mathrm{d}x-\widetilde{c_{3}}\int_{\mathbb{R}}\left(\partial_{x}^{2}u\right)^{2}(\chi_{\epsilon, b}^{2})^{(3)}\,\mathrm{d}x\\
	&\quad +\widetilde{c_{5}}\int_{\mathbb{R}}\left(\mathcal{H}\partial_{x}u\right)^{2}(\chi_{\epsilon, b}^{2})^{(5)}\,\mathrm{d}x\\
	&=A_{2,2,1}(t)+A_{2,2,2}(t)+A_{2,2,3}(t).
	\end{split}
	\end{equation*}	
		We underline  that $A_{2,2,1}$ is positive and represents the smoothing effect.
	
	 On the other hand, by (\ref{eq120}) with $(\epsilon,b)=(\epsilon/5,\epsilon)$ we have 
%
	\begin{equation}\label{d12}
	\begin{split}
	\int_{0}^{T}|A_{2,2,2}(t)|\,\mathrm{d}t &=c\int_{0}^{T}\int_{\mathbb{R}}(\partial_{x}^{2}u)^{2}\chi_{\frac{\epsilon}{5},\epsilon}^{2}(\chi_{\epsilon,b}^{2})'''\,\mathrm{d}x\,\mathrm{d}t
	\lesssim \sup_{0\leq t\leq T}\int_{\mathbb{R}}(\partial_{x}^{2}u)^{2} \chi_{\frac{\epsilon}{5},\epsilon}^{2}\mathrm{d}x\lesssim c^{*}_{1,2}.
		\end{split}
	\end{equation}
	Next, by the local theory
	\begin{equation}\label{d13}
	\begin{split}
		\int_{0}^{T}|A_{2,2,3}(t)|\,\mathrm{d}t&\lesssim \|u\|_{L_{T}^{\infty}H_{x}^{\frac{3-\alpha}{2}}}.
	\end{split}
		\end{equation}
After replacing $P_{2}(\alpha+2)$  into     $A_{2,3},$ 
	 and   using  the fact  that Hilbert transform is skew adjoint 
	\begin{equation*}
	\begin{split}
	A_{2,3}(t) &
	=\left(\frac{\alpha+2}{4}\right)\int_{\mathbb{R}} (\partial_{x}^{3}u)^{2}(\chi_{\epsilon, b}^{2})'\,\mathrm{d}x-c_{3}\left(\frac{\alpha+2}{16}\right)\int_{\mathbb{R}}(\mathcal{H}\partial_{x}^{2}u)^{2}(\chi_{\epsilon, b}^{2})'''\,\mathrm{d}x\\
	&\quad +c_{5}\left(\frac{\alpha+2}{64}\right)\int_{\mathbb{R}}(\partial_{x}u)^{2}(\chi_{\epsilon, b}^{2})^{(5)}\,\mathrm{d}x\\
	&=A_{2,3,1}(t)+A_{2,3,2}(t)+A_{2,3,3}(t).
 	\end{split}
	\end{equation*}
Notice that    $A_{2,3,1}\geq 0$  and it  represents the smoothing effect. However, $A_{2,3,2}$ can be handled if we take $(\epsilon,b)=(\epsilon/5,\epsilon)$ in (\ref{eq110})  as follows 
\begin{equation*}
\begin{split}
A_{2,3,3}(t)&=\int_{\mathbb{R}}(\partial_{x}^{2}u)^{2}\chi_{\frac{\epsilon}{5}, \epsilon}^{2}(\chi_{\epsilon,b }^{2})'''\,\mathrm{d}x
\lesssim \int_{\mathbb{R}}(\partial_{x}^{2}u)^{2}\chi_{\frac{\epsilon}{5}, \epsilon}^{2}\,\mathrm{d}x,
\end{split}
\end{equation*}
thus,
\begin{equation*}
\int_{0}^{T}|A_{2,3,3}(t)|\,\mathrm{d}t\lesssim \sup_{0\leq t\leq  T}\int_{\mathbb{R}} (\partial_{x}u)^{2}\chi_{\frac{\epsilon}{5}, \epsilon}^{2}\,\mathrm{d}x\lesssim c_{1,1}^{*}.
\end{equation*}
To finish the estimate of   $A_{2}$ only remains  to bound  $A_{2,3,2}.$ To do this we recall that 
\begin{equation*}
\left|\chi_{\epsilon,b }^{(j)}(x)\right|\lesssim \chi_{\epsilon/3,b+\epsilon}'(x),\quad \forall x\in \mathbb{R}, j\in\mathbb{Z}^{+},
\end{equation*}
 that joint with the property (9) of $\chi_{\epsilon, b}$ yields
 \begin{equation*}
 \begin{split}
 \int_{0}^{T} \int_{\mathbb{R}}\left(\mathcal{H}\partial_{x}^{2}u\right)^{2}\chi_{\epsilon/3,b+\epsilon}'\,\mathrm{d}x\,\mathrm{d}t& \lesssim \left\|\mathcal{H}\partial_{x}^{2}u\eta_{\epsilon/9,b+10\epsilon/9}\right\|_{L^{2}_{T}L^{2}_{x}}^{2} \lesssim c^{*}_{1,2},
 \end{split}
 \end{equation*} 
 where the last inequality is obtained taking $(\epsilon,b)=(\epsilon/9,b+10\epsilon/9)$ in (\ref{eq120}).
 The term $A_{2,3,3}$ can be handled by interpolation and the local theory.

\S.3\quad Finally we turn our attention to $A_{3}$.We start rewriting the nonlinear part as follows
\begin{equation}\label{d14}
\begin{split}
&D_{x}^{\frac{1-\alpha}{2}}\partial_{x}^{2}\left(u\partial_{x}u\right)\chi_{\epsilon, b}\\
&=-\frac{1}{2}\left[D_{x}^{\frac{1-\alpha}{2}}\partial_{x}^{2};\chi_{\epsilon, b}\right]\partial_{x}\left((u\chi_{\epsilon,b})^{2}+(u\widetilde{\phi_{\epsilon,b}})^{2}+(\psi_{\epsilon}u^{2})\right)\\
&\quad +\left[D_{x}^{\frac{1-\alpha}{2}}\partial_{x}^{2}; u\chi_{\epsilon, b}\right]\partial_{x}\left((u\chi_{\epsilon, b})+ (u\phi_{\epsilon,b})+(u\psi_{\epsilon})\right)
+ u\chi_{\epsilon, b}D_{x}^{\frac{1-\alpha}{2}}\partial_{x}^{3}u\\
&=\widetilde{A_{3,1}}(t)+\widetilde{A_{3,2}}(t)+\widetilde{A_{3,3}}(t)+\widetilde{A_{3,4}}(t)+\widetilde{A_{3,5}}(t)+\widetilde{A_{3,6}}(t)+\widetilde{A_{3,7}}(t).
\end{split}
\end{equation}
Hence,  after replacing (\ref{d13}) into $A_{3}$ and apply H\"{o}lder's inequality
\begin{equation*}
\begin{split}
&A_{3}(t)\\
&=\sum_{1\leq m\leq 6} \int_{\mathbb{R}}\widetilde{A_{3,m}}(t)\,D_{x}^{\frac{1-\alpha}{2}}\partial_{x}^{2}u\,\chi_{\epsilon,b}\,\mathrm{d}x+\int_{\mathbb{R}}\widetilde{A_{3,7}}(t)\,D_{x}^{\frac{1-\alpha}{2}}\partial_{x}^{2}u\,\chi_{\epsilon,b}\,\mathrm{d}x\\
&\leq \sum_{1\leq m \leq 6}\|\widetilde{A_{3,m}}(t)\|_{L_{x}^{2}}\left\|D_{x}^{2+\frac{1-\alpha}{2}}u(t)\chi_{\epsilon,b}(\cdot+vt)\right\|_{L_{x}^{2}}+\int_{\mathbb{R}}\widetilde{A_{3,7}}(t)\,D_{x}^{\frac{1-\alpha}{2}}\partial_{x}^{2}u\,\chi_{\epsilon,b}\,\mathrm{d}x\\
&=\left\|D_{x}^{2+\frac{1-\alpha}{2}}u(t)\chi_{\epsilon,b}(\cdot+vt)\right\|_{L_{x}^{2}}\sum_{1\leq m\leq 6}A_{3,m}(t)+ A_{3,7}(t).
\end{split}
\end{equation*}
Notice that the first factor  in the right hand side is the quantity to be estimated by Gronwall's inequality. So, we  shall focus   on establish  control in the remaining terms. 

First,  combining \eqref{eq5}, \eqref{kpdl} and Lemma \ref{lema2} one gets  that 
\begin{equation}\label{d17}
\begin{split}
\|\widetilde{A_{3,1}}(t)\|_{L^{2}_{x}}
&\lesssim  \left\|D_{x}^{2+\frac{1-\alpha}{2}}(u\chi_{\epsilon,b})\right\|_{L^{2}_{x}}\|u\|_{L^{\infty}_{x}}+\|u_{0}\|_{L_{x}^{2}}\|u\|_{L^{\infty}_{x}},
\end{split}
\end{equation}
and 
\begin{equation}\label{d15}
\begin{split}
\|\widetilde{A_{3,2}}(t)\|_{L^{2}_{x}}
&\lesssim   \left\|D_{x}^{2+\frac{1-\alpha}{2}}(u\widetilde{\phi_{\epsilon,b}})\right\|_{L^{2}_{x}}\|u\|_{L^{\infty}_{x}}+\|u_{0}\|_{L_{x}^{2}}\|u\|_{L^{\infty}_{x}}.
\end{split}
\end{equation}
To finish with the  quadratic terms, we employ Lemma \ref{lemma1}
\begin{equation*}
\begin{split}
\|\widetilde{A_{3,3}}(t)\|_{L^{2}_{x}}
&\lesssim\|u_{0}\|_{L^{2}_{x}}\|u\|_{L^{\infty}_{x}}.
\end{split}
\end{equation*}
Combining  (\ref{eq31}) and (\ref{kpdl}) we  obtain
\begin{equation*}
\begin{split}
\|\widetilde{A_{3,4}}(t)\|_{L^{2}_{x}}
&\lesssim\|\partial_{x}(u\chi_{\epsilon,b})\|_{L^{\infty}_{x}}\left\|D_{x}^{2+\frac{1-\alpha}{2}}(u\chi_{\epsilon,b})\right\|_{L^{2}_{x}}.
\end{split}
\end{equation*}
Meanwhile,
\begin{equation*}
\begin{split}
\|\widetilde{A_{3,5}}(t)\|_{L^{2}_{x}}
&\lesssim\|\partial_{x}(u\chi_{\epsilon,b})\|_{L^{\infty}_{x}}\left\|D_{x}^{2+\frac{1-\alpha}{2}}(u\phi_{\epsilon,b})\right\|_{L^{2}_{x}}+\|\partial_{x}(u\phi_{\epsilon,b})\|_{L^{\infty}_{x}}\left\|D_{x}^{2+\frac{1-\alpha}{2}}(u\chi_{\epsilon,b})\right\|_{L^{2}_{x}}.
\end{split}
\end{equation*}
	Next, we recall that  by construction
\begin{equation*}
\dist\left(\supp\left(\chi_{\epsilon, b}\right), \supp\left(\psi_{\epsilon}\right)\right)\geq \frac{\epsilon}{2}.
\end{equation*}
Thus by    Lemma \ref{lemma1} 
\begin{equation*}
\begin{split}
\|\widetilde{A_{3,6}}(t)\|_{L^{2}_{x}}
&\lesssim\|u_{0}\|_{L^{2}_{x}}\|u\|_{L^{\infty}_{x}}.
\end{split}
\end{equation*}
To complete the estimates in  (\ref{d17})-(\ref{d15}) only remains  to bound $\left\|D_{x}^{2+\frac{1-\alpha}{2}}\left(u\chi_{\epsilon,b}\right)\right\|_{L_{x}^{2}},$\linebreak
$\left\|D_{x}^{2+\frac{1-\alpha}{2}}\left(u\widetilde{\phi_{\epsilon,b}}\right)\right\|_{L_{x}^{2}},$ and $\left\|D_{x}^{2+\frac{1-\alpha}{2}}(u\phi_{\epsilon,b})\right\|_{L_{x}^{2}}.$ 

 For the first term  we proceed by writing 
\begin{equation*}
\begin{split}
D_{x}^{2+\frac{1-\alpha}{2}}(u\chi_{\epsilon, b})&=D_{x}^{2+\frac{1-\alpha}{2}}u \chi_{\epsilon, b}+\left[D_{x}^{2+\frac{1-\alpha}{2}}; \chi_{\epsilon, b}\right]\left(u\chi_{\epsilon, b}+u\phi_{\epsilon, b}+u\psi_{\epsilon}\right)\\
&= I_{1}+I_{2}+I_{3}+I_{4}.
\end{split}
\end{equation*}
Notice that $\|I_{1}\|_{L_{x}^{2}}$ is the quantity to be estimated by Gronwall's inequality. Meanwhile,  $\|I_{2}\|_{L_{x}^{2}},$\, $\|I_{3}\|_{L_{x}^{2}}$ and $\|I_{4}\|_{L_{x}^{2}}$ were estimated previously in the case $j=1,$ step 2.

Next, we focus on  estimate the term  $\left\|D_{x}^{2+\frac{1+\alpha}{2}}\left(u\phi_{\epsilon,b}\right)\right\|_{L_{x}^{2}}$ 
which will be treated by means of  H\"{o}lder's inequality  and Theorem \ref{thm11} as follows
\begin{equation*}
\begin{split}
&\left\|D_{x}^{2+\frac{1+\alpha}{2}}(u\phi_{\epsilon,b})\right\|_{L^{2}_{x}}\\
&\lesssim \|u_{0}\|_{L_{x}^{2}}^{1/2}\|u\|_{L_{x}^{\infty}}^{1/2}+ \left\|\eta_{\epsilon/24,b+7\epsilon/24} D_{x}^{2+\frac{1+\alpha}{2}}u\right\|_{L_{x}^{2}}+\left\|\eta_{\epsilon/24,b+7\epsilon/24} D_{x}^{\frac{1+\alpha}{2}}\partial_{x}u\right\|_{L_{x}^{2}}\\
&\quad +\left\|D_{x}^{\frac{1+\alpha}{2}}u\right\|_{L_{x}^{2}}.
\end{split}
\end{equation*}
After integrate in time, the second  and third term on  the right hand side  can be estimated taking $(\epsilon,b)=(\epsilon/24,b+7\epsilon/24 )$ in (\ref{eq121}) and  (\ref{eq110}) respectively. Hence, after integrate in time follows  by interpolation that $\left\|D_{x}^{2+\frac{1-\alpha}{2}}(u\phi_{\epsilon,b})\right\|_{L^{2}_{T}L_{x}^{2}}<\infty.$

Analogously can be bounded  $\left\|D_{x}^{2+\frac{1-\alpha}{2}}(u\widetilde{\phi_{\epsilon,b}})\right\|_{L_{x}^{2}}$.
%

\S.3\quad  Finally, after apply integration by parts
\begin{equation*}
\begin{split}
A_{3,7}(t)
&=-\frac{1}{2}\int_{\mathbb{R}}\partial_{x}u\,\chi_{\epsilon, b}^{2}\left(D_{x}^{\frac{1-\alpha}{2}}\partial_{x}^{2}u\right)^{2}\,\mathrm{d}x-\int_{\mathbb{R}}u\chi_{\epsilon, b} \,\chi_{\epsilon, b}'\left(D_{x}^{\frac{1-\alpha}{2}}\partial_{x}^{2}u\right)^{2}\,\mathrm{d}x\\
&=A_{3,7,1}(t)+A_{3,7,2}(t).
\end{split}
\end{equation*}
First, 
\begin{equation*}
|A_{3,7,1}(t)| \lesssim\|\partial_{x}u(t)\|_{L^{\infty}_{x}}\int_{\mathbb{R}}\left(D_{x}^{\frac{1-\alpha}{2}}\partial_{x}^{2}u\right)^{2}\chi_{\epsilon, b}^{2}\,\mathrm{d}x,
\end{equation*}
where the last integral  is the quantity that will be estimated using Gronwall's inequality, and the other factor will be controlled after integration in time.

After integration in time and   Sobolev's embedding it follows that 
\begin{equation*}
\begin{split}
\int_{0}^{T}|A_{3,7,2}(t)|\,\mathrm{d}t&\lesssim\int_{0}^{T}\int_{\mathbb{R}}u\,(\chi_{\epsilon, b}^{2})'\left(D_{x}^{\frac{1-\alpha}{2}}\partial_{x}^{2}u\right)^{2}\mathrm{d}x\,\mathrm{d}t\\
&\lesssim\left(\sup_{0\leq t\leq T}\|u(t)\|_{H^{s(\alpha)+}_{x}}\right)\int_{0}^{T}\int_{\mathbb{R}}\left(D_{x}^{\frac{1-\alpha}{2}}\partial_{x}^{2}u\right)^{2}\,(\chi_{\epsilon, b}^{2})'\mathrm{d}x\,\mathrm{d}t
\end{split}
\end{equation*}
and the last term was already estimated in (\ref{d16}).

Thus, after collecting all the information in this step and applying  Gronwall's inequality together with  hypothesis (\ref{clave1}), we obtain 
\begin{equation*}
\begin{split}
&\sup_{0\leq t\leq T}\left\|D_{x}^{\frac{1-\alpha}{2}}\partial_{x}^{2}u\chi_{\epsilon,b}\right\|_{L^{2}_{x}}^{2}
+\left\|\partial_{x}^{3}u\eta_{\epsilon,b}\right\|_{L^{2}_{T}L^{2}_{x}}^{2}+\left\| \mathcal{H}\partial_{x}^{3}u\eta_{\epsilon,b}\right\|_{L^{2}_{T}L^{2}_{x}}^{2}\leq c^{*}_{2,2}
\end{split}
\end{equation*}
where $c^{*}_{2,2}=c^{*}_{2,2}\left(\alpha;\epsilon;T;v;\|u_{0}\|_{H_{x}^{\frac{3-\alpha}{2}}};\left\|D_{x}^{\frac{1-\alpha}{2}}\partial_{x}^{2}u_{0}\chi_{\epsilon, b}\right\|_{L_{x}^{2}}\right)$
 for any $\epsilon>0,\,b\geq 5\epsilon$ and $v>0.$

 According to the induction argument  we shall assume that  (\ref{l1}) holds for $j\leq m$ with $j\in \mathbb{Z}$ and $j\geq 2,$ i.e.
\begin{equation}\label{l2}
\begin{split}
&\sup_{0\leq t\leq T}\left\|\partial_{x}^{j}u\chi_{\epsilon,b}\right\|_{L^{2}_{x}}^{2}+\left\|D_{x}^{\frac{1+\alpha}{2}}\partial_{x}^{j}u\eta_{\epsilon,b}\right\|_{L^{2}_{T}L^{2}_{x}}^{2}+\left\|\mathcal{H}D_{x}^{\frac{1+\alpha}{2}}\partial_{x}^{j}u\eta_{\epsilon,b}\right\|_{L^{2}_{T}L^{2}_{x}}^{2}\leq c^{*}_{j,1}
\end{split}
\end{equation}	
for $j=1,2,\dots,m$ with $m\geq 1,$ for any $\epsilon>0,\,b\geq 5\epsilon$ \, $v\geq0.$
\begin{flushleft}
\underline{Step 2}
\end{flushleft}
We will assume  $j$  an even integer. The case where $j$ is odd follows by an argument similar  to the case $j=1.$

By   an analogous reasoning to one  employed in the case $j=2$ it   follows that 
\begin{equation*}
\begin{split}
&D_{x}^{\frac{1-\alpha}{2}}\partial_{x}^{j}\partial_{t}uD_{x}^{\frac{1-\alpha}{2}}\partial_{x}^{j}u\chi_{\epsilon,b}^{2}-D_{x}^{\frac{1-\alpha}{2}}\partial_{x}^{j}D_{x}^{1+\alpha}\partial_{x}uD_{x}^{\frac{1-\alpha}{2}}\partial_{x}^{j}u\chi_{\epsilon, b}^{2}\\
&+D_{x}^{\frac{1-\alpha}{2}}\partial_{x}^{j}(u\partial_{x}u)D_{x}^{\frac{1-\alpha}{2}}\partial_{x}^{j}u\chi_{\epsilon, b}^{2}=0
\end{split}
\end{equation*}
which after integrating in time yields the identity
\begin{equation}\label{eqa1}
\begin{split}
&\frac{1}{2}\frac{\mathrm{d}}{\mathrm{d}t}\int_{\mathbb{R}}\left(D_{x}^{\frac{1-\alpha}{2}}\partial_{x}^{j}u\right)^{2}\chi_{\epsilon, b}^{2}\,\mathrm{d}x\underbrace{-\frac{v}{2}\int_{\mathbb{R}}\left(D_{x}^{\frac{1-\alpha}{2}}\partial_{x}^{j}u\right)^{2}(\chi_{\epsilon, b}^{2})'\,\mathrm{d}x}_{A_{1}(t)}\\
&\underbrace{-\int_{\mathbb{R}}\left(D_{x}^{\frac{1-\alpha}{2}}\partial_{x}^{j}D_{x}^{1+\alpha}\partial_{x}u\right)\left(D_{x}^{\frac{1-\alpha}{2}}\partial_{x}^{j}u \chi_{\epsilon, b}^{2}\right)\,\mathrm{d}x}_{A_{2}(t)} \\
&\underbrace{+\int_{\mathbb{R}}D_{x}^{\frac{1-\alpha}{2}}\partial_{x}^{j}( u\partial_{x}u)\left(D_{x}^{\frac{1-\alpha}{2}}\partial_{x}^{j}u \chi_{\epsilon, b}^{2}\right)\,\mathrm{d}x}_{A_{3}(t)}=0.
\end{split}
\end{equation}
\S.1  We claim that 
	\begin{equation}\label{d10.1}
	\left\|D_{x}^{j+\frac{1+\alpha}{2}}(u \eta_{\epsilon,b})\right\|_{L_{T}^{2}L_{x}^{2}}<\infty.
	\end{equation}
	We proceed as in  the case $j=2.$
A	combination of  the commutator estimate   (\ref{kpdl}), (\ref{eq61}), H\"{o}lder's inequality  and (\ref{l2}) 
yields
\begin{equation}\label{d1.1}
\begin{split}
&\left\|D_{x}^{\frac{1+\alpha}{2}}\partial_{x}^{j}(u\eta_{\epsilon,b})\right\|_{L_{T}^{2}L_{x}^{2}}\\
&\leq \left\|D_{x}^{j+\frac{1+\alpha}{2}}u \eta_{\epsilon,b}\right\|_{L_{T}^{2}L_{x}^{2}}+\left\|\left[D_{x}^{j+\frac{1+\alpha}{2}}; \eta_{\epsilon,b}\right]\left(u\chi_{\epsilon, b}+u\phi_{\epsilon,b}+u\psi_{\epsilon}\right)\right\|_{L_{T}^{2}L_{x}^{2}}\\
&\lesssim (c^{*}_{j,1} )^{2} +\underbrace{\left\|D_{x}^{j-1+\frac{1+\alpha}{2}}(u\chi_{\epsilon, b})\right\|_{L^{2}_{T}L_{x}^{2}}}_{B_{1}}+ \|u_{0}\|_{L_{x}^{2}}+\underbrace{\left\|D_{x}^{j-1+\frac{1+\alpha}{2}}(u\phi_{\epsilon, b})\right\|_{L^{2}_{T}
		L_{x}^{2}}}_{B_{2}}\\
	&\quad +\underbrace{\left\|\eta_{\epsilon,b
	}D_{x}^{j+\frac{1+\alpha}{2}}(u\psi_{\epsilon})\right\|_{L^{2}_{T}L_{x}^{2}}}_{B_{3}}.
\end{split}
\end{equation}
Since ${\displaystyle \chi_{\epsilon/5,\epsilon}=1}$ on the support of $\chi_{\epsilon, b}$ then 
\begin{equation*}
\chi_{\epsilon,b}(x)\chi_{\epsilon/5, \epsilon}(x)=\chi_{\epsilon, b}(x)\quad \forall x\in \mathbb{R}.
\end{equation*}
Combining  Lemma \ref{lema2} and  Young's inequality
	\begin{equation}\label{d2.1}
\begin{split}
\left\|D_{x}^{j+\frac{\alpha-1}{2}}\left(u\chi_{\epsilon, b}\right)\right\|_{L_{x}^{2}}
&\lesssim  \left\|\partial_{x}^{j}u\chi_{\epsilon,b}\right\|_{L_{x}^{2}}^{2}+\sum_{2\leq k\leq j-1} \gamma_{k,j}\left\|\chi_{\epsilon,b}^{(j-k)}\right\|_{L_{x}^{\infty}}\left\|\partial_{x}^{k}u\chi_{\epsilon/5, \epsilon}\right\|_{L_{x}^{2}}\\
& \quad +\|u\|_{L_{T}^{\infty}H_{x}^{\frac{3-\alpha}{2}}}+\|u_{0}\|_{L_{x}^{2}}.
\end{split}
\end{equation}
Hence, taking $(\epsilon,b)=(\epsilon/5,\epsilon)$ in (\ref{l2}) yields
 
\begin{equation}\label{eq71.1}
\begin{split}
B_{1}&\lesssim c^{*}_{j,1}+\sum_{2\leq k\leq j-1} \gamma_{k,j} c^{*}_{k,1}+\|u\|_{L_{T}^{\infty}H_{x}^{\frac{3-\alpha}{2}}}+\|u_{0}\|_{L_{x}^{2}}.
\end{split}
\end{equation}
$B_{2}$  can be estimated as in the step 2 of the case $j-1, $ so is bounded by the induction hypothesis.
		
	Next, since 
	\begin{equation*}
	\dist\left(\supp\left(\eta_{\epsilon,b}\right), \supp\left(\psi_{\epsilon}\right)\right)\geq \frac{\epsilon}{2}
	\end{equation*}
		we have by Lemma \ref{lemma1} 
	\begin{equation*}
	\begin{split}
\left\|\eta_{\epsilon,b}D_{x}^{j+\frac{\alpha+1}{2}}(u\psi_{\epsilon})\right\|_{L_{x}^{2}}&=\left\|\eta_{\epsilon,b}D_{x}^{j+\frac{1+\alpha}{2}}(u\psi_{\epsilon})\right\|_{L_{x}^{2}}\lesssim \left\|\eta_{\frac{\epsilon}{8},b+\epsilon}\right\|_{L_{x}^{\infty}}\|u_{0}\|_{L_{x}^{2}}.\\
	\end{split}
	\end{equation*}
	Gathering the estimates above     follows the claim 1. 

 We have proved that locally in the interval $[\epsilon,b]$ there exists $j+\frac{\alpha+1}{2}$ derivatives. So, by Lemma \ref{lema2} we obtain 
\begin{equation*}
\left\|D_{x}^{j+\frac{1-\alpha}{2}}(u\eta_{\epsilon,b})\right\|_{L_{T}^{2}L_{x}^{2}}\lesssim 
 \left\|D_{x}^{j+\frac{1+\alpha}{2}}(u\eta_{\epsilon,b})\right\|_{L_{T}^{2}L_{x}^{2}}+\|u_{0}\|_{L_{x}^{2}},
\end{equation*}
then,  as before 
\begin{equation*}
D_{x}^{j+\frac{1-\alpha}{2}}u\eta_{\epsilon,b}=c_{j}D_{x}^{j+\frac{1-\alpha}{2}}(u\eta_{\epsilon,b})-c_{j}\left[D_{x}^{j+\frac{1-\alpha}{2}}; \eta_{\epsilon,b}\right] \left(u\chi_{\epsilon,b}+u \phi_{\epsilon,b}+ u\psi_{\epsilon}\right),
\end{equation*}
where $c_{j}$ is a constant depending only on $j.$

Hence, if we proceed as in the proof of claim 1 we have 
\begin{equation}\label{eq41}
\left\|D_{x}^{j+\frac{1-\alpha}{2}}u\eta_{\epsilon,b}\right\|_{L^{2}_{T}L_{x}^{2}}<\infty.
\end{equation}
Therefore  
\begin{equation*}\label{d6.1}
\begin{split}
\int_{0}^{T}|A_{1}(t)|\,\mathrm{d}t &=v \left\|D_{x}^{j+\frac{1-\alpha}{2}}u\eta_{\epsilon,b}\right\|_{L_{T}^{2}L_{x}^{2}}^{2}<\infty.
\end{split}
\end{equation*} 
\S.2\quad 	To handle the  term   $A_{2}$ we use  the same procedure as in the  previous steps.

First,
	\begin{equation}\label{eq8.15.1.1.1}
	\begin{split}
	A_{2}(t)
	&=-\frac{1}{2}\int_{\mathbb{R}}D_{x}^{\frac{2j+1-\alpha}{2}}u\left[\mathcal{H}D_{x}^{2+\alpha}; \chi_{\epsilon, b}^{2}\right] D_{x}^{\frac{2j+1-\alpha}{2}}u\,\mathrm{d}x.
	\end{split}
	\end{equation}
Since
\begin{equation}\label{eq8.14.1.1.1}
\begin{split}
\left[\mathcal{H}D_{x}^{\alpha+2};\chi_{\epsilon, b}^{2}\right]+\frac{1}{2}P_{n}(\alpha+2)+R_{n}(\alpha+2)&=\frac{1}{2}\mathcal{H}P_{n}(\alpha+2)\mathcal{H}\\
\end{split}
\end{equation}
for some positive integer $n$. Replacing (\ref{eq8.14.1.1.1}) into (\ref{eq8.15.1.1.1})  produces
\begin{equation}\label{eq8.14.2.1}
\begin{split}
A_{2}(t)&
=\frac{1}{2}\int_{\mathbb{R}}D_{x}^{\frac{2j+1-\alpha}{2}}u\left(R_{n}(\alpha+2)D_{x}^{\frac{2j+1-\alpha}{2}}u\right)\,\mathrm{d}x\\
&\quad+
\frac{1}{4}\int_{\mathbb{R}}D_{x}^{\frac{2j+1-\alpha}{2}}u\left(P_{n}(\alpha+2)D_{x}^{\frac{2j+1-\alpha}{2}}u\right)\,\mathrm{d}x\\
&\quad -\frac{1}{4}\int_{\mathbb{R}}D_{x}^{\frac{2j+1-\alpha}{2}}u\left(\mathcal{H}P_{n}(\alpha+2)\mathcal{H}D_{x}^{\frac{2j+1-\alpha}{2}}u\right)\,\mathrm{d}x\\
&=A_{2,1}(t)+A_{2,2}(t)+A_{2,3}(t).
\end{split}
\end{equation}
As above we deal  first with the crucial   term in the decomposition associated to $A_{2},$that is   $A_{2,1}.$ 

Applying  Plancherel's identity  yields
\begin{equation*}
\begin{split}
A_{2,1}(t)&=\frac{1}{2}\int_{\mathbb{R}}uD_{x}^{\frac{2j+1-\alpha}{2}}\left\{R_{n}(\alpha+2)D_{x}^{\frac{2j+1-\alpha}{2}}u\right\}\,\mathrm{d}x.
\end{split}
\end{equation*}
We   fix $n$ such that   (\ref{eq21}) is satisfied. In this case we have to take  $a=\alpha+2$ and $\sigma=\frac{2j+1-\alpha}{2},$ to get $n=j.$  As occurs in the previous cases it  is possible for  $n=j+1.$ 
 
Thus, by construction $R_{j}(\alpha+2)$ is bounded in $L^{2}_{x}$ (see Proposition \ref{propo2}).

Then
\begin{equation*}
\begin{split}
|A_{2,1}(t)|&\lesssim \|u_{0}\|_{L^{2}_{x}}^{2}\left\|\widehat{D_{x}^{2j+3}(\chi_{\epsilon, b}^{2})}\right\|_{L^{1}_{\xi}},\\
\end{split}
\end{equation*} 
and 
\begin{equation*}
\int_{0}^{T}|A_{2,1}(t)|\,\mathrm{d}t \lesssim \|u_{0}\|_{L^{2}_{x}}^{2}\sup_{0\leq t\leq  T}\left\|\widehat{D_{x}^{2j+3}(\chi_{\epsilon, b}^{2})}\right\|_{L^{1}_{\xi}}.
\end{equation*}	
 Replacing  $P_{j}(\alpha+2)$ into $A_{2,2}$  
\begin{equation*}
\begin{split}
A_{2,2}(t)& 
=\left(\frac{\alpha+2}{4}\right)\int_{\mathbb{R}}\left(\mathcal{H}\partial_{x}^{j+1}u\right)^{2}(\chi_{\epsilon, b}^{2})'\,\mathrm{d}x\\
&\quad+ \left(\frac{\alpha+2}{2}\right)\sum_{l=1}^{j}c_{2l+1}(-1)^{l}4^{-l}\int_{\mathbb{R}}\left(D_{x}^{j-l+1}u\right)^{2}(\chi_{\epsilon, b}^{2})^{(2l+1)}\,\mathrm{d}x\\
&=A_{2,2,1}(t)+\sum_{l=1}^{j-1} A_{2,2,l}(t)+ A_{2,2,j}(t).
\end{split}
\end{equation*}
Note that $A_{2,2,1}$ is positive and it gives  the smoothing effect after  integration in time, and $A_{2,2,j}$ is bounded by using the local theory.
To handle the remainder terms we  recall that  by construction 
\begin{equation}\label{g2}
\left|(\chi_{\epsilon,b})^{(j)}(x)\right| \lesssim\,\chi_{\epsilon/3,b+\epsilon}'(x)\lesssim\chi_{\epsilon/9,b+10\epsilon/9}(x) \chi_{\epsilon/9,b+10\epsilon/9}'(x)
\end{equation}
 for $x\in\mathbb{R},$\,$j\in\mathbb{Z}^{+}.$

So that,  for $j> 2$
\begin{equation}\label{g1}
\begin{split}
\int_{0}^{T}|A_{2,2,l}(t)|\,\mathrm{d}t&\lesssim  \int_{0}^{T}\int_{\mathbb{R}}\left(D_{x}^{j-l+1}u\right)^{2}\chi_{\epsilon/3,b+\epsilon}'\,\mathrm{d}x\,\mathrm{d}t\\
&\lesssim \int_{0}^{T}\int_{\mathbb{R}}\left(D_{x}^{j-l+1}u\right)^{2}\chi_{\epsilon/9,b+10\epsilon/9} \chi_{\epsilon/9,b+10\epsilon/9}'\,\mathrm{d}x\,\mathrm{d}t,
\end{split}
\end{equation}
thus  if we apply (\ref{l2})  with $(\epsilon/9,b+4\epsilon/3)$ instead of $(\epsilon,b)$  we obtain  
\begin{equation*}
\int_{0}^{T}\int_{\mathbb{R}}\left(D_{x}^{j-l+1}u\right)^{2}(\chi_{\frac{\epsilon}{9},b+10\epsilon/9}\chi_{\frac{\epsilon}{9},b+10\epsilon/9}')\,\mathrm{d}x\,\mathrm{d}t\leq c^{*}_{l,2}
\end{equation*}
for $l=1,2,\dots j-1.$

Meanwhile, 
\begin{equation}\label{g3}
\begin{split}
A_{2,3}(t)& 
=\left(\frac{\alpha+2}{4}\right)\int_{\mathbb{R}}\left(\partial_{x}^{j+1}u\right)^{2}(\chi_{\epsilon,b}^{2})'\,\mathrm{d}x\\
&\quad+\left(\frac{\alpha+2}{4}\right)\sum_{l=1}^{j}c_{2l+1}(-1)^{l}4^{-l}\int_{\mathbb{R}}\left(\mathcal{H}D_{x}^{j-l+1}u\right)^{2}(\chi_{\epsilon, b}^{2})^{(2l+1)}\,\mathrm{d}x\\
&=A_{2,3,1}(t)+\sum_{l=1}^{j-1}A_{2,3,l}(t)+A_{2,3,j}(t)
\end{split}
\end{equation}
As we can see $A_{2,3,1}\geq 0$ and it represents the smoothing effect. Besides, applying  a similar argument to the employed in (\ref{g2})-(\ref{g3}) is possible to bound the  remainders terms in   (\ref{g3}).
Anyway,
\begin{equation*}
\int_{0}^{T}|A_{2,3,l}(t)|\,\mathrm{d}t\lesssim c^{*}_{l,2}\qquad  1\leq l\leq j-1.
\end{equation*}
\S.3\quad Only  remains to estimate    $A_{3}$ to finish this step.
\begin{equation}\label{eqA12}
\begin{split}
&D_{x}^{\frac{1-\alpha}{2}}\partial_{x}^{j}(u\partial_{x}u)\chi_{\epsilon, b}\\
&=-\frac{1}{2}\left[D_{x}^{\frac{1-\alpha}{2}}\partial_{x}^{j};\chi_{\epsilon, b}\right]\partial_{x}((u\chi_{\epsilon,b})^{2}+(u\widetilde{\phi_{\epsilon,b}})^{2}+(\psi_{\epsilon}u^{2}))\\
&\quad+\left[D_{x}^{\frac{1-\alpha}{2}}\partial_{x}^{j}; u\chi_{\epsilon, b}\right]\partial_{x}((u\chi_{\epsilon, b})+ (u\phi_{\epsilon,b})+(u\psi_{\epsilon}))
+ u\chi_{\epsilon, b}D_{x}^{\frac{1-\alpha}{2}}\partial_{x}^{j}(\partial_{x}u)\\
&=\widetilde{A_{3,1}}(t)+\widetilde{A_{3,2}}(t)+\widetilde{A_{3,3}}(t)+\widetilde{A_{3,4}}(t)+\widetilde{A_{3,5}}(t)+\widetilde{A_{3,6}}(t)+\widetilde{A_{3,7}}(t).
\end{split}
\end{equation}
Replacing (\ref{eqA12}) into $A_{3}$  and apply  H\"{o}lder's inequality 
\begin{equation*}
\begin{split}
A_{3}(t)
&=\sum_{1\leq k\leq 6}\int_{\mathbb{R}}\widetilde{A_{3,k}}(t)\,D_{x}^{\frac{1-\alpha}{2}}\partial_{x}^{j}u\,\chi_{\epsilon,b}\,\mathrm{d}x+\int_{\mathbb{R}}\widetilde{A_{3,7}}(t)\,D_{x}^{\frac{1-\alpha}{2}}\partial_{x}^{j}u\,\chi_{\epsilon,b}\,\mathrm{d}x\\
&\leq \sum_{1\leq k \leq 6}\|\widetilde{A_{3,k}}(t)\|_{L_{x}^{2}}\left\|D_{x}^{j+\frac{1-\alpha}{2}}u(t)\chi_{\epsilon,b}(\cdot+vt)\right\|_{L_{x}^{2}}+\int_{\mathbb{R}}\widetilde{A_{3,7}}(t)D_{x}^{\frac{1-\alpha}{2}}\partial_{x}^{j}u\chi_{\epsilon,b}\,\mathrm{d}x.\\
&=\left\|D_{x}^{j+\frac{1-\alpha}{2}}u(t)\chi_{\epsilon,b}(\cdot+vt)\right\|_{L_{x}^{2}}\sum_{1\leq m\leq 6}A_{3,k}(t)+ A_{3,7}(t).
\end{split}
\end{equation*}
The first factor on the right hand side is the quantity to be  estimated.

We  will start by estimating the easiest term.
\begin{equation*}
\begin{split}
A_{3,7}(t) 
&=-\frac{1}{2}\int_{\mathbb{R}}\partial_{x}u\,\chi_{\epsilon, b}^{2}\,(D_{x}^{\frac{1-\alpha}{2}}\partial_{x}^{j}u)^{2}\,\mathrm{d}x-\int_{\mathbb{R}}u\chi_{\epsilon, b} \,\chi_{\epsilon, b}'(D_{x}^{\frac{1-\alpha}{2}}\partial_{x}^{j}u)^{2}\,\mathrm{d}x\\
&=A_{3,7,1}(t)+A_{3,7,2}(t).
\end{split}
\end{equation*}
We have that 
\begin{equation*}
|A_{3,7,1}(t)|\lesssim\|\partial_{x}u(t)\|_{L^{\infty}_{x}}\int_{\mathbb{R}}\left(D_{x}^{\frac{1-\alpha}{2}}\partial_{x}^{j}u\right)^{2}\chi_{\epsilon, b}^{2}\,\mathrm{d}x,
\end{equation*}
where the last integral  is the quantity that we want to estimate, and the another factor will be controlled after integration in time.

After integrating in time and   Sobolev's embedding 
\begin{equation*}
\begin{split}
\int_{0}^{T}|A_{3,7,2}(t)|\,\mathrm{d}t&\lesssim \int_{0}^{T}\int_{\mathbb{R}}u\,(\chi_{\epsilon, b}^{2})'\left(D_{x}^{\frac{1-\alpha}{2}}\partial_{x}^{j}u\right)\mathrm{d}x\,\mathrm{d}t\\
&\lesssim \left(\sup_{0\leq t\leq T}\|u(t)\|_{H^{s(\alpha)+}_{x}}\right)\int_{0}^{T}\int_{\mathbb{R}}\left(D_{x}^{\frac{1-\alpha}{2}}\partial_{x}^{j}u\right)^{2}\,(\chi_{\epsilon, b}^{2})'\mathrm{d}x\,\mathrm{d}t
\end{split}
\end{equation*}
where the integral expression on the right hand side was already  estimated in \eqref{eq41}.

To handle the contribution coming from  $\widetilde{A_{3,1}}$  and $\widetilde{A_{3,2}}$  we apply a combination of \eqref{eq5}, \eqref{kpdl} and Lemma \ref{lema2}
to obtain 
\begin{equation}\label{d14.1}
\begin{split}
\|\widetilde{A_{3,1}}(t)\|_{L^{2}_{x}}
&\lesssim \left\|D_{x}^{j+\frac{1-\alpha}{2}}(u\chi_{\epsilon,b})\right\|_{L^{2}_{x}}\|u\|_{L^{\infty}_{x}}+\|u_{0}\|_{L_{x}^{2}}\|u\|_{L^{\infty}_{x}}
\end{split}
\end{equation}
and 
\begin{equation*}
\begin{split}
\|\widetilde{A_{3,2}}(t)\|_{L^{2}_{x}}
&\lesssim \left\|D_{x}^{j+\frac{1-\alpha}{2}}(u\widetilde{\phi_{\epsilon,b}})\right\|_{L^{2}_{x}}\|u\|_{L^{\infty}_{x}}+\|u_{0}\|_{L_{x}^{2}}\|u\|_{L^{\infty}_{x}}.
\end{split}
\end{equation*}
The condition on the supports of $\chi_{\epsilon, b}$ and $\psi_{\epsilon}$ 	combined with Lemma \ref{lemma1} 
implies 
\begin{equation*}
\begin{split}
\|\widetilde{A_{3,3}}(t)\|_{L^{2}_{x}}
&\lesssim\|u_{0}\|_{L^{2}_{x}}\|u\|_{L^{\infty}_{x}}.
\end{split}
\end{equation*}
By using (\ref{eq31}) and (\ref{kpdl})
\begin{equation*}
\begin{split}
\|\widetilde{A_{3,4}}(t)\|_{L^{2}_{x}}& 
\lesssim \|\partial_{x}(u\chi_{\epsilon,b})\|_{L^{\infty}_{x}}\|D_{x}^{j+\frac{1-\alpha}{2}}(u\chi_{\epsilon,b})\|_{L^{2}_{x}}
\end{split}
\end{equation*}
and 
\begin{equation*}
\begin{split}
\|\widetilde{A_{3,5}}(t)\|_{L^{2}_{x}}&
\lesssim\|\partial_{x}(u\chi_{\epsilon,b})\|_{L^{\infty}_{x}}\left\|D_{x}^{j+\frac{1-\alpha}{2}}(u\phi_{\epsilon,b})\right\|_{L^{2}_{x}}+\|\partial_{x}(u\phi_{\epsilon,b})\|_{L^{\infty}_{x}}\left\|D_{x}^{j+\frac{1-\alpha}{2}}(u\chi_{\epsilon,b})\right\|_{L^{2}_{x}}.
\end{split}
\end{equation*}
An application of    Lemma \ref{lemma1} leads to  
\begin{equation}\label{d15.1}
\begin{split}
\|\widetilde{A_{3,6}}(t)\|_{L^{2}_{x}}&=\left\|u\chi_{\epsilon,b}\partial_{x}D_{x}^{j+\frac{1-\alpha}{2}}(u\psi_{\epsilon})\right\|_{L^{2}_{x}}\lesssim \|u_{0}\|_{L^{2}_{x}}\|u\|_{L^{\infty}_{x}}.
\end{split}
\end{equation}
To complete the estimate in  (\ref{d14.1})-(\ref{d15.1})   we write
\begin{equation*}
\chi_{\epsilon,b }(x)+\phi_{\epsilon, b}(x)+\psi_{\epsilon}(x)=1\quad \forall x\in \mathbb{R};
\end{equation*}
then
\begin{equation*}
\begin{split}
D_{x}^{j+\frac{1-\alpha}{2}}(u\chi_{\epsilon, b})&=D_{x}^{j+\frac{1-\alpha}{2}}u \chi_{\epsilon, b}+ \left[D_{x}^{j+\frac{1-\alpha}{2}}; \chi_{\epsilon, b}\right](u\chi_{\epsilon, b}+u\phi_{\epsilon, b}+u\psi_{\epsilon})\\
&=I_{1}+I_{2}+I_{3}+I_{4}.
\end{split}
\end{equation*}
Notice that $\|I_{1}\|_{L_{x}^{2}}$ is the quantity to be  estimated. In contrast, 
$I_{4}$ is handled  by  using  Lemma \ref{lemma1}.
In regards to 
 $\|I_{2}\|_{L_{x}^{2}}$ and $\|I_{3}\|_{L_{x}^{2}}$  the  Lemma \ref{dlkp}  combined with  the local theory, and the step 2  corresponding to the case  $j-1$ produce the required bounds.   

By  Theorem \ref{thm11}  and  H\"{o}lder's inequality  
\begin{equation}\label{eq72.1}
\begin{split}
\left\|D_{x}^{j+\frac{1+\alpha}{2}}(u\phi_{\epsilon,b})\right\|_{L^{2}_{x}}&\lesssim \|u\|_{L_{x}^{4}}\left\|D_{x}^{j+\frac{1+\alpha}{2}}\phi_{\epsilon,b}\right\|_{L_{x}^{4}}+\left\|\sum_{\beta\leq j} \frac{1}{\beta!}\partial_{x}^{\beta}\phi_{\epsilon,b}D_{x}^{s,\beta}u\right\|_{L_{x}^{2}}\\
&\lesssim \|u_{0}\|_{L_{x}^{2}}^{1/2}\|u\|_{L_{x}^{\infty}}^{1/2}+ \sum_{\beta\in \mathbb{Q}_{1}(j)}\frac{1}{\beta!}\left\|\partial_{x}^{\beta}\phi_{\epsilon, b}D_{x}^{j-\beta+\frac{\alpha+1}{2}}u\right\|_{L_{x}^{2}} \\
&\quad +\sum_{\beta\in\mathbb{Q}_{2}(j)}\frac{1}{\beta!}\left\|\partial_{x}^{\beta}\phi_{\epsilon, b}\mathcal{H}D_{x}^{j-\beta+\frac{\alpha+1}{2}}u\right\|_{L_{x}^{2}}
\end{split}
\end{equation}
where $\mathbb{Q}_{1}(j), \mathbb{Q}_{2}(j)$ denotes  odd integers and even integers in $\{0,1,\dots,j\}$ respectively.

To estimate  the second term  in \eqref{eq72.1}, note that $\partial_{x}^{\beta}\phi_{\epsilon, b}$  is supported in $[\epsilon/4,b]$ then
\begin{equation*}
\begin{split}
\sum_{\beta\in \mathbb{Q}_{1}(j)}\frac{1}{\beta!}\left\|\partial_{x}^{\beta}\phi_{\epsilon, b}D_{x}^{j-\beta+\frac{\alpha+1}{2}}u\right\|_{L_{x}^{2}}&\lesssim\sum_{\beta\in \mathbb{Q}_{1}(j)}\frac{1}{\beta!}\left\|\mathbb{1}_{[\epsilon/8,b]}D_{x}^{j-\beta+\frac{\alpha+1}{2}}u\right\|_{L_{x}^{2}}\\
&\lesssim \sum_{\beta\in \mathbb{Q}_{1}(j)}\frac{1}{\beta!}\left\|\eta_{\epsilon/24,b+7\epsilon/24}D_{x}^{j-\beta+\frac{\alpha+1}{2}}u\right\|_{L_{x}^{2}}.
\end{split}
\end{equation*} 
Hence, after integrate in time and apply (\ref{l2})  with $(\epsilon,b)=(\epsilon/24, b+7\epsilon/24
)$  we obtain 
\begin{equation*}
\begin{split}
\sum_{\beta\in \mathbb{Q}_{1}(j)}\frac{1}{\beta!}\left\|\eta_{\epsilon/24,b+7\epsilon/24}D_{x}^{j-\beta+\frac{\alpha+1}{2}}u\right\|_{L^{2}_{T}L_{x}^{2}}&\lesssim \sum_{\beta\in \mathbb{Q}_{1}(j)}(c^{*}_{j-\beta,1})^{1/2}<\infty
\end{split}
\end{equation*}
by the induction hypothesis.

Analogously, we can handle the third  term in \eqref{eq72.1}   
\begin{equation*}
\begin{split}
\sum_{\beta\in \mathbb{Q}_{2}(j),\beta\neq j}\frac{1}{\beta!}\left\|\partial_{x}^{\beta}\phi_{\epsilon, b}\mathcal{H}D_{x}^{j-\beta+\frac{\alpha+1}{2}}u\right\|_{L^{2}_{T}L_{x}^{2}}&\lesssim \sum_{\beta\in \mathbb{Q}_{2}(j),\beta\neq j}(c^{*}_{j-\beta,1})^{1/2} + \|u\|_{L^{\infty}_{T}H_{x}^{\frac{3-\alpha}{2}}}\\
&<\infty.
\end{split}
\end{equation*}
Therefore, after integrate in time and apply H\"{o}lder's inequality  we have 
\begin{equation*}
\left\|D_{x}^{j+\frac{1+\alpha}{2}}(u\phi_{\epsilon,b})\right\|_{L_{T}^{2}L^{2}_{x}}<\infty.
\end{equation*}
Next, by interpolation and Young's inequality 
\begin{equation}\label{eq42}
\left\|D_{x}^{j+\frac{1-\alpha}{2}}(u\phi_{\epsilon, b})\right\|_{L_{T}^{2}L_{x}^{2}}\lesssim \left\|D_{x}^{j+\frac{1+\alpha}{2}}(u\phi_{\epsilon, b})\right\|_{L_{T}^{2}L_{x}^{2}}+\|u_{0}\|_{L_{x}^{2}}<\infty.
\end{equation}
 If  we apply (\ref{eq72.1})-(\ref{eq42}) then 
\begin{equation*}
\left\|D_{x}^{j+\frac{1-\alpha}{2}}(u\widetilde{\phi_{\epsilon, b}})\right\|_{L^{2}_{T}L_{x}^{2}} <\infty.
\end{equation*}
Finally, after collecting all information and apply  Gronwall's inequality  we obtain 
\begin{equation*}
\begin{split}
&\sup_{0\leq t\leq T}\left\|D_{x}^{\frac{1-\alpha}{2}}\partial_{x}^{j}u\chi_{\epsilon,b}\right\|_{L^{2}_{x}}^{2}+ \left\|\partial_{x}^{j+1}u\eta_{\epsilon,b}\right\|_{L^{2}_{T}L^{2}_{x}}^{2}+\left\|\mathcal{H}\partial_{x}^{j+1}u\eta_{\epsilon,b}\right\|_{L^{2}_{T}L^{2}_{x}}^{2}\leq c^{*}_{j,2}
\end{split}
\end{equation*}
where $c^{*}_{j,2}=c^{*}_{j,2}\left(\alpha;\epsilon;T; v;\|u_{0}\|_{H_{x}^{\frac{3-\alpha}{2}}}; \|D_{x}^{\frac{1-\alpha}{2}}\partial_{x}^{j}u_{0}\chi_{\epsilon, b}\|_{L_{x}^{2}}\right)$
for any $\epsilon>0,\,b\geq 5\epsilon$ and $v\geq0.$

This finishes the induction process.

To justify  the previous estimates  we shall follow the following argument of regularization.
For arbitrary initial data $u_{0}\in H^{s}(\mathbb{R})\, s>\frac{3-\alpha}{2},$
we  consider the regularized initial data $u_{0}^{\mu}=\rho_{\mu}*u_{0}$ with  $\rho\in C^{\infty}_{0}(\mathbb{R}),\, \supp \rho\subset (-1,1),\, \rho\geq 0,\, \|\rho\|_{L^{1}}=1$ and
\begin{equation*}
\rho_{\mu}(x)=\mu^{-1}\rho(x/\mu),\quad \mbox{for}\quad \mu>0.
\end{equation*}
The solution $u^{\mu}$ of the IVP (\ref{eq7})  corresponding to the smoothed data $u_{0}^{\mu}=\rho_{\mu}*u_{0},$ satisfies 
\begin{equation*}
u^{\mu}\in C([0,T]: H^{\infty}(\mathbb{R})),
\end{equation*}
  we shall remark  that the time of existence is independent of $\mu.$ 
  
Therefore, the smoothness of  $u^{\mu}$ allows us to  conclude that
\begin{equation*}
\begin{split}
&\sup_{0\leq t\leq  T}  \left\|\partial_{x}^{m}u^{\mu}\chi_{\epsilon, b}\right\|_{L^{2}_{x}}^{2}+ \left\|D_{x}^{m+\frac{1+\alpha}{2}}u^{\mu
}\right\|_{L^{2}_{T}L^{2}_{x}}^{2} +\left\|\mathcal{H}D_{x}^{m+\frac{1+\alpha}{2}}u^{\mu}\eta_{\epsilon, b}\right\|_{L^{2}_{T}L^{2}_{x}}^{2}  \leq c^{*}
\end{split}
\end{equation*}
where  ${\displaystyle 
c^{*}=c^{*}\left(\alpha; \epsilon;T;v;\left\|u_{0}^{\mu}\right\|_{H^{\frac{3-\alpha}{2}}_{x}}; \left\|\partial_{x}^{m}u^{\mu}_{0}\chi_{\epsilon, b}\right\|_{L^{2}_{x}}\right).}$
  In fact our next task is to  prove that the  constant  $c^{*}$ is  independent of the parameter $\mu.$

The independence from the parameter $\mu>0$ can be reached first noticing that
\begin{equation*}
\|u^{\mu}_{0}\|_{H^{\frac{3-\alpha}{2}}_{x}}\leq \|u_{0}\|_{H^{\frac{3-\alpha}{2}}_{x}}\left\|\widehat{\rho_{\mu}}\right\|_{L^{\infty}_{\xi}}=\|u_{0}\|_{H^{\frac{3-\alpha}{2}}_{x}}\|\rho_{\mu}\|_{L^{1}_{x}}=\|u_{0}\|_{H^{\frac{3-\alpha}{2}}_{x}}.
\end{equation*}

Next, since  $\chi_{\epsilon,b}(x)=0$ for $x\leq \epsilon,$ then restricting $\mu\in(0,\epsilon)$ it follows by Young's inequality 
\begin{equation*}
\begin{split}
\int_{\epsilon}^{\infty}(\partial_{x}^{m}u_{0}^{\mu})^{2}\,\mathrm{d}x&\leq \|\rho_{\mu}\|_{L^{1}_{\xi}}\|\partial_{x}^{m}u_{0}\|_{L^{2}_{x}((0,\infty))}=\|\partial_{x}^{m}u_{0}\|_{L^{2}_{x}((0,\infty))}.
\end{split}
\end{equation*}
Using the continuous dependence of  the  solution upon the data  we have that 
\begin{equation*}
\sup_{t\in[0,T]}\|u^{\mu}(t)-u(t)\|_{H^{\frac{3-\alpha}{2}}_{x}}\underrel[c]{\mu \to 0}{\longrightarrow} 0.
\end{equation*}
Combining this fact with the independence of  the constant $c^{*}$ from the parameter  $\mu$ , weak compactness  and Fatou's Lemma, the theorem  holds for  all $u_{0}\in H^{s}(\mathbb{R}),$\, $s>\frac{3-\alpha}{2}.$   
\end{proof}
\begin{rem}
	The proof of Theorem \ref{A} remains valid  for the defocussing  dispersive generalized Benjamin-Ono equation 
	\begin{equation*}
	\left\{
	\begin{array}{ll}
	\partial_{t}u-D_{x}^{\alpha+1}\partial_{x}u-u\partial_{x}u=0, & x,t\in\mathbb{R},\,0<\alpha<1, \\
	u(x,0)=u_{0}(x).&  \\
	\end{array} 
	\right.
	\end{equation*}
	In this direction, the propagation of regularity  holds for  $u(-x,-t),$ being  $u(x,t)$ a solution  of (\ref{eq7}). In other words, this means that for initial data satisfying the conditions (\ref{eq106}) and (\ref{clave1}) on the left hand side of the real line, the  Theorem \ref{A} remains valid backward in time.
\end{rem}
A consequence of the Theorem \ref{A} is the following corollary, that describe the asymptotic behavior of the function in \eqref{eqc1}.
\begin{cor}\label{cor4}
	Let $u\in C\left([-T,T]:H^{\frac{3-\alpha}{2}}(\mathbb{R})\right)$  be a solution  of the equation in (\ref{eq7}) described by Theorem \ref{A}. 
	
	Then, for any $t\in (0,T]$ and $\delta>0$
	\begin{equation}\label{corf}
	\int_{-\infty}^{\infty}\frac{1}{\langle x_{-}\rangle^{j+\delta}}\left(\partial_{x}^{j}u\right)^{2}(x,t)\,\mathrm{d}x\leq  \frac{c}{t},
	\end{equation}
	where $x_{-}=\max\{0,-x\},\, c$ is a positive constant and $\langle x\rangle:=\sqrt{1+x^{2}}.$
\end{cor}
For the proof \eqref{corf} we use the following  lemma  provided by  Segata and Smith \cite{SEGATASMITH}.
\begin{lem}\label{lem2}
	Let $f:[0,\infty)\longrightarrow [0,\infty)$ be a  continuous function. If for $a>0$
	\begin{equation*}
	\int_{0}^{a} f(x)\,\mathrm{d}x\leq c a^{p},
	\end{equation*}
	then for every $\delta>0,$ there  
	\begin{equation*}
	\int_{0}^{\infty} \frac{f(x)}{\langle x \rangle^{p+\delta}}\,\mathrm{d}x\leq c(p).
	\end{equation*}
\end{lem}
\begin{proof}
	The proof follows by using a smooth dyadic partition of unit of $\mathbb{R}^{+}.$
	\end{proof}

\begin{rem}\label{ref2}
	Observe that the lemma
	also applies when integrating a non-negative function on the interval $[-(a+\epsilon), -\epsilon],$
	implying decay on the left half-line.
\end{rem}

\begin{proof}[Proof of Corollary \ref{cor4}]
		We shall recall that Theorem \ref{A} with $x_{0}=0$ asserts that any $\epsilon>0$
	\begin{equation*}
	\sup_{t\in[0,T]} \int_{\epsilon-vt}^{\infty}\left(\partial_{x}^{j}u\right)^{2}(x,t)\,\mathrm{d}x\leq c^{*}.
	\end{equation*}
	
	For fixed  $t\in [0,T]$  we split the integral term  as follows
	\begin{equation*}
	\begin{split}
	\int_{\epsilon-vt}^{\infty} \left(\partial_{x}^{j}u\right)^{2}(x,t)\,\mathrm{d}x&=\int_{\epsilon-vt}^{\epsilon}	\left(\partial_{x}^{j}u\right)^{2}(x,t)\,\mathrm{d}x+\int_{\epsilon}^{\infty}\left(\partial_{x}^{j}u\right)^{2}(x,t)\,\mathrm{d}x.
	\end{split}
	\end{equation*}
	The second term in the right hand side is easily bounded by using Theorem \ref{A} with $v=0$. So that, we just need to estimate the first integral in the right hand side.  
	
	Notice  that after making a change of variables, 
	\begin{equation*}
	\int_{\epsilon-vt}^{\epsilon}\left(\partial_{x}^{j}u\right)^{2}(x,t)\,\mathrm{d}x=\int_{-(\epsilon-vt)}^{-\epsilon}\left(\partial_{x}^{j}u\right)^{2}(x+2\epsilon,t)\,\mathrm{d}x\leq c^{*}. 
	\end{equation*}
	So that, by using the lemma \ref{lem2} and the  remark \ref{ref2} we  find 
	\begin{equation*}
	\begin{split}
	\int_{-\infty}^{-\epsilon}\frac{1}{\langle x+2\epsilon\rangle^{j+\delta}}\left(\partial_{x}^{j}u\right)^{2}(x+2\epsilon,t)\,\mathrm{d}x&=\int_{-\infty}^{\epsilon}\frac{1}{\langle x\rangle^{j+\delta}}\left(\partial_{x}^{j}u\right)^{2}(x,t)\,\mathrm{d}x\leq \frac{c^{*}}{t^{j}}.
	\end{split}
	\end{equation*} 
	In summary, we have proved that  for all $j\in \mathbb{Z}^{+},\,j\geq 2$ and any $\delta>0,$
	\begin{equation}\label{eq52}
	\int_{-\infty}^{\epsilon}\frac{1}{\langle x\rangle^{j+\delta}}\left(\partial_{x}^{j}u\right)^{2}(x,t)\,\mathrm{d}x\leq  \frac{c^{*}}{t},
	\end{equation}
	and 
	\begin{equation}\label{eq53}
	\int_{\epsilon}^{\infty}\left(\partial_{x}^{j}u\right)^{2}(x,t)\,\mathrm{d}x\leq c^{*}.
	\end{equation}
	If we apply the lemma \ref{lem2} to \eqref{eq53} we  obtain a more extra decay in the right hand side, this  allow us to   obtain a  uniform expression that combines \eqref{eq52} and \eqref{eq53}, that is, there exist a constant $c$ such that 
	for any $t\in (0,T]$ and $\delta>0$
	\begin{equation*}
	\int_{-\infty}^{\infty}\frac{1}{\langle x_{-}\rangle^{j+\delta}}\left(\partial_{x}^{j}u\right)^{2}(x,t)\,\mathrm{d}x\leq  \frac{c}{t}.
	\end{equation*}
	
\end{proof}
\section{Acknowledgments}
The results of this paper are part of the author's Ph.D dissertation at  IMPA-Brazil. He gratefully  acknowledges  the encouragement and assistance  of his advisor, Prof. F. Linares. He also express appreciation for the careful reading of the manuscript done   by  R. Freire  and O. Ria\~{n}o. The author also thanks to Prof. Gustavo Ponce for the stimulating conversation on this topic. The author is grateful to the referees for their constructive input and suggestions.

\end{document}